%% file: Main.tex
\documentclass{amsart}
\usepackage{enumitem, subcaption, color, mathtools, url, mathrsfs, amsmath,amssymb,amsthm,color,mathrsfs, geometry, nicefrac, dutchcal, array}

\usepackage[initials]{amsrefs}
\usepackage{physics}
\usepackage[T1]{fontenc}
\usepackage{parskip}
\usepackage{tcolorbox}
\usepackage{bbm}
\usepackage[linewidth=1pt]{mdframed}
\usepackage{ulem, censor}

\newgeometry{vmargin={32mm}, hmargin={30mm,30mm}}

\newtheorem{theorem}{Theorem}
\newtheorem{definition}[theorem]{Definition}
\newtheorem{proposition}[theorem]{Proposition}
\newtheorem{lemma}[theorem]{Lemma}
\newtheorem{corollary}[theorem]{Corollary}
\newtheorem{remark}[theorem]{Remark}
\newtheorem{Notation}[theorem]{Notation}
\newtheorem{Convention}[theorem]{Convention}

\newcommand{\nc}{\newcommand}
\nc{\R}{\mathbb{R}}
\nc{\C}{\mathbb{C}}
\nc{\mrm}{\mathrm}
\nc{\mL}{\mrm{L}}
\nc{\mF}{\mrm{F}}
\nc{\mC}{\mrm{C}}
\nc{\mH}{\mrm{H}}
\nc{\mW}{\mrm{W}}
\nc{\mV}{\mrm{V}}
\nc{\mM}{\mrm{M}}
\nc{\mK}{\mrm{K}}
\nc{\mD}{\mrm{D}}
\nc{\mB}{\mrm{B}}
\nc{\mR}{\mrm{R}}
\nc{\mX}{\mrm{X}}
\nc{\mY}{\mrm{Y}}
\nc{\mS}{\mrm{S}}

\nc{\Ec}{\mrm{E_c}}
\nc{\calL}{\mathcal{L}}
\nc{\loc}{\mrm{loc}}
\nc{\comp}{c}
\nc{\supp}{\mrm{supp}}
\nc{\Hardy}{\mathfrak{H}}
\nc{\calH}{\mathcal{H}}
\nc{\ctru}{\mathfrak{u}}
\nc{\ctrv}{\mathfrak{v}}
\nc{\bc}{\boldsymbol{c}}
\nc{\be}{\boldsymbol{e}}
\nc{\br}{\boldsymbol{r}}
\nc{\bs}{\boldsymbol{s}}
\nc{\bt}{\boldsymbol{t}}
\nc{\bw}{\boldsymbol{w}}
\nc{\bx}{\boldsymbol{x}}
\nc{\by}{\boldsymbol{y}}
\nc{\bz}{\boldsymbol{z}}
\nc{\lbr}{\lbrack}
\nc{\rbr}{\rbrack}
\nc{\dsp}{\displaystyle}
\nc{\vphi}{\varphi}

\begin{document}
	%%-----------------------------
	%%      the top matter
	%%-----------------------------
	\title[Analysis of the Coupled Cluster Method]{Analysis of the Single Reference Coupled Cluster Method for Electronic Structure Calculations: The Full-Coupled Cluster Equations}
	
	\author{Muhammad Hassan}\address{(M. Hassan and Y. Wang) Sorbonne Université, CNRS, Université Paris Cité, Laboratoire Jacques-Louis Lions (LJLL), F-75005 Paris, France.}
	\author{Yvon Maday}\address{(Y. Maday) Sorbonne Université, CNRS, Université Paris Cité, Laboratoire Jacques-Louis Lions (LJLL), F-75005 Paris, France and Institut Universitaire de France,Paris, France.}
	\author{Yipeng Wang} %\address{Sorbonne Université, CNRS, Université Paris Cité, Laboratoire Jacques-Louis Lions (LJLL), F-75005 Paris, France.}

	\begin{abstract} 
		The central problem in electronic structure theory is the computation of the eigenvalues of the electronic Hamiltonian-- an unbounded, self-adjoint operator acting on a Hilbert space of antisymmetric functions. Coupled cluster (CC) methods, which are based on a non-linear parameterisation of the sought-after eigenfunction  and result in non-linear systems of equations, are the method of choice for high accuracy quantum chemical simulations but their numerical analysis is underdeveloped. The existing numerical analysis relies on a \emph{local, strong monotonicity} property of the CC function that is valid only in a perturbative regime, i.e., when the sought-after \emph{ground state} CC solution is sufficiently close to zero. In this article, we introduce a new well-posedness analysis for the single reference coupled cluster method based on the invertibility of the CC derivative. Under the minimal assumption that the sought-after eigenfunction is intermediately normalisable and the associated eigenvalue is isolated and non-degenerate, we prove that the continuous (infinite-dimensional) CC equations are \emph{always} locally well-posed. Under the same minimal assumptions and provided that the discretisation is fine enough, we prove that the discrete Full-CC equations are locally well-posed, and we derive residual-based error estimates with \emph{guaranteed positive} constants. Preliminary numerical experiments indicate that the constants that appear in our estimates are a significant improvement over those obtained from the local monotonicity approach. 
	\end{abstract}
	\maketitle
	%%-----------------------------
	%%      your text
	%%-----------------------------

	\section{Introduction}
	
	\input{Introduction.tex}

	\section{Problem Formulation and Setting}\label{sec:2}~

\input{Problem_Formulation.tex}

	\section{Excitation Operators and the Coupled Cluster Ansatz}\label{sec:4}~
	
	\input{Excitation_Operators_New.tex}

	\section{Well-posedness of the Continuous Coupled Cluster Equations}\label{sec:5}~

\input{FullCC_New.tex}
	
	\section{Well-posedness of the Full Coupled Cluster Equations in a Finite Basis}\label{sec:6}~
	
	\input{Truncated_CC_New.tex}

	\section*{Acknowledgements} 
	
	The first author warmly thanks Reinhold Schneider, Antoine Levitt and Eric Canc\`es for useful comments and fruitful discussions and also gratefully acknowledges IPAM where portions of this work were completed (March-June 2022). This project has received funding from the European Research Council (ERC) under the European Union’s Horizon 2020 research and innovation program (grant agreement No. 810367). 
	
	\bibliography{refs.bib}

\end{document}

%% file: Introduction.tex
Quantum computational chemistry is by now widely regarded as one of the central pillars of modern chemistry, as evidenced by the award of two Nobel prizes (Walter Kohn and John Pople (1998) \cite{Nobel_98}; Martin Karplus, Michael Levitt, and Arieh Warschel (2013) \cite{Nobel_13}) in recent years. The field is typically thought to have begun with the pioneering work of Walter Heitler and Fritz London \cite{heitler1927wechselwirkung} in the 1920s but major, concurrent advances were due to Vladimir Fock, Douglas Hartree, Egil Hylleraas and John Slater \cite{hartree1928wave, fock1930naherungsmethode, slater1930note, slater1930atomic, hylleraas1928grundzustand, hylleraas1929neue} among others. These first developments were followed by seminal contributions in the post-war period by the likes of Francis Boys \cite{boys1950electronic},  Ji{\v{r}}{\'\i}  {\v{C}}{\'\i}{\v{z}}ek \cite{vcivzek1966correlation}, George Hall \cite{hall1951molecular}, Clemens Roothan \cite{roothaan1951new, roothaan1960self} and many others (see, for instance, \cite[Chapter 1.2]{MR2008386} for a more comprehensive account). The subsequent explosion in available computing resources which began in the 1970s helped spur a tremendous development in the field (see, e.g., the development of computer software such as POLYATOM \cite{barnett1963mechanized}, IBMOL \cite{clementi1966electronic}, and GAUSSIAN 70 which is still used today \cite{g16}), and quantum-chemical simulations are today routinely performed by thousands of researchers, complementing painstaking laboratory work on the design of new compounds for {\emph{sustainable energy}}, {\emph{green catalysis}}, and {\emph{pharmaceutical drugs}} (see, e.g., \cite{dieterich2017opinion, lovelock2022road, manly2001impact, deglmann2015application, hillisch2015computational} and the references therein). Indeed, according to the 2021 annual report of the European High Performance Computing Joint Undertaking (EuroHPC-JU), nearly a quarter of of the simulations running on the supercomputers Karolina and Vega pertained to chemical and material science simulations, with similar or higher numbers reported by supercomputing centers in Germany \cite{Germany}, Italy \cite{Italy}, and Switzerland \cite{CSCS}.

The goal of quantum chemistry is to obtain a quantitative description of the behaviour of matter at the atomic scale, i.e., when matter is viewed as a collection of nuclei and electrons. In the so-called non-relativistic Born-Oppenheimer setting, the nuclei of the molecule under study are treated as clamped, point-like particles, and the aim is to describe the evolution of the electrons in the effective electrostatic potential generated by the static configuration of positively charged nuclei, this field of study being known as electronic structure theory. The behaviour of the electrons in this situation is governed by the spectrum of the so-called electronic Hamiltonian-- an unbounded, self-adjoint operator acting on a Hilbert space of anti-symmetric functions. It has been known since the seminal work of Grigorii Zhislin and Aleksandr Sigalov \cite{zhislin1960study, zhislin1965spectrum} that for neutral molecules and positively charged ions, the electronic Hamiltonian possesses a lowest eigenvalue, frequently called the ground state energy, and a majority of quantum chemical simulations are concerned with approximating this ground state energy.

The primary difficulty in the numerical computation of the lowest eigenvalue of the electronic Hamiltonian is the extremely high-dimensionality of the underlying Hilbert space. Indeed, for a system containing $N$ electrons, the sought-after ground state eigenfunction of the electronic Hamiltonian depends on $3N$ spatial variables. A naive application of traditional numerical methods such as finite element approximations or spectral schemes, etc., therefore fails spectacularly, and specialised approximation strategies have to be developed.  Broadly speaking, ab initio (first principles-based) deterministic numerical methods for approximating the ground state energy can be divided into three categories, each of which has a vast variety of subcategories and flavours (see, e.g., \cite[Chapter 1]{MR2008386} for a concise but comprehensive overview).

\begin{itemize}
	\item Wave-function methods which focus on approximating directly the ground state eigenfunction of the electronic Hamiltonian. 
	\item Density functional methods which are based on a reformulation of the minimisation problem for the electronic Hamiltonian (which acts on functions of $3N$ spatial variables) in terms of an equivalent minimisation problem over a set of \emph{electronic densities} (which are functions of $3$ spatial variables).
	\item Reduced density matrix approaches which are based on the electronic one-body and two-body reduced density matrices.
\end{itemize}

The coupled cluster (CC) methodology, which belongs to the class of wave-function methods, is based on a non-linear ansatz for the sought-after ground state eigenfunction of the electronic Hamiltonian. In its most common form-- the so-called single reference CC method-- the unknown ground state eigenfunction is expressed as the action of an \emph{exponential} cluster operator, i.e., the operator exponential of a linear combination of linear maps (so-called \emph{excitation operators}), acting on a judiciously chosen reference function (usually a so-called discrete Hartree-Fock determinant). Using this ansatz the eigenvalue problem for the ground state energy of the electronic Hamiltonian can be reformulated as a \emph{non-linear} system of equations for the unknown coefficients appearing in the linear combination of excitation operators entering the operator exponential. Approximations to the ground state energy are then obtained by restricting the class of excitation operators that appear inside the exponential, which leads to a hierarchy of computationally more tractable non-linear, root-finding problems. Usually these truncations are done on the basis of the excitation orders (see Section \ref{sec:3} below) and one thus speaks of CCD (double excitation operators only), CCSD (single and double excitation operators), CCSDT (single, double and triple excitation operators) and so on. We emphasise that these methods are not variational since the sought-after restricted coupled cluster wave-function is no longer the critical point of an energy functional.

Coupled cluster methods were originally introduced in the field of nuclear physics in the late 1950s by Fritz Coester and Hermann Kummel \cite{coester1958bound, coester1960short} but were reformulated for use in quantum chemistry in the following decades by pioneers such as Ji{\v{r}}{\'\i}  {\v{C}}{\'\i}{\v{z}}ek \cite{vcivzek1966correlation}, Josef Paldus \cite{paldus1972correlation}, and Oktay Sinano{\u{g}}lu \cite{sinanouglu1962many}. The original motivation for introducing such methods was the fact that they were size consistent: the approximate coupled cluster ground state energy of a molecular system composed of two independent sub-systems can be shown to be the sum of the individual approximate coupled cluster energies of the two sub-systems. Since size consistency seems to be a vital chemical property not conserved by some other numerical methods, and in practice, the CC methods seem to work extremely well, achieving, in many cases, the chemical accuracy of 1 kcal/mol, they quickly found wide adoption in the quantum chemical community \cite{lee1995achieving}. In particular, the so-called CCSD(T)\footnote{Here, the (T) emphasises the fact that triple excitation orders are not initially included in the CCSD(T) ansatz and are rather treated perturbatively through a post-processing step.}  variant, which can be applied to small and medium-sized molecules at a reasonable computational cost, is widely regarded as the `gold standard' of quantum chemistry~\cite{raghavachari1989fifth}.

Despite the ubiquitous use of this `gold standard' computational method in the quantum chemical community, there is a shockingly limited amount of mathematical literature on the numerical analysis of the coupled cluster methodology. Indeed, a simple search with the keyword ``coupled cluster'' on Google Scholar and MathSciNet, two databases that are representative of the scientific literature as a whole and the subset of mathematical literature thereof, reveals that, as of January 2023, there are more than 100,000 articles pertaining to coupled cluster theory of which less than 40 are listed on MathSciNet. Limiting ourselves to the subset of numerical analysis journals, there are a total seven articles on coupled cluster methods.

The first systematic study of the single reference coupled cluster method from a numerical analysis perspective was undertaken by Reinhold Schneider and Thorsten Rohwedder slightly more than ten years ago. In a series of three remarkable papers \cite{Schneider_1, MR3021693, MR3110488}, they were able to show that the excitation operators that appear inside the coupled cluster exponential are bounded linear maps between  Hilbert spaces of anti-symmetric functions with appropriate regularity and that consequently, the continuous (infinite-dimensional) coupled cluster equations could be given a precise functional-analytic meaning. The coupled cluster approximations (built by restricting the class of excitation operators that enter the exponential operator) could thus be viewed as classical Galerkin discretisations of an infinite-dimensional non-linear problem. Schneider also showed that under some assumptions, the underlying non-linear coupled cluster function is \emph{locally, strongly monotone} which could be exploited to prove the local well-posedness of both the continuous CC equations and its Galerkin discretisations. Schneider and Rohwedder also derived optimal error estimates for the coupled cluster energies using the dual-weighted residual approach of Rolf Rannacher and co-workers \cite{soner2003adaptive}. Since this pioneering work, two further contributions have been published which provide a similar numerical analysis for two other flavours of coupled cluster methods, namely, the \emph{extended} coupled cluster method~\cite{laestadius2018analysis} and the \emph{tailored} coupled cluster method~\cite{faulstich2019analysis} (see also \cite{laestadius2019coupled}). In addition to the aforementioned contributions which tackle the coupled cluster equations from a functional analysis perspective, there has been recent interest in analysing the CC equations using tools from other fields. Thus, the contribution \cite{csirik2021coupled_1} uses concepts from graph theory to present a unified framework for constructing different variants of coupled cluster methods and the follow-up contribution \cite{csirik2021coupled_2} uses topological index theory to study the solutions of the coupled cluster equations in finite-dimensions. Recently, an additional contribution has appeared which investigates the root structure of the CC equations using tools from algebraic geometry \cite{faulstich2022coupled}.

While the articles \cite{Schneider_1, MR3021693, MR3110488, laestadius2018analysis, faulstich2019analysis} listed above lay the groundwork for a rigorous a priori error analysis of the coupled cluster methods, they have one rather unfortunate drawback: in all cases, the well-posedness of the CC equations is established by demonstrating that the underlying CC function is locally strongly monotone, and this demonstration can only be shown to hold if the targeted root~$\bt^*$ of the CC function is sufficiently close to zero. In other words, the local well-posedness analysis and the resulting error estimates only hold in a perturbative regime~$\bt^* \approx 0$. On the other hand, as we discuss in more detail in Remark \ref{rem:monotone} in Section \ref{sec:5}, in many practical situations where the CC method is \emph{known numerically} to yield accurate approximations, the sought-after root $\bt^*$ is \emph{not} in the perturbative regime. For such problems, the existing a priori analysis yields estimates with \emph{negative} constants! The a priori analysis having failed, there is also no hope of developing \emph{a posteriori} error estimates for practical coupled cluster simulations which, in our opinion, would be the ultimate goal of the numerical analysis.

The aim of the current contribution is to develop a new \emph{a priori} error analysis for the single reference coupled cluster equations that is valid under more general conditions. The analysis we present here-- motivated by the existing literature on non-linear numerical analysis (see, for instance, \cite{MR1470227, MR1213837})-- is based on the invertibility of the Fr\'echet derivative of the non-linear coupled cluster function, which is established using a classical inf-sup-type approach. In contrast to the local, strong monotonicity approach pioneered by Schneider, our analysis does not require the sought-after root $\bt^*$ of the coupled cluster function to be close to zero. In this article, we will focus on the continuous (infinite-dimensional) CC equations and a specific version of the discrete CC equations, namely, the Full-CC equations in a finite basis (see Section \ref{sec:6}). The extension of our analysis to more general discretisations (the so-called \emph{truncated} CC equations \cite[Chapter 13]{helgaker2014molecular}) will be addressed in a forthcoming contribution.

The remainder of this article is organised as follows. In Section \ref{sec:2}, we introduce more rigorously than in this introduction, the problem formulation, i.e., the electronic Hamiltonian and the Hilbert spaces on which it acts. In Section~\ref{sec:4}, we introduce excitation operators and the coupled cluster ansatz, and we state the continuous and discrete coupled cluster equations. We begin our analysis proper in Section \ref{sec:5} where we prove, under the minimal assumptions that the sought-after eigenfunction is intermediately normalisable and the associated eigenvalue is non-degenerate, that the continuous (infinite-dimensional) CC equations are \emph{always} locally well-posed. In Section \ref{sec:6}, we analyse a specific discretisation of the CC equations, namely, the Full-CC equations in a finite basis. We prove under the same minimal assumptions of eigenpair non-degeneracy and CC ansatz validity that these equations are locally well-posed provided that the discretisation is fine enough, and we derive residual-based error estimates with \emph{guaranteed positive} constants. Preliminary numerical experiments indicate that the constants that appear in our estimates are a significant improvement over those obtained from the local monotonicity approach. 

%After introducing the Hartree-Fock methodology in Section \ref{sec:6b}, we subsequently consider in Section \ref{sec:6c}, certain classes of so-called \emph{truncated}-CC equations in a discrete Hartree-Fock eigenbasis  targeted at the ground state. Under the same minimal assumptions, provided once again that the discretisation is fine enough and the underlying discrete Hartree-Fock determinant fulfils certain uniform estimates, we prove that these equations are also locally well-posed and we derive \emph{guaranteed positive} constants that can be the basis for future a posteriori error estimation. 

%% file: Problem_Formulation.tex
Computational quantum chemistry is the study of the properties of matter through modelling at the molecular scale, i.e., when matter is viewed as a collection of positively charged nuclei and negatively charged electrons. To formalise the problem setting, we assume that we are given a molecule composed of $M \in \mathbb{N}$ nuclei carrying charges $\{Z_{\alpha}\}_{\alpha =1}^M \subset \mathbb{R}_+$ and located at positions $\{\bold{x}_{\alpha}\}_{\alpha =1}^{M} \subset \mathbb{R}^3$ respectively. We further assume the presence of $N\in \mathbb{N}$ electrons whose spatial coordinates are denoted by $\{\bold{x}_i\}_{i=1}^N \subset \mathbb{R}^3$. Throughout this article, we will assume that the Born-Oppenheimer approximation holds, i.e., we will treat the nuclei as fixed, classical particles and we will focus purely on the quantum mechanical description of the electrons.

In order to describe the behaviour of this system of nuclei and electrons under the Born-Oppenheimer approximation, we require the notion of several functions spaces. The following construction is partially based on \cite{rohwedder2010analysis}.

\subsection{Function Spaces and Norms}\label{sec:2a}~

To begin with, we denote by $\mL^2(\R^3; \C)$ the space of complex-valued square integrable functions of three variables, and we denote by $\mH^1(\R^3; \C)$ the closed subspace of $\mL^2(\R^3; \C)$ consisting of functions that additionally possess square integrable first derivatives. Both spaces are equipped with their usual inner products, and we adopt the convention that both inner products are anti-linear in their first argument. Following the convention in the quantum chemical literature, we will frequently refer to $\mL^2(\R^3; \C)$ and $\mH^1(\R^3; \C)$ as infinite-dimensional \emph{single particle} spaces.

Next, we define the tensor space\footnote{We remind the reader that topological tensor spaces can be defined by taking an orthonormal basis of the underlying single particle spaces (in this case $\mL^2(\R^3; \C)$) and using it to construct the algebraic tensor product vector space. This vector space is equipped with a tensorial inner product inherited from the single particle function spaces (in this case the inner product is given by Equation \eqref{eq:inner_product}). The topological tensor product space is then obtained by taking the completion of the algebraic tensor product vector space with respect to the norm induced by the tensorial inner product.}
\begin{align*}
	\mathcal{L}^2:= \bigotimes_{j=1}^N \mL^2(\R^3; \C),
\end{align*}
which is equipped with an inner product that is constructed by defining first for all elementary tensors $ \mathcal{f}, \mathcal{g} \in \mathcal{L}^2$ with  $\mathcal{f}= \otimes_{j=1}^N \mathcal{f}_j$ and $\mathcal{g}= \otimes_{j=1}^N \mathcal{g}_j$
\begin{equation}\label{eq:inner_product}
	\begin{split}
		\left(\mathcal{f}, \mathcal{g}\right)_{\mathcal{L}^2}:= \prod_{j=1}^N \left(\mathcal{f}_j, \mathcal{g}_j\right)_{\mL^2(\R^3; \C)},
	\end{split}
\end{equation}
and then extending bilinearly for general tensorial elements of $\mathcal{L}^2$. 

It is a consequence of Fubini's theorem that the tensor space $\mathcal{L}^2$ is isometrically isomorphic to the space $\mL^2(\R^{3N}; \C)$ of complex-valued square integrable functions of $3N$ variables with the associated $\mL^2$-inner product that is anti-linear in the first argument. Thanks to this result, we can define the tensor space $\mathcal{H}^1 \subset \mathcal{L}^2 $ as the closure of $\mathscr{C}_0^{\infty}(\mathbb{R}^{3N}; \C)$ in $\mL^2(\R^{3N}; \C)$ with respect to the usual gradient-gradient inner product on $\mathbb{R}^{3N}$.

In quantum mechanics, a fundamental distinction is made between so-called \emph{bosonic} and \emph{fermionic} particles, the latter obeying the so-called Pauli-exclusion principle and thus being described in terms of antisymmetric functions. We are therefore obligated to also define tensor spaces of antisymmetric functions. To this end, we first introduce the so-called \emph{antisymmetric projection operator} $\mathbb{P}^{\rm as} \colon \mathcal{L}^2\rightarrow\mathcal{L}^2$ that is defined through the action
\begin{align*}
	\forall \mathcal{f} \in \mathcal{L}^2\colon \quad (\mathbb{P}^{\rm as} \mathcal{f})(\bold{x}_1, \ldots, \bold{x}_N) :=\frac{1}{\sqrt{N!}} \sum_{\pi \in {\mS}(N)} (-1)^{\rm sgn(\pi)} \mathcal{f}(\bold{x}_{\pi(1)}, \ldots,\bold{x}_{\pi(N)}  ),
\end{align*}
where ${\mS}(N)$ denotes the permutation group of order $N$, and ${\rm sgn(\pi)} $ denotes the signature of the permutation~$\pi~\in~{\mS}(N)$. 

It is easy to establish that $\mathbb{P}^{\rm as}$ is an $\mathcal{L}^2$-orthogonal projection with a closed range. We therefore define the antisymmetric tensor spaces $\widehat{\mathcal{L}}^2 \subset \mathcal{L}^2$ and $	\widehat{\mathcal{H}}^1 \subset \mathcal{H}^1$ as
\begin{align*}
	\widehat{\mathcal{L}}^2:= \bigwedge_{j=1}^N \mL^2(\R^3; \C) := \text{\rm ran}\; \mathbb{P}^{\rm as} \quad \text{and} \quad \widehat{\mathcal{H}}^1 := \widehat{\mathcal{L}}^2 \cap \mathcal{H}^1,
\end{align*}
equipped with the $(\cdot, \cdot)_{\mathcal{L}^2}$ and $(\cdot, \cdot)_{\mathcal{H}^1}$ inner products respectively. We remark that normalised elements of $\widehat{\mathcal{L}}^2$ are known as \emph{wave-functions}, and these are antisymmetric in the sense that for any $\mathcal{f} \in	\widehat{\mathcal{L}}^2 $ we have that
\begin{align*}
	\mathcal{f}(\bold{x}_1, \ldots, \bold{x}_i, \ldots, \bold{x}_j, \ldots \bold{x}_N)= - \mathcal{f}(\bold{x}_1, \ldots, \bold{x}_j, \ldots, \bold{x}_i, \ldots \bold{x}_N) \qquad \forall~ i, j \in \{1, \ldots, N\} \text{ with } i\neq j.
\end{align*}

In the sequel, we will also occasionally make use of the dual space of $\widehat{\mathcal{H}}^1$. We therefore denote $\widehat{\mathcal{H}}^{-1}:= \big(\widehat{\mathcal{H}}^1\big)^*$, we equip $\widehat{\mathcal{H}}^{-1} $ with the canonical dual norm, and we write $\langle \cdot, \cdot \rangle_{\widehat{\mathcal{H}}^1, \widehat{\mathcal{H}}^{-1}} $ for the associated duality pairing. Finally, let us remark that the higher regularity, anti-symmetric tensor space $\widehat{\mathcal{H}}^2$ is defined analogously to $\widehat{\mathcal{H}}^1$.

Next, let us comment on the construction of basis sets for the tensor spaces $\mathcal{H}^1$ and $\widehat{\mathcal{H}}^1$. Given an $\mL^2$-orthonormal, complete basis $\mathcal{B}:=\{\phi_k\}_{k \in \mathbb{N}} \subset \mH^1(\mathbb{R}^3; \C)$, we can construct a complete basis $\mathcal{B}_{\otimes}$ for $\mathcal{H}^1$ by setting
\begin{align*}
	\mathcal{B}_{\otimes}= \Big\{\phi_{k_1} \otimes \phi_{k_2} \otimes \ldots \otimes \phi_{k_N} \colon ~k_1, k_2, \ldots, k_N \in \mathbb{N}\Big\},
\end{align*}
and it follows immediately that $\mathcal{B}_{\otimes}$ is $\mathcal{L}^2$-orthonormal. 

In order to construct a basis for the antisymmetric tensor space $\widehat{\mathcal{H}}^1$, we must first define a suitable subset of $\mathcal{B}_{\otimes}$. To this end, we introduce an index set $\mathcal{J}_{\infty}^N \subset \mathbb{N}^N$ given by
\begin{align*}
	\mathcal{J}_{\infty}^N := \Big\{\boldsymbol{\ell}= (\ell_1, \ell_2, \ldots, \ell_N)\in \mathbb{N}^N \colon \ell_1 < \ell_2 < \ldots < \ell_N\Big\}.
\end{align*}

We can thus define the subset $\mathcal{B}_{\otimes}^{\rm ord} $ of the basis $\mathcal{B}_{\otimes}$ given by
\begin{align*}
	\mathcal{B}_{\otimes}^{\rm ord}:= \Big\{\Phi_{\bold{k}}:= \phi_{k_1} \otimes \phi_{k_2} \otimes \ldots \otimes \phi_{k_N} \colon \bold{k}= (k_1, \ldots, k_N) \in \mathcal{J}_{\infty}^N \Big\}.
\end{align*}

A complete basis for the antisymmetric tensor space $\widehat{\mathcal{H}}^1$ is then given by
\begin{align*}
	\mathcal{B}_{\wedge}:=& \{ \mathbb{P}^{\rm as} \Phi \colon \Phi \in \mathcal{B}_{\otimes}^{\rm ord}\}\\
	=& \left\{ \Phi_{\bold{k}}(\bold{x}_1, \bold{x}_2, \ldots, \bold{x}_N)=\frac{1}{\sqrt{N!}} \sum_{\pi \in {\mS}(N)} (-1)^{\rm sgn(\pi)} \otimes_{i=1}^N \phi_{k_i}\big(\bold{x}_{\pi(i)}\big) \colon \hspace{1mm} \bold{k}=(k_1, k_2, \ldots, k_N) \in	\mathcal{J}_{\infty}^N\right\}.
\end{align*}

Elements of the basis set $\mathcal{B}_{\wedge}$ are called Slater determinants. For simplicity, given $\bold{k} \in \mathcal{J}_{\infty}^N$ and $\Phi_{\bold{k}} \in \mathcal{B}_{\wedge}$ of the~form
\begin{align*}
	\Phi_{\bold{k}}(\bold{x}_1, \bold{x}_2, \ldots, \bold{x}_N)=\frac{1}{\sqrt{N!}} \sum_{\pi \in {\mS}(N)} (-1)^{\rm sgn(\pi)} \otimes_{i=1}^N \phi_{k_i}\big(\bold{x}_{\pi(i)}\big),
\end{align*}
we will write $\Phi_{\bold{k}}$ in the succinct form 
\begin{align*}
	\Phi_{\bold{k}}(\bold{x}_1, \bold{x}_2, \ldots, \bold{x}_N)= \frac{1}{\sqrt{N!}}\text{\rm det} \big(\phi_{k_i}(\bold{x}_j)\big)_{i, j=1}^N.
\end{align*}

\subsection{Governing Operators and Problem Statement}~

Throughout this article, we assume that the electronic properties of the molecule that we study can be described by the action of a spinless\footnote{The presence or absence of spin has no bearing on the analysis that we present in this article. We choose to ignore spin simply to avoid unnecessary notational complexity and this choice is fully classical for such analyses.}, many-body electronic Hamiltonian given by
\begin{equation}\label{eq:Hamiltonian}
	H:= -\frac{1}{2} \sum_{j=1}^N \Delta_{\bold{x}_j} + \sum_{j =1}^{N} \sum_{\alpha =1}^{M} \frac{-Z_{\alpha}}{\vert \bold{x}_{\alpha}- \bold{x}_j\vert }  + \sum_{j =1}^{N} \sum_{i =1}^{j-1}  \frac{1}{\vert \bold{x}_i - \bold{x}_j\vert}\qquad \text{acting on } \widehat{\mathcal{L}}^2 \quad \text{with domain } \widehat{\mathcal{H}}^2.
\end{equation}

The electronic properties of the molecule that we study are functions of the spectrum of the electronic Hamiltonian~$H$, and we are therefore interested in its analysis and computation. Of particular importance (and the primary focus of the current contribution) is the lowest eigenpair(s) $\left(\mathcal{E}_{\rm GS}^*, \Psi_{\rm GS}^* \right) \in \mathbb{R} \times \widehat{\mathcal{L}}^2$-frequently called the ground state energy and ground state wave function(s) respectively-, which are defined as 
\begin{subequations}
	\begin{align}
		\label{eq:Ground_State}
		\mathcal{E}_{\rm GS}^*:= \min_{0\neq \Psi \in \widehat{\mathcal{H}}^1} \frac{\big\langle\Psi,H \Psi\big\rangle_{\widehat{\mathcal{H}}^1 \times \widehat{\mathcal{H}}^{-1}}}{\Vert \Psi\Vert^2_{\mathcal{L}^2}} \quad \text{and} \quad \Psi_{\rm GS}^*:=& \underset{0\neq \Psi \in \widehat{\mathcal{H}}^1}{\text{arg~min}} \frac{\big\langle\Psi, H \Psi\big\rangle_{\widehat{\mathcal{H}}^1 \times \widehat{\mathcal{H}}^{-1}}}{\Vert \Psi\Vert^2_{\mathcal{L}^2}} \quad \text{ subject to } \Vert \Psi^*_{\rm GS}\Vert_{\mathcal{L}^2}^2=1,\\ 
		\intertext{and which obviously satisfy}
		H\Psi_{\rm GS}^* =& \mathcal{E}^*_{\rm GS} \Psi^*_{\rm GS}. 
	\end{align}
\end{subequations}

It is a classical result (see, e.g., the review article \cite{MR1768629}) that the operator $H$ is self-adjoint on $\widehat{\mathcal{L}}^2$ with form domain~$\widehat{\mathcal{H}}^1$, and under the additional assumption that $Z:= \sum_{\alpha =1}^M Z_{\alpha} \geq N$, it holds that
\begin{enumerate}
	\item The operator $H$ has an essential spectrum $\sigma_{\rm ess}$ of the form $\sigma_{\rm ess}:= [\Sigma, \infty)$ where $-\infty<\Sigma \leq 0$;
	\item The operator $H$ has a bounded-below discrete spectrum that consists of a countably infinite number of eigenvalues, each with finite multiplicity, accumulating at $\Sigma$.
\end{enumerate}

The existence of the ground state energy $\mathcal{E}^*_{\rm GS}$ is therefore guaranteed as soon as $\sum_{\alpha =1}^M Z_{\alpha} \geq N$, and in this article, we assume that this is always the case.

From a functional analysis point of view, the electronic Hamiltonian $H$ possesses certain desirable properties, namely continuity and ellipticity on appropriate Sobolev spaces. More precisely (see, for instance, \cite[Chapter 4]{yserentant2010electronic}),
\begin{itemize}
	\item The electronic Hamiltonian defined through Equation \eqref{eq:Hamiltonian} is bounded as a mapping from $\widehat{\mathcal{H}}^1$ to $\widehat{\mathcal{H}}^{-1}$:
	\begin{align}\label{eq:contin}
		\forall \Phi, \Psi \in \widehat{\mathcal{H}}^1 \colon \qquad \left \vert \left \langle \Phi, H\Psi\right \rangle_{\widehat{\mathcal{H}}^1 \times \widehat{\mathcal{H}}^{-1}} \right \vert \leq \Big(\frac{1}{2}+ 3\sqrt{N}Z\Big)\Vert \Phi\Vert_{\widehat{\mathcal{H}}^1}\Vert \Psi\Vert_{\widehat{\mathcal{H}}^1};
	\end{align}
	
	\item The electronic Hamiltonian defined through Equation \eqref{eq:Hamiltonian} satisfies the following ellipticity condition on the Gelfand triple $\widehat{\mathcal{H}}^1 \hookrightarrow \widehat{\mathcal{L}}^2 \hookrightarrow \widehat{\mathcal{H}}^{-1}$:
	\begin{align}\label{eq:ellip}
		\forall \Phi \in \widehat{\mathcal{H}}^1 \colon \qquad \left \langle \Phi, H\Phi\right \rangle_{\widehat{\mathcal{H}}^1 \times \widehat{\mathcal{H}}^{-1}}  \geq \frac{1}{4}\Vert \Phi\Vert_{\widehat{\mathcal{H}}^1}^2 - \Big(9NZ^2 -\frac{1}{4}\Big)\Vert \Phi\Vert^2_{\widehat{\mathcal{L}}^2}.
	\end{align}
\end{itemize}

An important consequence of the above ellipticity estimate is that the electronic Hamiltonian, modified by any suitable shift, defines an invertible operator on a subspace of $\widehat{\mathcal{H}}^1$. This fact will be of great importance in our analysis and will be the subject of further discussion in Section \ref{sec:5} (see in particular Remark \ref{rem:interp_const}).

\subsection{Computing the Ground State Energy in a Finite-Dimensional Subspace}\label{sec:3}~

From a practical point of view, the ground state energy of the electronic Hamiltonian defined through Equation \eqref{eq:Hamiltonian} can only be approximated in a finite-dimensional subspace. The most common such approach is known in the quantum chemical literature as \emph{Full Configuration Interaction}. In this subsection, we introduce the terminology and briefly discuss the methodology of the full configuration interaction procedure since the underlying notions will be useful when we discuss the discrete coupled cluster equations in Section \ref{sec:6}

At its core, the full configuration interaction method (Full-CI) is based on a straightforward Galerkin approximation of the minimisation problem \eqref{eq:Ground_State}. We will therefore begin by defining an approximation space. To do so, we fix some $K \in \mathbb{N}$ with $K > N$ and assume that we are given a set $\{\phi_j\}_{j=1}^K \subset \mH^1(\R^3; \C)$ of $\mL^2(\R^3;\C)$-orthonormal functions.  We also introduce an index set $\mathcal{J}_{K}^N \subset \{1, \ldots, K\}^N$ given by
\begin{align*}
	\mathcal{J}_{K}^N := \Big\{\boldsymbol{\ell} = (\ell_1, \ell_2, \ldots, \ell_N)\in \{1, \ldots, K\}^N \colon \ell_1 < \ell_2 < \ldots < \ell_N\Big\}.
\end{align*}

\begin{definition}[Finite Dimensional Single-Particle Basis]\label{def:single_basis}~
	
	We define the $K$-dimensional single particle basis $ \mathcal{B}_K \subset \mH^1(\R^3; \C)$ as $\mathcal{B}_K:= \{\phi_j\}_{j=1}^K $. Additionally, we define the subspace spanned by this basis set as $\mX_K := \text{\rm span}\; \mathcal{B}_K$ and we refer to $\mX_K$ as the single particle approximation space.
\end{definition}

\begin{definition}[Finite Dimensional $N$-Particle Basis]\label{def:n_basis}~
	
	We define the $\mathcal{L}^2$-orthonormal, ${K}\choose{N}$-dimensional $N$-particle basis $ \mathcal{B}_K^{N} \subset \widehat{\mathcal{H}}^1$ as
	\begin{align*}
		\mathcal{B}^{N}_K := \left\{ \Phi_{\bold{k}}(\bold{x}_1, \bold{x}_2, \ldots, \bold{x}_N)=\frac{1}{\sqrt{N!}}\text{\rm det} \big(\phi_{k_i}(\bold{x}_j)\big)_{i, j=1}^N \colon \hspace{1mm} \bold{k}=(k_1, k_2, \ldots, k_N) \in \mathcal{J}_{K}^N \right\}.
	\end{align*} 
	Additionally, we define the subspace spanned by this basis set as $\mathcal{V}_K := \text{\rm span}\; \mathcal{B}^N_K$ and we refer to $\mathcal{V}_K$ as the $N$-particle approximation space.
\end{definition}

\textbf{Full Configuration Interaction Approximation of Minimisation Problem \eqref{eq:Ground_State}}~

Let the $N$-particle approximation space $\mathcal{V}_K$ be defined through Definition \ref{def:n_basis}. We seek the pair(s) $(\mathcal{E}^*_{\rm FCI}, \Psi^*_{\rm FCI}) \in \big(\mathbb{R}, \mathcal{V}_K\big)$ that satisfies
\begin{equation}\label{eq:FCI}
	\mathcal{E}_{\rm FCI}^*:= \min_{0\neq \Psi\in \mathcal{V}_K} \frac{\left \langle \Psi, H\Psi\right \rangle_{\widehat{\mathcal{H}}^1 \times \widehat{\mathcal{H}}^{-1}}}{\Vert \Psi\Vert^2_{\mathcal{L}^2}} \quad \text{and} \quad \Psi_{\rm FCI}^*:= \underset{0\neq \Psi \in \mathcal{V}_K}{\text{arg~min}} \frac{\left \langle \Psi, H\Psi\right \rangle_{\widehat{\mathcal{H}}^1 \times \widehat{\mathcal{H}}^{-1}}}{\Vert \Psi\Vert^2_{\mathcal{L}^2}}  ~ \text{ subject to }~ \Vert \Psi_{\rm FCI}^*\Vert_{\mathcal{L}^2}^2=1.
\end{equation}

Several remarks are now in order.

First, it follows from the variational principle that the minimum in Equation \eqref{eq:FCI} satisfies $\mathcal{E}_{\rm FCI}^* \geq \mathcal{E}^*_{\rm GS}$. %Moreover, for $K$ sufficiently large, the FCI ground state wave-function $\Psi^*_{\rm FCI}$ will be non-degenerate thanks to our earlier assumptions. {\color{red} In this sequel, we assume that this is indeed the case.}

Second, in practice the Full-CI minimisation problem \eqref{eq:FCI} is very often solved by writing first the associated Euler-Lagrange equations, i.e., the first order optimality conditions. This yields a linear eigenvalue problem on the finite-dimensional space $\mathcal{V}_K$ which can, in principle, be solved through the use of some iterative eigenvalue solver.

Third, despite the fact that the Full-CI methodology \eqref{eq:FCI} seems very amenable to numerical analysis- by virtue of being a Galerkin approximation to the exact minimisation problem \eqref{eq:Ground_State}- it has a fundamental computational draw-back: the dimension of the $N$-particle approximation space $\mathcal{V}_K$ grows combinatorially in ${N}$ which renders this approach computationally intractable for $N$ even moderately large. As a consequence, we are very often forced to introduce further approximations to the Full-CI methodology. %(see, for instance, the Hartree-Fock methodology detailed in Section~\ref{sec:6a}).

We end this section by defining the Full-CI Hamiltonian which will be referenced in Section~\ref{sec:6} below.

\begin{definition}[Full-CI Hamiltonian]\label{def:FCI_Hamiltonian}~
	
	Let the $N$-particle approximation space $\mathcal{V}_K$ be defined through Definition \ref{def:n_basis} and let the electronic Hamiltonian $H\colon \widehat{\mathcal{H}}^1 \rightarrow \widehat{\mathcal{H}}^{-1}$ be defined through Equation \eqref{eq:Hamiltonian}. We define the Full-CI Hamiltonian $H_{K} \colon \mathcal{V}_K \rightarrow \mathcal{V}_K^*$ as the mapping with the property that for all $\Psi_K, \Phi_K \in \mathcal{V}_K$ it holds that
	\begin{equation}\label{eq:FCI_Hamiltonian}
		\langle \Psi_K, H_K \Phi_K\rangle_{\mathcal{V}_K \times \mathcal{V}_K^*}:= \langle \Psi_K, H \Phi_K \rangle_{\widehat{\mathcal{H}}^1 \times \widehat{\mathcal{H}}^{-1}}.
	\end{equation}
	
\end{definition}

%% file: Excitation_Operators_New.tex
Throughout this section, we assume the setting of Section \ref{sec:2}. Our goal now is to introduce the notions of excitation indices, excitation operators and the coupled cluster non-linear parameterisation ansatz.

Let us begin by recalling that we have introduced \emph{complete} single-particle and $N$-particle basis sets  $\mathcal{B}= \{\phi_j\}_{j \in \mathbb{N}} \subset \mH^1(\R^3;\C)$ and $\mathcal{B}_{\wedge} \subset \widehat{\mathcal{H}}^1$ respectively in Section \ref{sec:2a}. Next, we define a collection of index sets.

\vspace{4mm}

\begin{definition}[Excitation Index Sets]\label{def:Excitation_Index}~
	
	For each $j \in \{1, \ldots, N\}$ we define the index set $\mathcal{I}_j$ as
	\begin{align*}
		\mathcal{I}_j := \left\{ {{i_1, \ldots, i_j}\choose{\ell_1, \ldots, \ell_j}} \colon i_1 < \ldots< i_j \in \{1, \ldots, N\} \text{ and } \ell_1< \ldots < \ell_j  \in \{N+1, N+2, \ldots\} \right\},
	\end{align*}
	and we say that $\mathcal{I}_j$ is the excitation index set of order $j$. Additionally, we define
	\begin{align*}
		\mathcal{I}:= \bigcup_{j=1}^N \mathcal{I}_j,
	\end{align*}
	and we say that $\mathcal{I}$ is the global excitation index set.
\end{definition}

The excitation index sets $\{\mathcal{I}_j\}_{j=1}^N$ will be used to construct the so-called excitation and de-excitation operators which play a central role in post-Hartree Fock wave-function methods for further approximating the minimisation problem~\eqref{eq:FCI}.

\begin{definition}[Excitation Operators]\label{def:Excitation_Operator}~
	
	Let $j \in \mathbb{N}$ and let $\mu \in \mathcal{I}_j$ be of the form
	\begin{align*}
		\mu={{i_1, \ldots, i_j}\choose{\ell_1, \ldots, \ell_j}} \colon i_1 < \ldots< i_j \in \{1, \ldots, N\} \text{ and } \ell_1< \ldots < \ell_j  \in \{N+1, N+2, \ldots\}.
	\end{align*}
	
	We define the excitation operator $\mathcal{X}_{\mu} \colon \widehat{\mathcal{H}}^1 \rightarrow \widehat{\mathcal{H}}^1$ through its action on the $N$-particle basis set $\mathcal{B}_{\wedge}$: For $\Psi_{\nu}(\bold{x}_1, \ldots, \bold{x}_N) =  \frac{1}{\sqrt{N!}} \text{\rm det}\; \big(\phi_{\nu_j}(\bold{x}_i)\big)_{ i, j=1}^N$, we set 
	\begin{align*}
		\mathcal{X}_{\mu} \Psi_{\nu} = \begin{cases}
			0  &\quad \text{ if } \{i_1, \ldots, i_j\} \not \subset \{\nu_1, \ldots, \nu_N\},\\
			0   & \quad \text{ if } \exists  \ell_{m} \in  \{\ell_1, \ldots, \ell_j\}  \text{ such that } \ell_m \in \{\nu_1, \ldots, \nu_N\},\\
			\Psi_{\nu, \ell} \in \mathcal{B}_{\wedge} & \quad \text{ otherwise},
		\end{cases}
	\end{align*}
	where the determinant $\Psi_{\nu, \ell} $ is constructed from $\Psi_{\nu}$ by replacing all functions $\phi_{i_1},\ldots \phi_{i_j} $ used to construct $\Psi_{\nu}$ with functions $\phi_{\ell_1}, \ldots, \phi_{\ell_j}$ respectively.
\end{definition}

\begin{definition}[De-excitation Operators]\label{def:De-excitation_Operator}~
	
	Let $j \in \mathbb{N}$ and let $\mu \in \mathcal{I}_j$ be of the form
	\begin{align*}
		\mu={{i_1, \ldots, i_j}\choose{\ell_1, \ldots, \ell_j}} \colon i_1 < \ldots< i_j \in \{1, \ldots, N\} \text{ and } \ell_1< \ldots < \ell_j  \in \{N+1, N+2, \ldots\}.
	\end{align*}
	
	We define the de-excitation operator $\mathcal{X}^{\dagger}_{\mu} \colon \widehat{\mathcal{H}}^1 \rightarrow \widehat{\mathcal{H}}^1$ through its action on the $N$-particle basis set $\mathcal{B}_{\wedge}$: For $\Psi_{\nu}(\bold{x}_1, \ldots, \bold{x}_N) =  \frac{1}{\sqrt{N!}} \text{\rm det}\; \big(\phi_{\nu_j}(\bold{x}_i)\big)_{ i, j=1}^N$, we set 
	\begin{align*}
		\mathcal{X}_{\mu}^{\dagger} \Psi_{\nu} = \begin{cases}
			0  &\quad \text{ if } \{\ell_1, \ldots, \ell_j\} \not \subset \{\nu_1, \ldots, \nu_N\},\\
			0   & \quad \text{ if } \exists~i_{m} \in  \{i_1, \ldots, i_j\}  \text{ such that } i_m \in \{\nu_1, \ldots, \nu_N\},\\
			\Psi_{\nu}^{\ell} \in \mathcal{B}_{\wedge} & \quad \text{ otherwise},
		\end{cases}
	\end{align*}
	where the determinant $\Psi_{\nu}^{\ell}$ is constructed from $\Psi_{\nu}$ by replacing all functions $\phi_{\ell_1},\ldots \phi_{\ell_j} $ used to construct $\Psi_{\nu}$ with functions $\phi_{i_1}, \ldots, \phi_{i_j}$ respectively.
\end{definition}

It is natural to ask how de-excitation operators are related to excitation operators. The following remark summarises this relationship.

\begin{remark}[Relationship between Excitation and De-excitation Operators]\label{rem:L2_adjoint}~
	
	Consider the setting of Definitions \ref{def:Excitation_Operator} and \ref{def:De-excitation_Operator}. In some sense, each de-excitation operator reverses the action the corresponding excitation operator. More precisely, it can be shown that for any $\mu \in \mathcal{I}$, the de-excitation operator $\mathcal{X}_{\mu} \colon \widehat{\mathcal{H}}^1 \rightarrow \widehat{\mathcal{H}}^1$ is the $\widehat{\mathcal{L}}^2$-adjoint of the excitation operator $\mathcal{X}_\mu\colon \widehat{\mathcal{H}}^1 \rightarrow \widehat{\mathcal{H}}^1$, i.e.,
	\begin{align*}
		\forall \Phi, \widetilde{\Phi} \in \widehat{\mathcal{H}}^1, ~\forall \mu \in \mathcal{I} \colon \qquad \left( \widetilde{\Phi}, \mathcal{X}_{\mu}  \Phi\ \right)_{\widehat{\mathcal{L}}^2} = \left( \mathcal{X}_{\mu}^{\dagger}\widetilde{\Phi},  \Phi \right)_{\widehat{\mathcal{L}}^2}.
	\end{align*}
\end{remark}

% \begin{remark}[Excitation and De-excitation Operators on $\widehat{\mathcal{L}}^2$]\label{rem:L2_excitation}~
	
	% Consider the setting of Definitions \ref{def:Excitation_Operator} and \ref{def:De-excitation_Operator}. It is easy to see that if the $N$-particle basis set $\mathcal{B}_{\wedge}$ is a complete, $\widehat{\mathcal{L}}^2$-orthonormal basis of $\widehat{\mathcal{L}}^2$ rather than $\widehat{\mathcal{H}}^1$, then the excitation and de-excitation operators define bounded linear maps on $\widehat{\mathcal{L}}^2$.
	% \end{remark}

Several properties of the excitation operators can now be deduced. We begin with a remark.

\begin{remark}[Interpretation of $N$-particle Basis in Terms of Excited Determinants]\label{rem:excited_determinants}~
	
	It is a simple exercise to show that the entire $N$-particle basis set $\mathcal{B}_{\wedge}$ can be generated through the action of the excitation operators on a so-called reference determinant. More precisely, we define $\Psi_0(\bold{x}_1, \ldots, \bold{x}_N):=\frac{1}{\sqrt{N!}} \text{\rm det}\; \big(\phi_{j}(\bold{x}_i)\big)_{ i, j=1}^N$, and it then follows that
	\begin{equation}\label{eq:basis_excited}
		\begin{split}
			\mathcal{B}_{\wedge}&= \{\Psi_0 \} \cup \left\{\mathcal{X}_\mu \Psi_0 \colon \mu \in \mathcal{I}_1\right\} \cup \left\{\mathcal{X}_\mu \Psi_0 \colon \mu \in \mathcal{I}_2\right\} \cup \ldots \cup \left\{\mathcal{X}_\mu \Psi_0 \colon \mu \in \mathcal{I}_N\right\}\\
			& =\{\Psi_0 \} \cup \left\{\mathcal{X}_\mu \Psi_0 \colon \mu \in \mathcal{I}\right\}.
		\end{split}
	\end{equation}
	
	This observation motivates the following convention and definition.
\end{remark}

\begin{Convention}[Reference Determinant]~
	
	Consider the setting of Remark \ref{rem:excited_determinants}. In the sequel, we will refer to the function $\mathcal{B}_{\wedge} \ni \Psi_0(\bold{x}_1, \ldots, \bold{x}_N)=\frac{1}{\sqrt{N!}} \text{\rm det}\; \big(\phi_{j}(\bold{x}_i)\big)_{ i, j=1}^N$, i.e., the determinant constructed from the first $N$ single particle basis functions $\{\phi_i\}_{i=1}^N$ as the reference determinant. Moreover, for any $\mu \in \mathcal{I}$, we will frequently denote $\Psi_{\mu}:= \mathcal{X}_{\mu} \Psi_0$. Finally, we will often refer to each set $\left\{\mathcal{X}_\mu \Psi_0 \colon \mu \in \mathcal{I}_j\right\}$ as the set of $j$-excited determinants,
\end{Convention}

\begin{definition}[Orthogonal Complement of the Reference Determinant]\label{def:V_K} ~
	
	Let the excitation index set $\mathcal{I}$ be defined through Definition \ref{def:Excitation_Index}, let the excitation operators $\{\mathcal{X}_{\mu}\}_{\mu \in \mathcal{I}}$ be defined through Definition \ref{def:Excitation_Operator}, and let $\Psi_0(\bold{x}_1, \bold{x}_2, \ldots, \bold{x}_N):=\frac{1}{\sqrt{N!}} \text{\rm det}\; \big(\phi_{j}(\bold{x}_i)\big)_{ i, j=1}^N$ denote the reference determinant. Then we define the set $\widetilde{\mathcal{B}_{\wedge}} \subset \mathcal{B}_{\wedge}$ and the subspace $\widetilde{\mathcal{V}} \subset \widehat{\mathcal{H}}^1$ as
	\begin{align*}
		\widetilde{\mathcal{B}_{\wedge}}:=& \left\{\mathcal{X}_\mu \Psi_0 \colon \mu \in \mathcal{I}\right\}, \qquad \text{and}\\
		\widetilde{\mathcal{V}}:=&  \left\{\Psi_0\right\}^{\perp}:= \left\{\Phi \in \widehat{\mathcal{H}}^1\colon \quad (\Phi, \Psi_0)_{\widehat{\mathcal{L}}^2} =0\right\},
	\end{align*}
	and we observe that $\widetilde{\mathcal{B}_{\wedge}}$ is a complete, $\widehat{\mathcal{L}}^2$-orthonormal basis for $\widetilde{\mathcal{V}}$.
\end{definition}

\begin{definition}[Complementary Decomposition of $\widehat{\mathcal{H}}^1$] \label{lem:comp}~
	
	Let the excitation index set $\mathcal{I}$ be defined through Definition \ref{def:Excitation_Index}, let the excitation operators $\{\mathcal{X}_{\mu}\}_{\mu \in \mathcal{I}}$ be defined through Definition \ref{def:Excitation_Operator}, and let $\Psi_0(\bold{x}_1, \bold{x}_2, \ldots, \bold{x}_N):=\frac{1}{\sqrt{N!}} \text{\rm det}\; \big(\phi_{j}(\bold{x}_i)\big)_{ i, j=1}^N$ denote the reference determinant. We define $\mathbb{P}_0 \colon \widehat{\mathcal{H}}^1 \rightarrow  \widehat{\mathcal{H}}^1$ as the $\widehat{\mathcal{L}}^2$-orthogonal projection operator onto $ \text{span} \left\{\Psi_0\right\}$, and we define $P_0^{\perp}:=\mathbb{I}-\mathbb{P}_0 \colon \widehat{\mathcal{H}}^1 \rightarrow  \widehat{\mathcal{H}}^1$ as its complement. Additionally, we introduce the complementary decomposition of the $N$-particle space $\widehat{\mathcal{H}}^1$ given by
	\begin{align}\label{eq:decomp}
		\widehat{\mathcal{H}}^1=  \text{\rm span} \left\{\Psi_0\right\} \oplus \widetilde{\mathcal{V}}, \quad \text{where we emphasise that }~ \widetilde{\mathcal{V}}=\text{\rm Ran}\mathbb{P}_0^{\perp}.
	\end{align}
\end{definition}

The complementary decomposition introduced through Equation \eqref{eq:decomp} will be particularly important in our subsequent analysis of the coupled cluster method in Section \ref{sec:5}. Let us emphasise that the construction of these complementary spaces is based on $\widehat{\mathcal{L}}^2$-orthogonality rather than $\widehat{\mathcal{H}}^1$ orthogonality. This choice is intentional as it simplifies considerably the analysis in Sections \ref{sec:5} and \ref{sec:6}. Let us also remark that the projection operators $\mathbb{P}_0 \colon \widehat{\mathcal{H}}^1 \rightarrow  \widehat{\mathcal{H}}^1 $ and $\mathbb{P}^{\perp}_0 \colon \widehat{\mathcal{H}}^1 \rightarrow  \widehat{\mathcal{H}}^1 $ are both, nevertheless, bounded operators with respect to the $\Vert \cdot \Vert_{\widehat{\mathcal{H}}^1}$ norm since they both possess a closed range and a closed kernel.

Returning for the moment to the notion of excitation operators, we see that it is easy to deduce that each excitation operator $\mathcal{X}_{\mu}, ~ \mu \in \mathcal{I}$ is a bounded linear operator from $\widehat{\mathcal{H}}^1$ to $\widehat{\mathcal{H}}^1$. However, we will frequently be interested in so-called cluster operators which are summations of the excitation operators $\mathcal{X}_{\mu}, ~\mu \in \mathcal{I}$, and such summations need not be bounded operators from $\widehat{\mathcal{H}}^1$ to $\widehat{\mathcal{H}}^1$ or even from $\widehat{\mathcal{L}}^2$ to $\widehat{\mathcal{L}}^2$. Fortunately, the following result was proven in \cite{MR3021693}.

\begin{proposition}[Cluster Operators as Bounded Maps on $\widehat{\mathcal{L}}^2$]\label{prop:Yvon}~
	
	Let the excitation index set $\mathcal{I}$ be defined through Definition \ref{def:Excitation_Index} and let $\bt= \{\bt_{\mu}\}_{\mu \in \mathcal{I}} \in \ell^2(\mathcal{I})$. Then there exists a unique bounded linear operator $\mathcal{T} \colon\widehat{\mathcal{L}}^2 \rightarrow \widehat{\mathcal{L}}^2$, the so-called \emph{cluster operator} generated by $\bt$, such that $\mathcal{T}=\sum_{\mu \in \mathcal{I}} \bt_{\mu}\mathcal{X}_{\mu}$ where the series convergence holds with respect to the operator norm $\Vert \cdot \Vert_{\widehat{\mathcal{L}}^2 \to\widehat{\mathcal{L}}^2} $.
\end{proposition}

Next, we introduce a coefficient subspace of $\ell^2(\mathcal{I})$ which will limit the class of cluster operators that we consider in the sequel.   

\begin{definition}[Coefficient Space For Cluster Operators]\label{def:coeff_space}~
	
	Let the excitation index set $\mathcal{I}$ be defined through Definition \ref{def:Excitation_Index}. We define the Hilbert space of sequences $\mathbb{V} \subset \ell^2(\mathcal{I})$ as the set
	\begin{equation}\label{eq:coeff_space}
		\mathbb{V}:= \left\{\bold{t}:= \{\bt_{\mu}\}_{\mu \in \mathcal{I}} \in \ell^2(\mathcal{I}) \colon\quad \sum_{\mu \in \mathcal{I}}\bt_{\mu} \Psi_\mu ~\in \widehat{\mathcal{H}}^1 \right\},
	\end{equation}
	equipped with the inner product
	\begin{equation}\label{eq:coeff_inner}
		\forall \bold{t}, \bs \in \mathbb{V} ~\text{\rm with }~ \bt:= (\bt_{\mu})_{\mu \in \mathcal{I}}, ~\bs:= (\bs_{\mu})_{\mu \in \mathcal{I}} \colon \qquad \left(\bs, \bt\right)_{\mathbb{V}} := \left(\sum_{\mu \in \mathcal{I}}\bs_{\mu} \Psi_\mu, \sum_{\mu \in \mathcal{I}}\bt_{\mu} \Psi_\mu\right)_{\widehat{\mathcal{H}}^1}.
	\end{equation}
	
	Additionally, we define $\mathbb{V}^*$ as the topological dual space of $\mathbb{V}$, equipped with the canonical dual norm
	\begin{align*}
		\forall \bw \in \mathbb{V}^* \colon \qquad  \Vert \bw\Vert_{\mathbb{V}^*}:= \sup_{0\neq \bt \in \mathbb{V}}\frac{\big \vert \left\langle \bw, \bt\right \rangle_{\mathbb{V}^* \times \mathbb{V}} \big \vert}{\Vert \bt \Vert_{\mathbb{V}}},
	\end{align*} 
\end{definition}

Some remarks are now in order.

\begin{remark}[Clarification of the Definition of the Coefficient Space]~
	
	Consider Definition \ref{def:coeff_space} of the coefficient space $\mathbb{V} \subset \ell^2(\mathcal{I})$ and let $\bt = \{\bt_{\mu}\}_{\mu \in \mathcal{I}} \in \mathbb{V}$. We emphasise here that the assertion $\sum_{\mu \in \mathcal{I}}\bt_{\mu} \Psi_\mu \in \widehat{\mathcal{H}}^1$ should be understood in the following sense: there exists $\Psi_{\bold{t}} \in \widehat{\mathcal{H}}^1 \subset \widehat{\mathcal{L}}^2$ such that $\Psi_{\bold{t}} =\sum_{\mu \in \mathcal{I}}\bt_{\mu} \Psi_\mu$ where the series convergence holds with respect to the $\widehat{\mathcal{L}}^2$ norm. In particular, this series convergence does not a priori hold with respect to the $\widehat{\mathcal{H}}^1$-norm, and it is only the limit function $\Psi_{\bt}$ that is an element of $\widehat{\mathcal{H}}^1$.
\end{remark}

\begin{remark}[Dual Coefficient Space]\label{rem:coeff_space}~
	
	Consider the setting of Definition \ref{def:coeff_space}. Throughout, this article, we will denote by $\mathbb{V}^*$  the topological dual space of $\mathbb{V}$ equipped with the canonical dual norm. Note that since $\widehat{\mathcal{H}}^1$ is dense and continuously embedded in $\widehat{\mathcal{L}}^2$, we can deduce that the coefficient space $\mathbb{V}$ is dense and continuously embedded in $\ell^2(\mathcal{I})$. As a consequence, the inner product $(\cdot, \cdot)_{\ell^2}$ can be continuously extended to the duality pairing $\langle \cdot, \cdot\rangle_{\mathbb{V} \times \mathbb{V}^*}$ on $\mathbb{V} \times \mathbb{V}^*$. This fact will be of occasional use in the sequel.
 
 %  It is not difficult to see that the dual space $\mathbb{V}^*$ is isometrically isomorphic to the Banach space of sequences $\widehat{\mathbb{V}} \subset \ell^{\infty}(\mathcal{I})$ given by
	% \begin{align}
	% 	\widehat{\mathbb{V}}:= \left\{{\bw}:= (\bw_{\mu})_{\mu \in \mathcal{I}} \in \ell^{\infty}(\mathcal{I}) \colon\quad ~\exists C_{\bw}>0 ~\text{ such that } ~\forall \bt \in \mathbb{V} ~\text{ it holds that }~ \left( \bw, \bt\right)_{\ell^2(\mathcal{I})} < C_{\bw} \Vert \bt\Vert_{\mathbb{V}} \right\},
	% \end{align}
	% equipped with the norm
	% \begin{align}
	% 	\forall \bw \in \widehat{\mathbb{V}} \colon \qquad  \Vert \bw\Vert_{\widehat{\mathbb{V}}}:= \sup_{0\neq \bt \in \mathbb{V}}\frac{\big \vert \left(\bw, \bt\right)_{\ell^2(\mathcal{I})} \big \vert}{\Vert \bt \Vert_{\mathbb{V}}}.
	% \end{align}
	
	% By an abuse of notation therefore, we will frequently regard $\mathbb{V}^*$ as a Banach space of sequences satisfying appropriate boundedness assumptions.
\end{remark}

% \begin{remark}[Sequences in the Dual Coefficient Space]\label{rem:coeff_space}~
	
% 	Consider the setting of Definition \ref{def:coeff_space}. It is not difficult to see that the dual space $\mathbb{V}^*$ is isometrically isomorphic to the Banach space of sequences $\widehat{\mathbb{V}} \subset \ell^{\infty}(\mathcal{I})$ given by
% 	\begin{align}
% 		\widehat{\mathbb{V}}:= \left\{{\bw}:= (\bw_{\mu})_{\mu \in \mathcal{I}} \in \ell^{\infty}(\mathcal{I}) \colon\quad ~\exists C_{\bw}>0 ~\text{ such that } ~\forall \bt \in \mathbb{V} ~\text{ it holds that }~ \left( \bw, \bt\right)_{\ell^2(\mathcal{I})} < C_{\bw} \Vert \bt\Vert_{\mathbb{V}} \right\},
% 	\end{align}
% 	equipped with the norm
% 	\begin{align}
% 		\forall \bw \in \widehat{\mathbb{V}} \colon \qquad  \Vert \bw\Vert_{\widehat{\mathbb{V}}}:= \sup_{0\neq \bt \in \mathbb{V}}\frac{\big \vert \left(\bw, \bt\right)_{\ell^2(\mathcal{I})} \big \vert}{\Vert \bt \Vert_{\mathbb{V}}}.
% 	\end{align}
	
% 	By an abuse of notation therefore, we will frequently regard $\mathbb{V}^*$ as a Banach space of sequences satisfying appropriate boundedness assumptions.
% \end{remark}    

\begin{Notation}[Coefficient Sequences and Cluster Operators]~
	
	Let $\bt =\{\bt_{\mu}\}_{\mu \in \mathcal{I}} \in \mathbb{V}$. As mentioned previously, the operator $\mathcal{T}:=\sum_{\mu \in \mathcal{I}}\bt_{\mu} \mathcal{X}_{\mu}$ is known as the cluster operator generated by $\bt$, and it plays a key role in the coupled cluster formalism.
	
	For clarity of exposition, we will adopt the convention of denoting by small bold letters such as $\br, \bs, \bt$, and $\bw$, etc., elements of the coefficient space $\mathbb{V}$ and denoting by capital curly letters such as $\mathcal{R}, \mathcal{S}, \mathcal{T}$, and $\mathcal{W}$, etc., the corresponding cluster operators with the understanding that $\mathcal{R}:= \sum_{\mu \in \mathcal{I}}\br_{\mu} \mathcal{X}_{\mu}$, $\mathcal{S}:= \sum_{\mu \in \mathcal{I}}\bs_{\mu} \mathcal{X}_{\mu}$, and so on.
\end{Notation}

\begin{remark}[Representation of Elements of the Complementary Subspace $\widetilde{\mathcal{V}}$]\label{rem:comp}~
	
	Consider the setting of Definition \ref{def:coeff_space} and recall Definition \ref{def:V_K} of the space $\widetilde{\mathcal{V}} \subset \widehat{\mathcal{H}}^1$. It is not difficult to see that every element ${\Phi}_{\bs}:= \sum_{\mu \in \mathcal{I}} \bold{s}_{\mu} \mathcal{X}_{\mu} \Psi_0 \in \widetilde{\mathcal{V}}$ generates a sequence $\bold{s}:= \{\bold{s}_\mu\}_{\mu \in \mathcal{I}} \in \mathbb{V}$ such that
	\begin{align*}
		{\Phi}_{\bs}= \sum_{\mu \in \mathcal{I}}\bold{s}_{\mu} \mathcal{X}_{\mu} \Psi_0  = \mathcal{S}\Psi_0.% \quad \text{where we have defined }~    \mathcal{S}:= \sum_{\mu \in \mathcal{I}}\bold{s}_{\mu} \mathcal{X}_{\mu}.
	\end{align*}
	Conversely, given any sequence $\bold{w}:= \{\bold{w}_\mu\}_{\mu \in \mathcal{I}} \in \mathbb{V}$, we can define the function ${\Phi}_{\bw} \in \widetilde{\mathcal{V}}$ as
	\begin{align*}
		{\Phi}_{\bw}=\sum_{\mu \in \mathcal{I}}\bold{w}_{\mu} \mathcal{X}_{\mu} \Psi_0= \mathcal{W}\Psi_0. % \quad \text{where we have defined }   \mathcal{W}:= \sum_{\mu \in \mathcal{I}}\bold{w}_{\mu} \mathcal{X}_{\mu}.
	\end{align*}
	
	Therefore, in the sequel (in particular in Section \ref{sec:5}), we will occasionally write elements of the space $\widetilde{\mathcal{V}}$ as, for instance, $\mathcal{S}\Psi_0$ or $\mathcal{W}\Psi_0$ where $\mathcal{S}:= \sum_{\mu \in \mathcal{I}}\bold{s}_{\mu} \mathcal{X}_{\mu}$ and $\mathcal{W}:= \sum_{\mu \in \mathcal{I}}\bold{w}_{\mu} \mathcal{X}_{\mu}$ for some sequences $\bold{s}:= \{\bold{s}_\mu\}_{\mu \in \mathcal{I}} \in \mathbb{V}$ and  $\bold{w}:= \{\bold{w}_\mu\}_{\mu \in \mathcal{I}} \in \mathbb{V}$.
\end{remark}

The following theorem now summarises the main properties of the excitation operators $\mathcal{X}_\mu, ~\mu \in \mathcal{I}$ and cluster operators constructed from these excitation operators. The establishment of these properties in infinite dimensions was the main achievement of the article \cite{MR3021693}. In finite-dimensions, where the situation is considerably simpler from a topological point of view, these results were first proven in the mathematical literature in \cite{Schneider_1}.

\vspace{5mm}

\begin{theorem}[Properties of Excitation and Cluster Operators]\label{pro:excit}~
	
	Let the excitation index set $\mathcal{I}$ be defined through Definition \ref{def:Excitation_Index}, let the excitation operators $\{\mathcal{X}_{\mu}\}_{\mu \in \mathcal{I}}$ and de-excitation operators $\{\mathcal{X}^{\dagger}_{\mu}\}_{\mu \in \mathcal{I}}$ be defined through Definitions \ref{def:Excitation_Operator} and \ref{def:De-excitation_Operator} respectively, and let the Hilbert space $\mathbb{V}$ of sequences be defined through Definition \ref{def:coeff_space}. Then
	\begin{enumerate}
		\item For all $\mu, \nu \in \mathcal{I}$, it holds that $\mathcal{X}_{\mu} \mathcal{X}_{\nu}=\mathcal{X}_{\nu} \mathcal{X}_{\mu}$ and $\mathcal{X}^{\dagger}_{\mu} \mathcal{X}^{\dagger}_{\nu}=\mathcal{X}^{\dagger}_{\nu} \mathcal{X}^{\dagger}_{\mu}$ .
		
		\item For every $\Phi \in \widehat{\mathcal{H}}^1$ that satisfies the so-called intermediate normalisation condition $\left(\Phi, \Psi_0\right)_{\mathcal{L}^2}=1$, there exists a unique sequence $\bold{r}=\left\{\bold{r}_\mu\right\}_{\mu \in \mathcal{I}} \in \mathbb{V}$ with corresponding cluster operator $\mathcal{R}=\sum_{\mu \in \mathcal{I}}\br_{\mu}\mathcal{X}_{\mu}$ such that
		\begin{align*}
			\Phi = \Psi_0 + \mathcal{R} \Psi_0.
		\end{align*}
		
		\vspace{-4mm}
		\item Let $\bt \in \mathbb{V}$. Then
		\begin{itemize}
			\item The cluster operator $\mathcal{T}= \sum_{\mu \in \mathcal{I}}\bt_{\mu}\mathcal{X}_{\mu}$ is a bounded linear map from $\widehat{\mathcal{H}}^1$ to $\widehat{\mathcal{H}}^1$ and there exists a constant $\beta>0$ depending only on $N$ such that
			\begin{align*}
				\Vert \bt \Vert_{\mathbb{V}} \leq \Vert \mathcal{T}\Vert_{\widehat{\mathcal{H}}^1 \to \widehat{\mathcal{H}}^1} \leq \beta \Vert \bt \Vert_{\mathbb{V}}.
			\end{align*}
			
			\item The de-excitation cluster $\mathcal{T}^{\dagger}= \sum_{\mu \in \mathcal{I}}\bt_{\mu}\mathcal{X}^{\dagger}_{\mu}$ is also bounded linear map from $\widehat{\mathcal{H}}^1$ to $\widehat{\mathcal{H}}^1$ and there exists a constant $\beta^{\dagger}>0$ depending only on $N$ such that
			\begin{align*}
				\Vert \mathcal{T}^{\dagger}\Vert_{\widehat{\mathcal{H}}^1 \to \widehat{\mathcal{H}}^1} \leq \beta^{\dagger} \Vert \bt \Vert_{\mathbb{V}}.
			\end{align*}
			
			\item The cluster operator $\mathcal{T}= \sum_{\mu \in \mathcal{I}}\bt_{\mu}\mathcal{X}_{\mu}$ has an extension to a bounded linear operator from $\widehat{\mathcal{H}}^{-1}$ to $\widehat{\mathcal{H}}^{-1}$.
			
		\end{itemize}

		\item Define the set of operators
		\begin{align*}
			\mathfrak{L}:= \left\{t_0 {\rm I} +  \mathcal{T} \colon \quad t_0 \in \mathbb{R}, ~\mathcal{T}=\sum_{\mu \in \mathcal{I}}  \bt_{\mu} \mathcal{X}_{\mu} \quad \text{such that }~ \bt=\{\bt_{\mu}\}_{\mu \in \mathcal{I}} \in \mathbb{V} \right\}
		\end{align*}
		The following hold:
		
		\begin{itemize}
			\item The set $\mathfrak{L}$ forms a closed commutative subalgebra in the algebra of bounded linear operators acting from $\widehat{\mathcal{H}}^{1}$ to $\widehat{\mathcal{H}}^{1}$ (and also from $\widehat{\mathcal{H}}^{-1}$ to $\widehat{\mathcal{H}}^{-1}$).
			
			\item The subalgebra $\mathfrak{L}$ is closed under inversion and the spectrum of any $\mathfrak{L}\ni \mathcal{L}= t_0{\rm I} +  \mathcal{T}$ is exactly $\sigma(\mathcal{L})=\{t_0\}$.
			
			\item Any element in $\mathfrak{L}$ of the form $\mathcal{T} =\sum_{\mu \in \mathcal{I}}  \bt_{\mu} \mathcal{X}_{\mu}$ with $\bt \in \mathbb{V}$ is nilpotent: it holds that $\mathcal{T}^{N+1}\equiv 0$.
			
			\item The exponential function is a locally $\mathscr{C}^{\infty}$ diffeomorphism from $\mathfrak{L}$ to $\mathfrak{L}\setminus \{0\}$.
		\end{itemize} 
	\end{enumerate}
\end{theorem}

As a consequence of Theorem \ref{pro:excit} (c.f., Properties $(\textit{2})$ and $(\textit{4})$), one can prove that any intermediately normalised element of the $N$-particle space can be parameterised through an exponential cluster operator. More precisely, given the excitation index set $\mathcal{I}$ defined through Definition \ref{def:Excitation_Index} and the excitation operators $\{\mathcal{X}_{\mu}\}_{\mu \in \mathcal{I}}$ defined through Definition \ref{def:Excitation_Operator}, for any $\Phi \in \widehat{\mathcal{H}}^1$ such that $\left(\Phi, \Psi_0\right)_{\mathcal{L}^2}=1$, there exists a unique sequence $\bt =\{\bt_{\mu}\}_{\mu \in \mathcal{I}} \in \mathbb{V}$ and a unique cluster operator $\mathcal{T}=\sum_{\mu \in \mathcal{I}} \bt_{\mu} \mathcal{X}_{\mu}$ such that
\begin{align}\label{thm:3}
	\Phi = e^{\mathcal{T}} \Psi_0.
\end{align}
A proof of this statement in the infinite-dimensional setting can be found in \cite{MR3021693} while the corresponding proof for the finite-dimensional case is given in \cite{Schneider_1}.

Equation \eqref{thm:3} implies in particular that if the sought-after ground state wave-function $\Psi^*_{\rm GS} \in \widehat{\mathcal{H}}^1$ that solves the minimisation problem \eqref{eq:Ground_State} is intermediately normalised, then it can also  be written in the form
\begin{align*}
	\Psi^*_{\rm GS} = e^{\mathcal{T}^*}\Psi_0,
\end{align*}
for some sequence $\bt^* = \{\bt^*_{\mu}\}_{\mu \in \mathcal{I}}$ and corresponding cluster operator $\mathcal{T}^*=\sum_{\mu \in \mathcal{I}} \bt_{\mu}^* \mathcal{X}_{\mu}$. In other words, the minimisation problem \eqref{eq:Ground_State} can be replaced by an equivalent problem which consists of finding the sequence $\bt^* =\{\bt^*_{\mu}\}_{\mu \in \mathcal{I}} \in \mathbb{V}$ used to construct the appropriate cluster operator $\mathcal{T}^*$ that appears in the exponential parametrisation of $\Psi^*$. This leads to the following so-called \emph{continuous} coupled cluster equations.

\vspace{3mm}

\textbf{Continuous Coupled Cluster Equations:}~

Let the excitation index set $\mathcal{I}$ be defined through Definition \ref{def:Excitation_Index} and let the excitation operators $\{\mathcal{X}_{\mu}\}_{\mu \in \mathcal{I}}$ be defined through Definition \ref{def:Excitation_Operator}. We seek a sequence $\bt^* =\{\bt^*_{\mu}\}_{\mu \in \mathcal{I}} \in \mathbb{V}$ such that for all $\mu \in \mathcal{I}$ we have
\begin{equation}\label{eq:CC}
	\left\langle\mathcal{X}_\mu \Psi_0, e^{-\mathcal{T}^*}H e^{\mathcal{T}^*} \Psi_0\right\rangle_{\widehat{\mathcal{H}}^{1} \times \widehat{\mathcal{H}}^{-1}}	=0, \quad \text{ where } ~\mathcal{T}^*= \sum_{\mu \in \mathcal{I}} \bt^*_{\mu} \mathcal{X}_{\mu}.
\end{equation}

Once Equation \eqref{eq:CC} has been solved, the associated coupled cluster energy $\mathcal{E}_{\rm CC}^*$ is given by
\begin{equation}\label{eq:CC_Energy}
	\mathcal{E}_{\rm CC}^* := \left\langle\Psi_0, e^{-\mathcal{T}^*}H e^{\mathcal{T}^*} \Psi_0\right\rangle_{\widehat{\mathcal{H}}^{1} \times \widehat{\mathcal{H}}^{-1}}, \quad \text{ where }~ \mathcal{T}^*= \sum_{\mu \in \mathcal{I}} \bt^*_{\mu} \mathcal{X}_{\mu}.
\end{equation}
%and it is easy to deduce that in fact $\mathcal{E}_{\rm CC}^*  = \mathcal{E}^*$.

\begin{remark}[Solutions to the Continuous Coupled Cluster Equations]~
	
	Consider the continuous coupled cluster equations \eqref{eq:CC}. Under the assumption that the ground state wave-function $\Psi_{\rm GS}$ of the electronic Hamiltonian $H \colon \widehat{\mathcal{H}}^{1} \rightarrow \widehat{\mathcal{H}}^{-1}$ is intermediately normalisable with respect to the chosen reference determinant $\Psi_0$, i.e., it is not orthogonal to $\Psi_0$, it is obvious that there exists a corresponding solution to this non-linear system of equations. Indeed, by Equation \eqref{thm:3}, there exists a sequence $\bt_{\rm GS}^* =\{\bt^*_{\mu}\}_{\mu \in \mathcal{I}} \in \mathbb{V}$ such that $\Psi^*_{\rm GS}= e^{\mathcal{T}_{\rm GS}^*} \Psi_0$ with $\mathcal{T}_{\rm GS}^*= \sum_{\mu \in \mathcal{I}} \bt^*_{\mu} \mathcal{X}_{\mu}$, and it can readily be verified that this sequence $\bt_{\rm GS}^*$ solves exactly Equation \eqref{eq:CC}, and consequently $\mathcal{E}_{\rm CC}^*  = \mathcal{E}_{\rm GS}^*$.
	
	Of course, $\bt_{\rm GS}^*$ defined as above need not be the unique solution to the coupled cluster equations \eqref{eq:CC}. In fact, as we discuss in the next Section \ref{sec:5}, \emph{every} intermediately normalisable eigenfunction of the electronic Hamiltonian will generate a solution to Equation \eqref{eq:CC}. From a theoretical point of view, this means that only \emph{local} well-posedness results can be expected to hold for the continuous CC equations.
\end{remark}

The continuous coupled cluster equations are an infinite system of non-linear equations and thus cannot be solved exactly. Instead, one introduces an approximation of the continuous coupled cluster equations by considering, instead of the global excitation index $\mathcal{I}$, some finite subset ${\mathcal{I}_h} \subset \mathcal{I}$ and solving only the equations associated with this subset of excitation indices. This procedure results in the so-called discrete coupled cluster equations.

\vspace{3mm}

\textbf{Discrete Coupled Cluster Equations:}~

Let the excitation index set $\mathcal{I}$ be defined through Definition \ref{def:Excitation_Index}, let $\mathcal{I}_h \subset \mathcal{I}$ denote any \emph{finite} subset of excitation indices, and let the excitation operators $\{\mathcal{X}_{\mu}\}_{\mu \in \mathcal{I}}$ be defined through Definition \ref{def:Excitation_Operator}. We seek a coefficient vector $\bt^{*}_h=\{\bt^*_{\mu}\}_{\mu \in \mathcal{I}_h} \in \ell^2(\mathcal{I}_h)$ such that for all $\mu \in {\mathcal{I}}_h$ we have
\begin{equation}\label{eq:CC_truncated}
	\left\langle\mathcal{X}_\mu \Psi_0, e^{-\mathcal{T}^*_h}H e^{\mathcal{T}^*_h} \Psi_0\right\rangle_{\widehat{\mathcal{H}}^{-1} \times \widehat{\mathcal{H}}^{1}}	=0, \quad \text{ where } \mathcal{T}^*_h= \sum_{\mu \in {\mathcal{I}_h}} \bt^*_{\mu} \mathcal{X}_{\mu}.
\end{equation}

The associated discrete ground state energy $\mathcal{E}^*_{h, \rm CC}$ is given by
\begin{equation}\label{eq:CC_energy_proj}
	\mathcal{E}^*_{h, \rm CC} := \left\langle\Psi_0, e^{-\mathcal{T}_h^*}He^{\mathcal{T}^*_h} \Psi_0\right\rangle_{\widehat{\mathcal{H}}^{-1} \times \widehat{\mathcal{H}}^{1}}, \quad \text{ where } \mathcal{T}^*_h= \sum_{\mu \in \mathcal{I}_h}\bt^*_{\mu} \mathcal{X}_{\mu}.
\end{equation}

The discrete CC equations will be the subject of further discussion in Section \ref{sec:6} where we will analyse their well-posedness for some specific choices of the excitation index subsets and $N$-particle basis sets. For the moment, we conclude this section with a remark on the nature of the solutions to these discrete equations.

\begin{remark}[Solutions to the Discrete Coupled Cluster Equations]~
	
	Consider the discrete coupled cluster equations \eqref{eq:CC_truncated}. As in the continuous case, there is a priori no reason for solutions of Equation \eqref{eq:CC_truncated} to be globally unique. Indeed, numerical experience confirms that solutions to Equation~\eqref{eq:CC_truncated} are very often \underline{not} unique. Nevertheless, in practice the discrete CC equations are solved very frequently by the quantum chemical community when performing electronic structure calculations, usually using some type of iterative Newton method, and it is \emph{hoped} that if one starts from a sufficiently accurate initial point, then the resulting solution $\bt_h^* \in \ell^2(\mathcal{I}_h) $ of Equation \eqref{eq:CC_truncated} approximates, in some sense, an exact solution $\bt^*$ of the continuous CC equations \eqref{eq:CC} that generates the ground state wave-function. Of course there are no mathematical guarantees that this procedure works at all, and the current reputation of coupled cluster methods as a `gold-standard' in quantum computational chemistry seems to be based mostly on successful empirical experience.
\end{remark}

Having introduced the continuous and discrete coupled cluster equations, the remainder of this article will be concerned with their (local) well-posedness analysis. We will first analyse the continuous coupled cluster equations~\eqref{eq:CC} in Section \ref{sec:5}, following which we will study a particular class of the discrete coupled cluster equations \eqref{eq:CC_truncated} in Section \ref{sec:6}.

%% file: FullCC_New.tex
Throughout this section, we assume the setting of Sections \ref{sec:2} and \ref{sec:4}, and we recall in particular the notion of excitation operators and the continuous coupled cluster equations. We begin by defining the so-called coupled cluster function, which will be the main object of study in this section.

\begin{definition}[Coupled Cluster function]\label{def:CC_function}~
	
	Let the excitation index set $\mathcal{I}$ be defined through Definition \ref{def:Excitation_Index} and let the excitation operators $\{\mathcal{X}_{\mu}\}_{\mu \in \mathcal{I}}$ be defined through Definition \ref{def:Excitation_Operator}.  We define the coupled cluster function $\mathcal{f} \colon \mathbb{V}\rightarrow \mathbb{V}^*$ as the mapping with the property that for all $\bt=\{\bt_{\mu}\}_{\mu \in \mathcal{I}}, \bs=\{\bs_{\mu}\}_{\mu \in \mathcal{I}} \in \mathbb{V}$ it holds that 
	\begin{align*}
		\big\langle \bs, \mathcal{f}(\bt) \big\rangle_{\mathbb{V} \times \mathbb{V}^*}:=\left \langle \sum_{\mu \in \mathcal{I}}\bs_{\mu}\mathcal{X}_{\mu}\Psi_0, e^{-\mathcal{T}}He^{\mathcal{T}} \Psi_0\right \rangle_{\widehat{\mathcal{H}}^1 \times \widehat{\mathcal{H}}^{-1}}, \qquad \text{ where }~ \mathcal{T}=\sum_{\mu \in \mathcal{I}} \bt_\mu \mathcal{X}_\mu.
	\end{align*}
\end{definition}

\begin{remark}[Justification of the Domain and Range of Coupled Cluster Function]\label{rem:CC_function}~
	
	Consider Definition \ref{def:CC_function} of the coupled cluster function. The fact that $\mathcal{f}$ is indeed a mapping from $\mathbb{V}$ to $\mathbb{V}^*$ is a direct consequence of the boundedness of the electronic Hamiltonian $H\colon \widehat{\mathcal{H}}^1 \rightarrow \widehat{\mathcal{H}}^{-1}$ and the exponential cluster operators $e^{\mathcal{T}} \colon \widehat{\mathcal{H}}^1 \rightarrow\widehat{\mathcal{H}}^1$ and $e^{-\mathcal{T}} \colon \widehat{\mathcal{H}}^{-1}\rightarrow\widehat{\mathcal{H}}^{-1}$. Indeed, for all $\bs= \{\bs_{\mu}\}_{\mu \in \mathcal{I}} \in \mathbb{V}$ and all $\bt= \{\bt_\mu\}_{\mu \in \mathcal{I}} \in \mathbb{V}$ it holds that
	\begin{align*}
		\big\vert \big\langle \bs, \mathcal{f}(\bt)\big\rangle_{\mathbb{V} \times \mathbb{V}^*} \big\vert &= \left \vert \left \langle \sum_{\mu \in \mathcal{I}}\bs_{\mu}\mathcal{X}_{\mu}\Psi_0, e^{-\mathcal{T}}He^{\mathcal{T}} \Psi_0\right \rangle_{\widehat{\mathcal{H}}^1 \times \widehat{\mathcal{H}}^{-1}} \right \vert\\[1em]
		&\leq \Big\Vert \sum_{\mu\in \mathcal{I}}\bs_{\mu} \mathcal{X}_{\mu} \Psi_0\Big\Vert_{\widehat{\mathcal{H}}^1} \Big\Vert e^{-\mathcal{T}}He^{\mathcal{T}} \Psi_0\Big\Vert_{\widehat{\mathcal{H}}^{-1}}\\[1em]
		&\leq \Vert \bs\Vert_{\mathbb{V}} \Vert e^{-\mathcal{T}}\Vert_{\widehat{\mathcal{H}}^{-1} \to \widehat{\mathcal{H}}^{-1}} \Vert H\Vert_{\widehat{\mathcal{H}}^{1} \to \widehat{\mathcal{H}}^{-1}} \Vert e^{\mathcal{T}} \Psi_0\Vert_{\widehat{\mathcal{H}}^{1}}.
	\end{align*}
	%and as a consequence, for any $\bt \in \mathbb{V}$, the mapping $\big(\cdot, \mathcal{f}(\bt)\big)_{\ell^2} $ defines a bounded linear functional on $\mathbb{V}$. Recalling Remark \ref{rem:coeff_space} completes the justification.
\end{remark}

Equipped with Definition \ref{def:CC_function} of the coupled cluster function, let us point out that the continuous coupled cluster equations \eqref{eq:CC} can be re-written in the following weak form.

\vspace{3mm}
%which will particularly viewed as a \emph{Galerkin} approximation to the untruncated coupled cluster equations \eqref{eq:CC}.

\textbf{Weak Form of the Continuous Coupled Cluster equations:}~

Let the excitation index set $\mathcal{I}$ be defined through Definition \ref{def:Excitation_Index}, let the Hilbert space of sequences $\mathbb{V} \subset \ell^2(\mathcal{I})$ be defined through Definition \ref{def:coeff_space}, and let the coupled cluster function $\mathcal{f}\colon \mathbb{V} \rightarrow \mathbb{V}^*$ be defined through Definition \ref{def:CC_function}. We seek a sequence $\bt =\{\bt_{\mu}\}_{\mu \in \mathcal{I}} \in \mathbb{V}$ such that for all sequences $\bs =\{\bs_{\mu}\}_{\mu \in \mathcal{I}} \in \mathbb{V}$ it holds that
\begin{equation}\label{eq:Galerkin_Full_CC}
	\langle \bs, \mathcal{f}(\bt)\rangle_{\mathbb{V} \times \mathbb{V}^*}=0.
\end{equation}        
As we shall see in Section \ref{sec:6}, this point of view will allow us to interpret the \emph{truncated} coupled cluster equations~\eqref{eq:CC_truncated} as Galerkin discretisations of Equation~\eqref{eq:Galerkin_Full_CC}, which will be useful for the purpose of the numerical analysis.

The following extremely significant result, proven in \cite{MR3021693}, establishes a precise relationship between zeros of the coupled cluster function defined through Definition \ref{def:CC_function} (i.e., solutions of the continuous CC equations \eqref{eq:Galerkin_Full_CC}) and intermediately normalised eigenfunctions of the electronic Hamiltonian defined through Equation \eqref{eq:Hamiltonian}.

\begin{theorem}[Relation between Coupled Cluster Zeros and Eigenfunctions of Electronic Hamiltonian]\label{thm:Yvon}~

	Let the coupled cluster function $\mathcal{f} \colon \mathbb{V} \rightarrow \mathbb{V}^*$ be defined through Definition \ref{def:CC_function} and let the electronic Hamiltonian be given by Equation \eqref{eq:Hamiltonian}. Then 
	\begin{enumerate}
		\item For any zero $\bt^* = \{\bt_{\mu}^*\}_{\mu \in \mathcal{I}}\in \mathbb{V}$ of the CC function, the function $\Psi^*=e^{\mathcal{T}^*}\Psi_0 \in \widehat{\mathcal{H}}^1$ with $\mathcal{T}^*=\sum_{\mu \in \mathcal{I}} \bt^*_\mu \mathcal{X}_\mu$ is an intermediately normalised eigenfunction of the electronic Hamiltonian. Moreover, the eigenvalue corresponding to the eigenfunction $\Psi^*$ coincides with the CC energy $\mathcal{E}_{\rm CC}^*$ generated by $\bt^*$ as defined through Equation \eqref{eq:CC_Energy}. 
		
		\item Conversely, for any intermediately normalised eigenfunction $\Psi^* \in \widehat{\mathcal{H}}^1$ of the electronic Hamiltonian, there exists $\bt^* = \{\bt_{\mu}^*\}_{\mu \in \mathcal{I}}\in \mathbb{V}$ such that $\bt^*$ is a zero of the CC function and $\Psi^*=e^{\mathcal{T}^*}\Psi_0 \in \widehat{\mathcal{H}}^1$ with $\mathcal{T}^*=\sum_{\mu \in \mathcal{I}} \bt^*_\mu \mathcal{X}_\mu$. Moreover, the CC energy $\mathcal{E}_{\rm CC}^*$ generated by $\bt^*$ as defined through Equation \eqref{eq:CC_Energy} coincides with the eigenvalue corresponding to the eigenfunction $\Psi^*$.
	\end{enumerate}
\end{theorem}

In other words every intermediately normalisable eigenfunction of the electronic Hamiltonian \eqref{eq:Hamiltonian} corresponds to a zero of the coupled cluster function defined through Definition \ref{def:CC_function} and vice-versa. The goal of our analysis in this section is to study the nature of these zeros of the coupled cluster function and, in particular, to derive sufficient conditions that guarantee the \emph{simplicity} of the zeros. Indeed, if we know that some $\bt^* \in \mathbb{V}$ is a simple zero of the coupled cluster function, then this will allow us to deduce \emph{local invertibility} of the coupled cluster function at $\bt^* \in \mathbb{V}$ and thereby derive both \emph{local uniqueness} and \emph{local residual-based} error estimates for the CC equations \eqref{eq:Galerkin_Full_CC}. Arguments of this nature are standard in the literature on non-linear numerical analysis (see, e.g., \cite[Proposition 2.1]{MR1213837} or \cite[Theorem 2.1]{MR1470227}) and are usually based on the invertibility of the Fr\'echet derivative of the non-linear function being studied. The next step in our analysis therefore will be to study carefully the Fr\'echet derivative of the coupled cluster function. Before proceeding with this analysis however, let us comment on the existing numerical analysis of the coupled cluster equation \eqref{eq:Galerkin_Full_CC}.

\begin{remark}[Existing Approaches in the Numerical Analysis of the CC equations \eqref{eq:Galerkin_Full_CC}]\label{rem:monotone}~

The existing literature on the numerical analysis of coupled cluster methods is rather sparse. The first numerical analysis of the single reference coupled cluster-- in the finite-dimensional setting-- is due to R. Schneider in \cite{Schneider_1}. The analysis carried out in \cite{Schneider_1} was then extended to the infinite-dimensional setting (as considered here) in the subsequent articles \cite{MR3021693} and \cite{MR3110488}. The former article showed that the mathematical objects used to formulate the coupled cluster method (such as excitation operators) are bounded operators on appropriate infinite-dimensional Hilbert spaces so that the coupled cluster equations can be stated in infinite-dimensions (prior to this article, the CC equations were \emph{always} written in a finite-dimensional setting). The article \cite{MR3110488} used these tools and the ideas developed in~\cite{Schneider_1} to perform a numerical analysis of the infinite-dimensional coupled cluster equations. Additional articles on the mathematical analysis of CC methods have since appeared, including \cite{laestadius2018analysis} which studies the so-called extended coupled cluster method, \cite{faulstich2019analysis} which studies the so-called tailored coupled cluster method, \cite{csirik2021coupled_2} which studies the finite-dimensional CC equations using topological degree theory, and \cite{faulstich2022coupled} which analyses the root structure of the CC equations using tools from algebraic geometry.

The aforementioned articles have two important features in common: First they are concerned with the (local) analysis of the `ground-state' zero of the coupled cluster function, i.e., with the zero $\bt^*_{\rm GS}$ such that
$e^{\mathcal{T}^*_{\rm GS}}\Psi_0= \Psi^*_{\rm GS}$. This of course makes sense since the vast majority of coupled cluster calculations are targeted at approximating the ground state energy of the electronic Hamiltonian.

Second and more importantly, the well-posedness analysis in all of the above articles is based on proving a  \emph{local, strong monotonicity} property of the coupled cluster function at $\bt^*_{\rm GS}$. Taking the example of the article \cite{MR3110488} whose notation closely aligns with ours, let the coupled cluster function $\mathcal{f} \colon \mathbb{V} \rightarrow \mathbb{V}^*$ be defined through Definition \ref{def:CC_function}. Then it is shown in \cite{MR3110488} that for $\delta >0$ sufficiently small, there exists a constant $\Gamma$ such that for all $\bw, \bs \in \mathbb{B}_{\delta}(\bt^*_{\rm GS}) \subset \mathbb{V}$ it holds that 
\begin{align}\label{eq:loc_mono}
	\left\langle \bw - \bs, \mathcal{f}(\bw) -\mathcal{f}(\bs)\right\rangle_{\mathbb{V} \times \mathbb{V}^*} \geq \Gamma \Vert \bw - \bs\Vert_{\mathbb{V}}.
\end{align}

If the constant $\Gamma$ can be shown to be strictly positive, then the local monotonicity property \eqref{eq:loc_mono} immediately yields local well-posedness of both the continuous coupled cluster equations as well as sufficiently rich Galerkin discretisations thereof\hspace{0.8mm}\footnote{Although this point of view is not taken in these articles, the local monotonicity condition \eqref{eq:loc_mono} essentially corresponds to proving that the coupled cluster Fr\'echet derivative at $\bt^*_{\rm GS}$ is a positive definite operator.}. Quasi-optimal error estimates for the CC energy can then also be derived using the dual weighted residual approach developed by Rannacher et al. \cite[Chapter 6]{MR1960405} . 

The main drawback of the above approach is that the actual local monotonicity constant $\Gamma$ derived from this analysis (see \cite[Theorem 3.4]{MR3110488}) is of the form:
\begin{subequations}
	\begin{eqnarray}\label{eq:loc_mono_2}
		\Gamma &= \omega \gamma -   \left\Vert \mathcal{T}^*_{\rm GS} - (\mathcal{T}^*_{\rm GS})^{\dagger}\right\Vert_{\widehat{\mathcal{H}}^1\rightarrow  \widehat{\mathcal{H}}^{1}}\; \left\Vert H- \mathcal{E}_{\rm GS}^*\right\Vert_{\widehat{\mathcal{H}}^1\rightarrow  \widehat{\mathcal{H}}^{-1}}  - \mathcal{O}\left(\Vert \bt^*_{\rm GS}\Vert^2_{\mathbb{V}}\right)\\[0.5em] \label{eq:loc_mono_3}
		&\geq \omega \gamma - \left(\beta +\beta^{\dagger}\right) \left\Vert \bt ^{*}_{\rm GS}\right\Vert_{\mathbb{V}}\left\Vert H- \mathcal{E}_{\rm GS}^*\right\Vert_{\widehat{\mathcal{H}}^1\rightarrow  \widehat{\mathcal{H}}^{-1}} - \mathcal{O}\left(\Vert \bt^*_{\rm GS}\Vert^2_{\mathbb{V}}\right)
	\end{eqnarray}
\end{subequations}
where $\gamma>0$ denotes the coercivity constant of the shifted electronic Hamiltonian $H- \mathcal{E}^*_{\rm GS}$ on $\{\Psi_{\rm GS}^*\}^{\perp}$, the constant $\omega \in (0, 1)$ is a prefactor depending on $\Vert \Psi_0 - \Psi_{\rm GS}^* \Vert_{\widehat{\mathcal{L}}^2}$,  and $\beta, \beta^{\dagger}$ are the continuity constants of the mappings $\mathbb{V}\ni \bt \mapsto \mathcal{T}\colon \widehat{\mathcal{H}}^1\rightarrow  \widehat{\mathcal{H}}^{1}$ and  and $\mathbb{V}\ni \bt \mapsto \mathcal{T}^{\dagger}\colon \widehat{\mathcal{H}}^1\rightarrow  \widehat{\mathcal{H}}^{1}$ respectively as given in Theorem \ref{pro:excit}.

Consequently, the constant $\Gamma$ is positive provided that $\left\Vert \bt ^{*}_{\rm GS}\right\Vert_{\mathbb{V}}$ is small enough. However, according to the theoretical analysis in \cite{MR3021693}, the constants $\beta, \beta^{\dagger}$ grow \emph{combinatorially} in the number of electrons $N$ in the system, and thus as soon as $N\approx 10$ or larger, the lower bound \eqref{eq:loc_mono_3} for the constant $\Gamma$ is no longer positive. Similar issues arise in the local monotonicity constants derived in the other articles~\cite{faulstich2019analysis} and \cite{laestadius2018analysis}.

To make matters worse, even if we rely on the sharper Inequality \eqref{eq:loc_mono_2}, numerical experiments involving small, relatively well-behaved molecules for which it is well-known (from numerical experience) that the coupled cluster method works well, reveal that (see Table \ref{table1} below)
\begin{align*}
	\left\Vert \mathcal{T}^*_{\rm GS} - (\mathcal{T}^*_{\rm GS})^{\dagger}\right\Vert_{\widehat{\mathcal{H}}^1\rightarrow  \widehat{\mathcal{H}}^{1}}\; \left\Vert H- \mathcal{E}_{\rm GS}^*\right\Vert_{\widehat{\mathcal{H}}^1\rightarrow  \widehat{\mathcal{H}}^{-1}} > \gamma.
\end{align*}
In other words, the assumptions required to establish local strong monotonicity of the coupled cluster function, namely, smallness of the amplitude vector norm $\left\Vert \bt ^{*}_{\rm GS}\right\Vert_{\mathbb{V}}$ seem very restrictive and not satisfied in many practical examples. As a consequence, the hope of obtaining \emph{quantitative} a posteriori error estimates for the coupled cluster equations seems difficult to achieve. It is our purpose in this article to remedy this situation by, as a first step, filling these gaps in the existing well-posedness analysis.

\begin{table}[h!]
	\centering
	\begin{tabular}{||c| c| c| c| |c|c||} 
		\hline \hline
		Molecule & \shortstack{Coercivity \\ constant $\gamma$} & $\Vert \bt^*_{\rm FCI}\Vert_{\mathbb{V}} $ &  \shortstack{Monotonicity\\ constant $\Gamma$ \\from Eq. \eqref{eq:loc_mono_2}} & \shortstack{Hartree-Fock\\ Energy Error \\ (Hartree)} &  \shortstack{CCSD\\ Energy Error \\ (Hartree)}\\ [0.5ex] 
		\hline\hline
		{\rm BeH2} & 0.3257 & 0.2343 &  \hphantom{-}0.0363& $3.50\times 10^{-2}$ & $3.83\times 10^{-4}$\\ 
		{\rm BH3} & 0.2903 & 0.2844 & {\color{red}-0.0950}& $5.40\times 10^{-2}$ &  $3.74\times 10^{-4}$\\
		\rm{HF} & 0.3010 & 0.2038 & {\color{red}-0.0083} & $2.81\times 10^{-2}$&$3.02\times 10^{-5}$\\
		\rm{H2O} & 0.3471 & 0.2687 &  \hphantom{-}0.0249 & $5.01\times 10^{-2}$&$1.18\times 10^{-4}$\\ 
		\rm{LiH} & 0.2617 & 0.1792 & {\color{red}-0.0065} & $2.04\times 10^{-2}$&$1.14\times 10^{-5}$\\
		\rm{NH3} & 0.3868 & 0.3074 & {\color{red}-0.0325} & $6.61\times 10^{-2}$&$2.18\times 10^{-4}$\\[1ex] 
		\hline\hline
	\end{tabular}
	\vspace{2mm}
	\caption{Examples of numerically computed local monotonicity constants for a collection of small molecules at equilibrium geometries. The calculations were performed in STO-6G basis sets with the exception of the HF and LiH molecules for which 6-31G basis sets were used. In all cases, the Full-CI solution was taken as the reference solution. To simplify calculations, the canonical $\widehat{\mathcal{H}}^1$ norm was replaced with an equivalent norm induced by the diagonal part of the shifted Full-CI Hamiltonian.}
	\label{table1}
\end{table}
\end{remark}

We begin our analysis proper with the following proposition whose essence seems known (c.f., \cite[Theorem 4.16]{Schneider_1}, \cite[Lemma 3.1]{MR3110488} and \cite[Lemma 4.6]{csirik2021coupled_1}) but that has not been expressed in the current form in the existing literature. 
\begin{proposition}[Coupled Cluster Fr\'echet Derivative] \label{prop:CC_der}~		

Let the excitation index set $\mathcal{I}$ be defined through Definition \ref{def:Excitation_Index}, let the excitation operators $\{\mathcal{X}_{\mu}\}_{\mu \in \mathcal{I}}$ be defined through Definition \ref{def:Excitation_Operator}, and let the coupled cluster function $\mathcal{f} \colon \mathbb{V} \rightarrow \mathbb{V}^* $ be defined through Definition \ref{def:CC_function}. Then, 

\begin{itemize}
	\item For any $\bt =\{\bt_{\mu}\}_{\mu \in \mathcal{I}}\in \mathbb{V}$, the Fr\'echet derivative $\mD \mathcal{f}(\bt) \colon \mathbb{V}\rightarrow \mathbb{V}^*$ of the coupled cluster function $\mathcal{f} $ at $\bt$ is the mapping with the property that for all $\bs, \bw \in \mathbb{V}$ with $\bs=\{\bs_{\nu}\}_{\nu \in \mathcal{I}}$ and $\bw=\{\bw_{\mu}\}_{\mu \in \mathcal{I}}$ it holds that
	\begin{equation}\label{eq:CC_Jac_1}
		\left \langle  \bw, \mD \mathcal{f}(\bt) \bs \right\rangle_{\mathbb{V} \times \mathbb{V}^*} = \left \langle  \sum_{\mu \in \mathcal{I}} \bw_{\mu} \mathcal{X}_\mu \Psi_0, e^{-\mathcal{T}}\comm{H}{\sum_{\nu \in \mathcal{I}} \bs_{\nu} \mathcal{X}_\nu} e^{\mathcal{T}}\Psi_0\right \rangle_{\widehat{\mathcal{H}}^1 \times \widehat{\mathcal{H}}^{-1}},
	\end{equation}
	where $[\cdot, \cdot]$ denotes the commutator and $\mathcal{T}:=\sum_{\mu \in \mathcal{I}} \bt_{\mu} \mathcal{X}_\mu$.

	\item $\mathcal{f} \colon \mathbb{V} \rightarrow \mathbb{V}^*$ is a $\mathscr{C}^{\infty}$ mapping.
\end{itemize}
\end{proposition}
\begin{proof}
We start with the proof of the first assertion. This portion of the proof will proceed in two steps:

\begin{itemize}
	\item We will obtain an expression for the Gateaux derivative $\mD\mathcal{f}(\bt), ~t\in \mathbb{V}$ of the coupled cluster function $\mathcal{f}$, and we will show that this agrees with the expression offered by Equation \eqref{eq:CC_Jac_1}.
	
	\item We will show that the Gateaux derivative is continuous as a function of $\bt$, i.e., the mapping $\bt \mapsto \mD\mathcal{f}(t) \colon \mathbb{V} \rightarrow \mathbb{V}^*$ is continuous.
\end{itemize}

Let $\bt, \bs \in \mathbb{V}$ be arbitrary. Thanks to Remark \ref{rem:CC_function}, we observe that for any $h \geq 0$ it holds that $\mathcal{f}(\bt + h \bs) \in \mathbb{V}^*$. It follows that for any $h > 0$ and any $\bw \in \mathbb{V}$ we have that 
\begin{align*}
	\langle \bw, \mathcal{f}(\bt + h\bs)- \mathcal{f}(\bt)\rangle_{\mathbb{V}\times \mathbb{V}^*} &= \left \langle \sum_{\mu \in \mathcal{I}}\bw_{\mu}\mathcal{X}_{\mu}\Psi_0, \left(e^{-\mathcal{T} -h \mathcal{S}}He^{\mathcal{T}+ h\mathcal{S}} -e^{-\mathcal{T}}He^{\mathcal{T}}\right)\Psi_0\right \rangle_{\widehat{\mathcal{H}}^1 \times \widehat{\mathcal{H}}^{-1}}\\
	&= \left \langle \sum_{\mu \in \mathcal{I}}\bw_{\mu}\mathcal{X}_{\mu}\Psi_0, e^{-\mathcal{T}}\left(e^{-h\mathcal{S}}He^{h\mathcal{S}} -H\right)e^{\mathcal{T}}\Psi_0\right \rangle_{\widehat{\mathcal{H}}^1 \times \widehat{\mathcal{H}}^{-1}},
\end{align*}
where we have denoted $\mathcal{S}:= \sum_{\mu \in \mathcal{I}} \bs_{\mu} \mathcal{X}_\mu$ and we have used the fact that $\mathcal{T}$ and $\mathcal{S}$ commute (see the first assertion of Theorem~\ref{pro:excit}).

As a consequence, using once again the commutativity of $\mathcal{T}$ and $\mathcal{S}$ together with the power series expansion of the exponential cluster operator, we deduce that
\begin{align}\nonumber
	\lim_{h \to 0}  \frac{\langle \bw, \mathcal{f}(\bt + h\bs)- \mathcal{f}(\bt)\rangle_{\mathbb{V}\times \mathbb{V}^*}}{h} &= \left \langle \sum_{\mu \in \mathcal{I}}\bw_{\mu}\mathcal{X}_{\mu}\Psi_0, e^{-\mathcal{T}}\left(-\mathcal{S}H + H\mathcal{S}\right)e^{\mathcal{T}}\Psi_0\right \rangle_{\widehat{\mathcal{H}}^1 \times \widehat{\mathcal{H}}^{-1}}\\
	&=\left \langle \sum_{\mu \in \mathcal{I}}\bw_{\mu}\mathcal{X}_{\mu}\Psi_0, e^{-\mathcal{T}}\comm{H}{\sum_{\nu \in \mathcal{I}} \bs_{\nu} \mathcal{X}_\nu}e^{\mathcal{T}}\Psi_0\right \rangle_{\widehat{\mathcal{H}}^1 \times \widehat{\mathcal{H}}^{-1}}.\label{eq:Gateaux}
\end{align}
%with the final expression being well-defined as a consequence of the boundedness properties of the electronic Hamiltonian $\mathcal{H}$ and exponential cluster operators $e^{\pm\mathcal{T}}$ (c.f., Remark \ref{rem:CC_function}).

In order to show that the expression offered by Equation \eqref{eq:Gateaux} defines the Gateaux derivative $\mD \mathcal{f}(t) \colon \mathbb{V} \rightarrow \mathbb{V}^*$, we must~show that this operator is bounded. Recalling from Definition \ref{def:V_K}, the subspace $\widetilde{\mathcal{V}} \subset \widehat{\mathcal{H}}^1$ as the orthogonal complement of $\{\Psi_0\}$, let us therefore define $\mathcal{A}(\bt) \colon \widetilde{\mathcal{V}} \rightarrow \widehat{\mathcal{H}}^{-1}$ as
\begin{align*}
	\forall {\Phi} \in \widehat{\mathcal{H}}^{1}, ~~\forall\mathcal{S}\Psi_0  \in \widetilde{\mathcal{V}}  ~ \text{ with } \mathcal{S}=& \sum_{\mu}{\bs}_{\mu}\mathcal{X}_{\mu}\colon \quad \left \langle \Phi, \mathcal{A}(\bt) \mathcal{S}\Psi_0\right \rangle_{\widehat{\mathcal{H}}^1 \times \widehat{\mathcal{H}}^{-1}}  :=\left \langle \Phi, e^{-\mathcal{T}}\comm{H}{\mathcal{S}}e^{\mathcal{T}}\Psi_0\right \rangle_{\widehat{\mathcal{H}}^1 \times \widehat{\mathcal{H}}^{-1}}.
\end{align*}

We claim that $\mathcal{A}(\bt)$ defines a bounded linear operator. Indeed, a direct calculation shows that $\forall {\Phi} \in \widehat{\mathcal{H}}^{1}$ and $\forall\mathcal{S}\Psi_0  \in \widetilde{\mathcal{V}}  ~ \text{ with } \mathcal{S}= \sum_{\mu}{\bs}_{\mu}\mathcal{X}_{\mu}$ we have
\begin{align*}
	\big \vert\left \langle \Phi , \mathcal{A}(\bt) \mathcal{S}\Psi_0\right \rangle_{\widehat{\mathcal{H}}^1 \times \widehat{\mathcal{H}}^{-1}} \big \vert & = \left\vert\left \langle \Phi, e^{-\mathcal{T}}H\mathcal{S}e^{\mathcal{T}}\Psi_0\right \rangle_{\widehat{\mathcal{H}}^1 \times \widehat{\mathcal{H}}^{-1}} - \left \langle \Phi, e^{-\mathcal{T}}\mathcal{S}He^{\mathcal{T}}\Psi_0\right \rangle_{\widehat{\mathcal{H}}^1 \times \widehat{\mathcal{H}}^{-1}}\right \vert\\[0.5em]
	&\leq \left\Vert  \Phi\right\Vert_{\widehat{\mathcal{H}}^{1}} \left\Vert  e^{-\mathcal{T}}\right\Vert_{\widehat{\mathcal{H}}^{-1}\to \widehat{\mathcal{H}}^{-1}} \left\Vert  H\right\Vert_{\widehat{\mathcal{H}}^{1}\to \widehat{\mathcal{H}}^{-1}}\left\Vert  e^{\mathcal{T}}\right\Vert_{\widehat{\mathcal{H}}^{1}\to \widehat{\mathcal{H}}^{1}}\left\Vert  \mathcal{S}\Psi_0\right\Vert_{\widehat{\mathcal{H}}^{1}} \\[0.5em] 
	&+\left\Vert  \mathcal{S}^{\dagger}\Phi\right\Vert_{\widehat{\mathcal{H}}^{1}} \left\Vert  e^{-\mathcal{T}}\right\Vert_{\widehat{\mathcal{H}}^{-1}\to \widehat{\mathcal{H}}^{-1}} \left\Vert  H\right\Vert_{\widehat{\mathcal{H}}^{1}\to \widehat{\mathcal{H}}^{-1}}\left\Vert  e^{\mathcal{T}}\Psi_0\right\Vert_{\widehat{\mathcal{H}}^{1}},
\end{align*}
where we have used the fact that the cluster operators $e^{\mathcal{T}}$ and $\mathcal{S}$ commute. Next, let us observe that by definition of the cluster operator $\mathcal{S}$ and the norm $\Vert \cdot \Vert_{\mathbb{V}}$, it holds that $\Vert \mathcal{S} \Psi_0 \Vert_{\widehat{\mathcal{H}}^{1}} = \Vert \bs \Vert_{\mathbb{V}}$. Consequently, recalling the continuity properties of cluster operators from Theorem \ref{pro:excit} we deduce that
\begin{align*}
	\left\Vert  \mathcal{S}^{\dagger}\Phi\right\Vert_{\widehat{\mathcal{H}}^{1}}  \leq \beta^{\dagger} \Vert \bs\Vert_{\mathbb{V}} \Vert \Phi\Vert_{\widehat{\mathcal{H}}^{1}}= \beta^{\dagger} \Vert \mathcal{S}\Psi_0\Vert_{\widehat{\mathcal{H}}^{1}} \Vert \mathcal{W}\Psi_0\Vert_{\widehat{\mathcal{H}}^{1}},
\end{align*}
where the constant $\beta^{\dagger} >0$ is independent of $\mathcal{S}$.

Collecting terms now shows that $\mathcal{A}(\bt) \colon \widetilde{\mathcal{V}}\rightarrow \widehat{\mathcal{H}}^{-1}$ is indeed bounded, and therefore the Gateaux derivative $\mD \mathcal{f}(\bt) \colon \mathbb{V} \rightarrow \mathbb{V}^*$ is well-defined according to the expression offered by Equation \eqref{eq:Gateaux}. Since $\bt \in \mathbb{V}$ was arbitrary, the coupled cluster function $\mathcal{f}$ is everywhere Gateaux differentiable.

%= \lim_{n \to \infty}\sup_{\substack{\bs \in \mathbb{V}\\ \Vert \bs \Vert_{\mathbb{V}} =1}} \Vert \mD \mathcal{f}(\bt)\bs  - \mD \mathcal{f}(\bt_n)\bs \Vert_{\mathbb{V}^*} 
It remains to prove that the Gateaux derivative $\mD \mathcal{f}(\bt)$ is in fact a Fr\'echet derivative. To this end, it suffices to show that the mapping $\mathbb{V}\ni \bt \mapsto \mD \mathcal{f}(\bt) \colon \mathbb{V} \rightarrow \mathbb{V}^*$ is continuous. To this end, let $\{\bt_n\}_{n \in \mathbb{N}} \subset \mathbb{V}$ be a sequence that converges to $\bt$. It follows that
\begin{align*}
	&\lim_{n \to \infty}\Vert \mD \mathcal{f}(\bt) - \mD \mathcal{f}(\bt_n)\Vert_{\mathbb{V} \to \mathbb{V}^*} = \lim_{n \to \infty} \sup_{\substack{\bs \in \mathbb{V}\\ \Vert \bs \Vert_{\mathbb{V}} =1}} \sup_{\substack{\bw \in \mathbb{V}\\ \Vert \bw \Vert_{\mathbb{V}} =1}} \vert \left \langle \bw, \mD \mathcal{f}(\bt)\bs  - \mD \mathcal{f}(\bt_n)\bs\right \rangle_{\mathbb{V}\times \mathbb{V}^*} \vert\\[0.5em]
	&=  \lim_{n \to \infty} \sup_{\substack{\bs \in \mathbb{V}\\ \Vert \bs \Vert_{\mathbb{V}} =1}} \sup_{\substack{\bw \in \mathbb{V}\\ \Vert \bw \Vert_{\mathbb{V}} =1}}\left \vert\left \langle  \sum_{\mu \in \mathcal{I}} \bw_{\mu} \mathcal{X}_\mu \Psi_0, e^{-\mathcal{T}}\comm\Big{H}{\sum_{\nu \in \mathcal{I}} \bs_{\nu} \mathcal{X}_\nu} e^{\mathcal{T}}\Psi_0-e^{-\mathcal{T}_n}\comm\Big{H}{\sum_{\nu \in \mathcal{I}} \bs_{\nu} \mathcal{X}_\nu} e^{\mathcal{T}_n}\Psi_0\right \rangle_{\widehat{\mathcal{H}}^1 \times \widehat{\mathcal{H}}^{-1}}\right \vert,
\end{align*}

where for all $n \in \mathbb{N}$ we denote  $\mathcal{T}_n:= \sum_{\mu \in \mathcal{I}} (\bt_n)_{\mu} \mathcal{X}_\mu$. Adding and subtracting suitable terms yields the inequality
\begin{align*}
	&\lim_{n \to \infty}\Vert \mD \mathcal{f}(t) - \mD \mathcal{f}(t_n)\Vert_{\mathbb{V} \to \mathbb{V}^*} \\[0.75em]
	\leq &\lim_{n \to \infty} \sup_{\substack{\bs \in \mathbb{V}\\ \Vert \bs \Vert_{\mathbb{V}} =1}} \sup_{\substack{\bw \in \mathbb{V}\\ \Vert \bw \Vert_{\mathbb{V}} =1}}\left \vert\left \langle  \sum_{\mu \in \mathcal{I}} \bw_{\mu} \mathcal{X}_\mu \Psi_0, \left(e^{-\mathcal{T}} - e^{-\mathcal{T}_n}\right)\comm{H}{\sum_{\nu \in \mathcal{I}} \bs_{\nu} \mathcal{X}_\nu} e^{\mathcal{T}}\Psi_0\right \rangle_{\widehat{\mathcal{H}}^1 \times \widehat{\mathcal{H}}^{-1}}\right \vert\\[1.5em] 
	+&\lim_{n \to \infty} \sup_{\substack{\bs \in \mathbb{V}\\ \Vert \bs \Vert_{\mathbb{V}} =1}} \sup_{\substack{\bw \in \mathbb{V}\\ \Vert \bw \Vert_{\mathbb{V}} =1}}\left \vert\left \langle  \sum_{\mu \in \mathcal{I}} \bw_{\mu} \mathcal{X}_\mu \Psi_0, e^{-\mathcal{T}_n}\comm{H}{\sum_{\nu \in \mathcal{I}} \bs_{\nu} \mathcal{X}_\nu} \left(e^{\mathcal{T}}-e^{\mathcal{T}_n}\right) \Psi_0\right \rangle_{\widehat{\mathcal{H}}^1 \times \widehat{\mathcal{H}}^{-1}}\right \vert\\[1.5em]
	\leq &\lim_{n \to \infty} \sup_{\substack{\bs \in \mathbb{V}\\ \Vert \bs \Vert_{\mathbb{V}} =1}} \left \Vert \left(e^{-\mathcal{T}} - e^{-\mathcal{T}_n}\right)\comm{H}{\sum_{\nu \in \mathcal{I}} \bs_{\nu} \mathcal{X}_\nu} e^{\mathcal{T}}\Psi_0\right \Vert_{\widehat{\mathcal{H}}^{-1} }\\
	+&\lim_{n \to \infty} \sup_{\substack{\bs \in \mathbb{V}\\ \Vert \bs \Vert_{\mathbb{V}} =1}} \left \Vert e^{-\mathcal{T}_n}\comm{H}{\sum_{\nu \in \mathcal{I}} \bs_{\nu} \mathcal{X}_\nu} \left(e^{\mathcal{T}}-e^{\mathcal{T}_n}\right) \Psi_0\right \Vert_{\widehat{\mathcal{H}}^{-1}},
\end{align*}
where we have used the fact that $\Vert \bw \Vert_{\mathbb{V}}= \left\Vert \sum_{\mu \in \mathcal{I}} \bw_{\mu} \mathcal{X}_{\mu} \Psi_0\right\Vert_{\widehat{\mathcal{H}}^1}$ by definition.

We can now use the fact that the exponential cluster operator is a local $\mathscr{C}^{\infty}$ diffeomorphism on the algebra of cluster operators (see Theorem \ref{pro:excit}) together with the boundedness properties of the Hamiltonian $H$ and excitation operators to deduce that both of the above limits are zero. Thus, $\lim_{n \to \infty}\Vert \mD \mathcal{f}(\bt) - \mD \mathcal{f}(\bt_n)\Vert_{\mathbb{V} \to \mathbb{V}^*}= 0,$
which shows that $\mD \mathcal{f}(t) \colon \mathbb{V} \rightarrow \mathbb{V}^*$ as defined through Equation \eqref{eq:CC_Jac_1} is indeed the Fr\'echet derivative of the coupled cluster function $\mathcal{f} \colon \mathbb{V} \rightarrow \mathbb{V}^*$ at $\bt \in \mathbb{V}$.

In order to complete the proof of this proposition, we must demonstrate that the second assertion also holds, namely, that $\mathcal{f}$ is a $\mathscr{C}^{\infty}$ mapping from $\mathbb{V}$ to $\mathbb{V}^*$. To this end, it is sufficient to observe that higher order Gateaux derivatives of the coupled cluster function can be computed exactly as the first order Gateaux derivative given by Equation~\eqref{eq:Gateaux} with the single commutator being replaced by nested commutators. The fact that these Gateaux derivatives are also Fr\'echet derivatives is deduced in an identical fashion by making use of the fact that exponential cluster operator is a local $\mathscr{C}^{\infty}$ diffeomorphism. This completes the proof.
\end{proof}

Proposition \ref{prop:CC_der} has a number of important consequences that we now state. For the first result, let us recall from Theorem \eqref{thm:Yvon} that every zero $\bt^* \in \mathbb{V}$ of the coupled cluster function is associated with an intermediately normalised eigenfunction $\Psi^* \in \widehat{\mathcal{H}}^1$ of the electronic Hamiltonian $H \colon \widehat{\mathcal{H}}^1 \rightarrow \widehat{\mathcal{H}}^{-1}$ defined through Equation \eqref{eq:Hamiltonian}. %and that the corresponding eigenvalue $\mathcal{E}^*$ and CC energy $\mathcal{E}_{\rm CC}^*$ defined through Equation \eqref{eq:CC_Energy} coincide.

\begin{corollary}[Coupled Cluster Fr\'echet Derivative at Zeros of the Coupled Cluster Function]\label{cor:CC_der}~

Let the excitation index set $\mathcal{I}$ be defined through Definition \ref{def:Excitation_Index}, let the excitation operators $\{\mathcal{X}_{\mu}\}_{\mu \in \mathcal{I}}$ be defined through Definition \ref{def:Excitation_Operator}, let the coupled cluster function $\mathcal{f} \colon \mathbb{V} \rightarrow \mathbb{V} $ be defined through Definition \ref{def:CC_function}, for any $\bt \in \mathbb{V}$ let $\mD \mathcal{f}(\bt)$ denote the Fr\'echet derivative of the coupled cluster function as defined through Equation \eqref{eq:CC_Jac_1}, let $\bt^*= \{\bt^*_{\mu}\}_{\mu \in \mathcal{I}} \in \mathbb{V}$ denote a zero of the CC function that generates the intermediately normalised eigenfunction $\Psi^* \in \widehat{\mathcal{H}}^1$ of the electronic Hamiltonian with corresponding eigenvalue $\mathcal{E}^*$. Then for all $\bs, \bw \in \mathbb{V} $ with $\bs =\{\bs_{\nu}\}_{\nu \in \mathcal{I}}$ and $ \bw=\{\bw_{\mu}\}_{\mu \in \mathcal{I}}$, it holds that
\begin{equation}\label{eq:CC_Jac_1prime}
	\left \langle  \bw, \mD \mathcal{f}(\bt^*) \bs \right\rangle_{\mathbb{V} \times \mathbb{V}^*} = \left \langle  \sum_{\mu \in \mathcal{I}} \bw_{\mu} \mathcal{X}_\mu \Psi_0, e^{-\mathcal{T}^*} \left(H - \mathcal{E}^*\right) e^{\mathcal{T}^*} \sum_{\nu \in \mathcal{I}} \bs_{\nu} \mathcal{X}_\nu\Psi_0\right \rangle_{\widehat{\mathcal{H}}^1 \times \widehat{\mathcal{H}}^{-1}} ~ \text{ where } ~ \mathcal{T}^*:=\sum_{\mu \in \mathcal{I}} \bt^*_{\mu} \mathcal{X}_\mu.
\end{equation}
\end{corollary}
\begin{proof}
The proof follows by a direct calculation from Equation \eqref{eq:CC_Jac_1} by expanding the commutator, making use of the fact that $H\Psi^* = \mathcal{E}^* \Psi^*= \mathcal{E}^* e^{\mathcal{T}^*}\Psi_0$ by definition together with the commutativity of the cluster operators $\mathcal{T}^*$ and $\mathcal{S}= \sum_{\mu} \bs_{\mu}\mathcal{X}_{\mu}$.
\end{proof}

Consider the setting of Corollary \ref{cor:CC_der}. Let us remark here that, thanks to Theorem \ref{thm:Yvon}, the eigenvalue $\mathcal{E}^*$ which appears in \eqref{eq:CC_Jac_1prime} coincides with the CC energy $\mathcal{E}^*_{\rm CC}$ generated by $\bt^*$ through Equation \eqref{eq:CC_Energy}. Therefore, when considering expressions of the form \eqref{eq:CC_Jac_1prime} involving the CC Fr\'echet derivative, we may refer to $\mathcal{E}^*$ as simply the coupled cluster energy associated with $\bt^*$ without reference to the underlying eigenpair of the electronic Hamiltonian.

%Let us recall from Theorem \eqref{thm:Yvon} that every intermediately normalised eigenfunction $\Psi^*$ is associated with a zero $\bt^* \in \mathbb{V}$ and that the corresponding eigenvalue $\mathcal{E}^*$ and CC energy $\mathcal{E}_{\rm CC}^*$ defined through Equation \eqref{eq:CC_Energy} coincide. Consequently, the coupled cluster Fr\'echet derivative will have the form \eqref{eq:CC_Jac_1prime} at \emph{any} zero $\bt^* \in \mathbb{V}$ of the coupled cluster function, and in this case, $\mathcal{E}^*=\mathcal{E}_{\rm CC}^*$.

\begin{corollary}[Local Lipschitz Continuity of Coupled Cluster Fr\'echet Derivative]\label{cor:CC_der_Lipschitz}~

Let the coupled cluster function $\mathcal{f} \colon \mathbb{V} \rightarrow \mathbb{V} $ be defined through Definition \ref{def:CC_function}, and for any $\bt \in \mathbb{V}$ let $\mD \mathcal{f}(\bt)$ denote the Fr\'echet derivative of the coupled cluster function as defined through Equation \eqref{eq:CC_Jac_1}. Then the mapping $\mathbb{V} \ni \bt \mapsto \mD \mathcal{f}(\bt)  \colon \mathbb{V} \rightarrow \mathbb{V}^*$ is Lipschitz continuous for bounded arguments, i.e., for any~$\bt \in \mathbb{V}$ and any any $\delta > 0$, there exists a constant $\mL_{\bt}(\delta) >0$ such that
\begin{align*}
	\sup_{\bt \neq \bs \in {\mB_{\delta}(\bt)}}\frac{\Vert \mD \mathcal{f}(\bt) - \mD \mathcal{f}(\bs)\Vert_{\mathbb{V} \rightarrow \mathbb{V}^*}}{\Vert \bt - \bs \Vert_{\mathbb{V}}} := \mL_{\bt}(\delta)  < \infty. 
\end{align*}
\end{corollary}
Corollary \ref{cor:CC_der_Lipschitz} follows immediately from the regularity of the coupled cluster function.

Having obtained an expression for the first Fr\'echet derivative $\mD \mathcal{f}(\bt), ~\bt \in \mathbb{V}$ of the coupled cluster function and studied some regularity properties of the mapping $\mathbb{V} \ni \bt \mapsto \mD \mathcal{f}(\bt)$, the next step in our analysis will be to study the invertibility of the Fr\'echet derivative $\mD \mathcal{f}$ at any zero $\bt^* \in \mathbb{V}$ of the coupled cluster function. In order to proceed with this analysis, let us first notice that thanks to the expression offered by Equation \eqref{eq:CC_Jac_1prime} in Corollary \ref{cor:CC_der}, the coupled cluster Fr\'echet derivative at any zero $\bt^*\in \mathbb{V}$ can be described in terms of an operator acting on a subspace of the infinite-dimensional $N$-particle space $\widehat{\mathcal{H}}^1$. This observation motivates us to introduce the following operator acting on the space $\widetilde{\mathcal{V}}=\{\Psi_0\}^{\perp} \subset \widehat{\mathcal{H}}^1$ (recall Definition \ref{def:V_K}).

\begin{definition}[Operator Induced by Coupled Cluster Fr\'echet Derivative at $\bt^*$]\label{def:A_op}~

Let the excitation index set $\mathcal{I}$ be defined through Definition \ref{def:Excitation_Index}, let the excitation operators $\{\mathcal{X}_{\mu}\}_{\mu \in \mathcal{I}}$ be defined through Definition~\ref{def:Excitation_Operator}, let $\bt^*=\{\bt^*_{\mu}\}_{\mu \in \mathcal{I}} \in \mathbb{V}$ be any zero of the coupled cluster function defined through Definition~\ref{def:CC_function}, let $\mathcal{E}^*$ be the associated coupled cluster energy calculated through \eqref{eq:CC_Energy}, and let the space $\widetilde{\mathcal{V}} \subset \widehat{\mathcal{H}}^1$ be defined as in Definition~\ref{def:V_K}. We define the operator $\mathcal{A}(t^*) \colon\widetilde{\mathcal{V}}\rightarrow \widehat{\mathcal{H}}^{-1}$ as the mapping with the property that
\begin{equation}\label{eq:CC_Jac_3}
	\forall \widetilde{\Psi}\in \widetilde{\mathcal{V}}\colon \qquad \mathcal{A}(\bt^*) \widetilde{\Psi}:= e^{-\mathcal{T}^*}\left(H- \mathcal{E}^*\right) e^{\mathcal{T}^*}\widetilde{\Psi}\qquad \text{where} \quad \mathcal{T}^*= \sum_{\mu \in \mathcal{I}} \bt_{\mu}^* \mathcal{X}_{\mu}.
\end{equation}
\end{definition}

\begin{Notation}\label{not:1}~
Let $\bt^* \in \mathbb{V}$ be any zero of the coupled cluster function defined through Definition \ref{def:CC_function} and let $\mathcal{E}^*$ be the associated coupled cluster energy calculated through Equation \eqref{eq:CC_Energy}. 

\begin{itemize}
	\item We denote by $\alpha_{\bt^*} > 0$ the constant defined as 
	\begin{align*}
		\alpha_{\bt^*}:= \Vert \mathcal{A}(\bt^*) \Vert_{\widetilde{\mathcal{V}}\to \widetilde{\mathcal{V}}^*}:= \sup_{0\neq \widetilde{\Phi} \in \widetilde{\mathcal{V}}}\; \sup_{0\neq \widetilde{\Psi} \in \widetilde{\mathcal{V}}} \frac{\langle\widetilde{\Phi}, \mathcal{A}(\bt^*) \widetilde{\Psi} \rangle_{\widehat{\mathcal{H}}^{-1} \times \widehat{\mathcal{H}}^{1}} }{\Vert \widetilde{\Phi}\Vert_{\widehat{\mathcal{H}}^{1}} \Vert\widetilde{\Psi}\Vert_{\widehat{\mathcal{H}}^{1}}}, 
	\end{align*}
	with the existence of $\alpha_{\bt^*}$ being guaranteed by Proposition \ref{prop:CC_der}.
	
	\item For any $\bt \in \mathbb{V}$ we denote by $\mL_{\bt} \colon \mathbb{R}_+ \rightarrow \mathbb{R}_+$ the so-called `Lipschitz continuity function' as the mapping with the property that for all $\delta >0$ it holds that
	\begin{align*}
		\forall \delta > 0 \colon \quad \mL_{\bt}(\delta):= \sup_{\bt \neq \bs \in {\mB_{\delta}(\bt)}}\frac{\Vert \mD \mathcal{f}(\bt) - \mD \mathcal{f}(\bs)\Vert_{\mathbb{V} \rightarrow \mathbb{V}^*}}{\Vert \bt - \bs \Vert_{\mathbb{V}}},
	\end{align*}
	with the existence of the function $\mL_{\bt}$ being guaranteed by Corollary \ref{cor:CC_der_Lipschitz}.
	
	\item We denote by $\mathcal{A}(\bt^*)^{\dagger} \colon \widetilde{\mathcal{V}} \rightarrow \widehat{\mathcal{H}}^{-1}$ the mapping with the property that with the property that for all $\widetilde{\Psi} \in \widetilde{\mathcal{V}}$ it holds that
	\begin{equation}\label{eq:CC_Jac_adjoint}
		\mathcal{A}(\bt^*)^\dagger \widetilde{\Psi}:=	e^{(\mathcal{T}^*)^{\dagger}}\left(H- \mathcal{E}^*\right) e^{-(\mathcal{T}^*)^{\dagger}}\widetilde{\Psi},
	\end{equation}
	
	and we emphasise that for all $\widetilde{\Psi}, \widetilde{\Phi} \in \widetilde{\mathcal{V}}$ it holds that
	\begin{align*}
		\langle\widetilde{\Phi}, \mathcal{A}(\bt^*)^{\dagger} \widetilde{\Psi} \rangle_{\widehat{\mathcal{H}}^{1} \times \widehat{\mathcal{H}}^{-1}}= \langle \mathcal{A}(\bt^*)\widetilde{\Phi}, \widetilde{\Psi} \rangle_{\widehat{\mathcal{H}}^{-1} \times \widehat{\mathcal{H}}^{1}},
	\end{align*}
	
	so that in particular
	\begin{align*}
		\Vert \mathcal{A}(\bt^*)^{\dagger} \Vert_{\widetilde{\mathcal{V}}\to \widetilde{\mathcal{V}}^*} = \Vert \mathcal{A}(\bt^*) \Vert_{\widetilde{\mathcal{V}}\to \widetilde{\mathcal{V}}^*}= \alpha_{\bt^*}.
	\end{align*}
	
\end{itemize}

\end{Notation}

Consider now Definition \ref{def:A_op} of the bounded linear operator $\mathcal{A}(\bt^*) \colon \widetilde{\mathcal{V}} \rightarrow \widehat{\mathcal{H}}^{-1}$ for an arbitrary zero $\bt^*\in \mathbb{V}$ of the coupled cluster function. Since the coefficient space $\mathbb{V}$ inherits its inner product from the inner product on $\widehat{\mathcal{H}}^1$, it immediately follows that
\begin{align*}
\mD \mathcal{f}(\bt^*)\colon \mathbb{V} \rightarrow \mathbb{V}^* \quad \text{ is an isomorphism } \qquad \iff \qquad \mathcal{A}(\bt^*) \colon \widetilde{\mathcal{V}} \rightarrow \widetilde{\mathcal{V}}^* \quad \text{is an isomorphism}. 
\end{align*}

We claim that the mapping $\mathcal{A}(\bt^*) \colon \widetilde{\mathcal{V}} \rightarrow \widetilde{\mathcal{V}}^*$ can indeed be shown to be an isomorphism \emph{provided} that the zero $\bt^*$ is generated by an \emph{intermediately normalisable eigenfunction} $\Psi^* \in \widehat{\mathcal{H}}^1$ of the electronic Hamiltonian that corresponds to a \emph{simple and isolated eigenvalue}. The proof of this claim, which is the subject of the next theorem, is based on classical functional analysis arguments, and will proceed in the following steps: Assuming that the zero $\bt^*$ is generated by a non-degenerate, intermediately normalisable eigenfunction of the electronic Hamiltonian:
\begin{enumerate}
\item We will first show that $\mathcal{A}(\bt^*) \colon \widetilde{\mathcal{V}} \rightarrow \widetilde{\mathcal{V}}^*$ is injective. As a consequence of the Hahn-Banach theorem, we will deduce that the adjoint operator $\mathcal{A}(\bt^*)^{\dagger} \colon \widetilde{\mathcal{V}} \rightarrow \widetilde{\mathcal{V}}^*$ has \emph{dense} range.

\item Next, we will show that the operator $\mathcal{A}(\bt^*)^{\dagger} \colon \widetilde{\mathcal{V}} \rightarrow \widetilde{\mathcal{V}}^*$ is bounded below. This will imply that $\mathcal{A}(\bt^*)^{\dagger} \colon \widetilde{\mathcal{V}} \rightarrow \widetilde{\mathcal{V}}^*$ is injective, and has \emph{closed} range. 
\end{enumerate}

Combining the above two steps, will allow us to deduce that the adjoint operator $\mathcal{A}(\bt^*)^{\dagger} \colon \widetilde{\mathcal{V}} \rightarrow \widetilde{\mathcal{V}}^*$ is an isomorphism, and therefore so too is the operator $\mathcal{A}(\bt^*)\colon \widetilde{\mathcal{V}} \rightarrow \widetilde{\mathcal{V}}^*$. Let us emphasise here that rather than attacking directly the operator $\mathcal{A}(\bt^*)$ induced by the Fr\'echet derivative of the coupled cluster function at $\bt^* \in \mathbb{V}$, we are choosing to analyse its adjoint. This choice is motivated by practical reasons: there is a technical difficulty in proving directly the invertibility of $\mathcal{A}(\bt^*)$ which is avoided if we study instead $\mathcal{A}(\bt^*)^{\dagger}$.

\begin{theorem}[Invertibility of Operator Induced by Coupled Cluster Fr\'echet Derivative at $\bt^*$]\label{thm:CC}~

Let $\bt^* =\{\bt^*_{\mu}\}_{\mu \in \mathcal{I}}\in \mathbb{V}$ be associated with a non-degenerate, intermediately normalisable eigenpair $(\mathcal{E}^*, \Psi^*) \in \mathbb{R} \times \widehat{\mathcal{H}}^1$ of the electronic Hamiltonian $H \colon \widehat{\mathcal{H}}^1 \rightarrow \widehat{\mathcal{H}}^{-1}$ defined through Equation \eqref{eq:Hamiltonian}, i.e.,
\begin{align*}
	H\Psi^* = \mathcal{E}^* \Psi^*, \qquad \text{with }~ \mathcal{E}^* ~\text{ simple, isolated} \qquad \text{and} \qquad \Psi^* = e^{\mathcal{T}^*}\Psi_0 \quad \text{where} \quad \mathcal{T}^* = \sum_{\mu \in \mathcal{I}}\bt_{\mu}^*\mathcal{X}_{\mu}.
\end{align*} 
Then the operator $\mathcal{A}(\bt^*) \colon \widetilde{\mathcal{V}} \rightarrow \widetilde{\mathcal{V}}^*$ defined through Definition \ref{def:A_op} is an isomorphism. 
\end{theorem}
\begin{proof}
The proof follows the aforementioned two steps. We begin with the injectivity of $\mathcal{A}(\bt^*)$. 

\vspace{0.75cm}

\begin{mdframed}
	\textbf{Step 1:} $\mathcal{A}(\bt^*) \colon \widetilde{\mathcal{V}} \rightarrow \widetilde{\mathcal{V}}^*$ is injective.
\end{mdframed}

Suppose there exists $0 \neq \widetilde{\Psi} \in \widetilde{\mathcal{V}}$ such that $\mathcal{A}(\bt^*)\widetilde{\Psi}\equiv 0$ in $\widetilde{\mathcal{V}}^*$, i.e., for all $\widetilde{\Phi} \in \widetilde{\mathcal{V}}$ it holds that
\begin{align}\label{eq:injective_1}
	\left \langle \widetilde{\Phi}, \mathcal{A}(\bt^*) \widetilde{\Psi} \right\rangle_{\widehat{\mathcal{H}}^1 \times \widehat{\mathcal{H}}^{-1}}=0.% =\left\langle \widetilde{\Phi}, e^{-\mathcal{T}^*} \left(H- \mathcal{E}^*\right)e^{\mathcal{T}^*}\widetilde{\Psi} \right \rangle_{\widehat{\mathcal{H}}^1 \times \widehat{\mathcal{H}}^{-1}}.
\end{align}

As a first step, we claim that from Equation \eqref{eq:injective_1} it must follow that $\mathcal{A}(\bt^*)\widetilde{\Psi}\equiv 0$ in $\widehat{\mathcal{H}}^{-1}$. Recalling the complementary decomposition of $\widehat{\mathcal{H}}^1$ given by Definition \ref{lem:comp}, we see that it suffices to prove that
\begin{align}\label{eq:injective_2}
 	\left \langle\Psi_0, \mathcal{A}(\bt^*) \widetilde{\Psi} \right\rangle_{\widehat{\mathcal{H}}^1 \times \widehat{\mathcal{H}}^{-1}}=0.
 \end{align}
 
Consider now the element $\widehat{\Phi} = e^{(\mathcal{T}^*)^{\dagger}}e^{\mathcal{T}^*} \Psi_0 \in \widehat{\mathcal{H}}^1$ and recall that we denote by $\mathbb{P}_0 \colon \widehat{\mathcal{H}}^1 \rightarrow \widehat{\mathcal{H}}^1$ the $\widehat{\mathcal{L}}^2$-orthogonal projection operator onto $\text{span} \{\Psi_0\}$ defined through Definition~\ref{lem:comp} and we have defined $\mathbb{P}^{\perp}_0:= \mathbb{I}-\mathbb{P}_0$. Clearly, we have that $\mathbb{P}_0 \widehat{\Phi} \neq 0$ since 
\begin{align}\label{eq:injective_3}
(\widehat{\Phi}, \Psi_0)_{\widehat{\mathcal{L}}^2}= \left(e^{(\mathcal{T}^*)^{\dagger}}e^{\mathcal{T}^*} \Psi_0, \Psi_0\right)_{\widehat{\mathcal{L}}^2}= \left(e^{\mathcal{T}^*} \Psi_0,  e^{\mathcal{T}^*} \Psi_0\right)_{\widehat{\mathcal{L}}^2}=\Vert \Psi^*\Vert_{\widehat{\mathcal{L}}^2}^2=:\widehat{c}_0 \neq 0.
\end{align}

Since $\Psi^*= e^{\mathcal{T}^*} \Psi_0\in \widehat{\mathcal{H}}^1$ is by definition an eigenfunction of the electronic Hamiltonian with associated eigenvalue $\mathcal{E}^*$, a direct calculation also reveals that
\begin{align}\nonumber 
  \left \langle\widehat{\Phi}, \mathcal{A}(\bt^*) \widetilde{\Psi} \right\rangle_{\widehat{\mathcal{H}}^1 \times \widehat{\mathcal{H}}^{-1}}&= \left\langle  e^{(\mathcal{T}^*)^{\dagger}}e^{\mathcal{T}^*} \Psi_0, e^{-\mathcal{T}^*}\left(H- \mathcal{E}^*\right)e^{\mathcal{T}^*}\widetilde{\Psi}\right\rangle_{\widehat{\mathcal{H}}^1 \times \widehat{\mathcal{H}}^{-1}}\\  \label{eq:injective_4} &= \left\langle  e^{\mathcal{T}^*} \Psi_0,  \left(H- \mathcal{E}^*\right)e^{\mathcal{T}^*}\widetilde{\Psi}\right\rangle_{\widehat{\mathcal{H}}^1 \times \widehat{\mathcal{H}}^{-1}}\\
  &=0. \nonumber
\end{align}

On the other hand, we also have 
\begin{align}\label{eq:injective_5}
  \left \langle\widehat{\Phi}, \mathcal{A}(\bt^*) \widetilde{\Psi} \right\rangle_{\widehat{\mathcal{H}}^1 \times \widehat{\mathcal{H}}^{-1}}=  \left \langle\mathbb{P}_0\widehat{\Phi}, \mathcal{A}(\bt^*) \widetilde{\Psi} \right\rangle_{\widehat{\mathcal{H}}^1 \times \widehat{\mathcal{H}}^{-1}}+\left \langle\mathbb{P}_0^{\perp}\widehat{\Phi}, \mathcal{A}(\bt^*) \widetilde{\Psi} \right\rangle_{\widehat{\mathcal{H}}^1 \times \widehat{\mathcal{H}}^{-1}}=\left \langle\mathbb{P}_0\widehat{\Phi}, \mathcal{A}(\bt^*) \widetilde{\Psi} \right\rangle_{\widehat{\mathcal{H}}^1 \times \widehat{\mathcal{H}}^{-1}},
\end{align} 
where the second equality is due to the fact that $\mathcal{A}(\bt^*)\widetilde{\Psi}\equiv 0$ in $\widetilde{\mathcal{V}}^*$ by assumption.

Combining therefore Equations \eqref{eq:injective_3}-\eqref{eq:injective_5}, we deduce that
\begin{align*}
0=\left \langle\widehat{\Phi}, \mathcal{A}(\bt^*) \widetilde{\Psi} \right\rangle_{\widehat{\mathcal{H}}^1 \times \widehat{\mathcal{H}}^{-1}}=\left \langle\mathbb{P}_0\widehat{\Phi}, \mathcal{A}(\bt^*) \widetilde{\Psi} \right\rangle_{\widehat{\mathcal{H}}^1 \times \widehat{\mathcal{H}}^{-1}}= \widehat{c}_0\left \langle\Psi_0, \mathcal{A}(\bt^*) \widetilde{\Psi} \right\rangle_{\widehat{\mathcal{H}}^1 \times \widehat{\mathcal{H}}^{-1}}.
\end{align*}
Since $\widehat{c}_0\neq 0$, we immediately deduce that Equation \eqref{eq:injective_2} holds and therefore $\mathcal{A}(\bt^*)\widetilde{\Psi}\equiv 0$ in $\widehat{\mathcal{H}}^{-1}$ as claimed.

Since $e^{-(\mathcal{T}^*)^{\dagger}} \colon \widehat{\mathcal{H}}^{1} \rightarrow \widehat{\mathcal{H}}^{1} $ is a bijection, we next deduce that for all $\Phi \in \widehat{\mathcal{H}}^1$ it holds that
\begin{align*}
	0= \left \langle e^{(\mathcal{T}^*)^{\dagger}}\Phi, \mathcal{A}(\bt^*) \widetilde{\Psi} \right\rangle_{\widehat{\mathcal{H}}^1 \times \widehat{\mathcal{H}}^{-1}}=\left\langle{\Phi}, \left(H- \mathcal{E}^*\right)e^{\mathcal{T}^*}\widetilde{\Psi}\right\rangle_{\widehat{\mathcal{H}}^1 \times \widehat{\mathcal{H}}^{-1}}.
\end{align*}

The simplicity of the eigenvalue $\mathcal{E}^*$ now implies that we must have 
\begin{align*}
	e^{\mathcal{T}^*}\widetilde{\Psi} \in ~\text{ span}\{\Psi^{*}\}.
\end{align*}

Using again the fact that $\Psi^* = e^{\mathcal{T}^*} \Psi_0$ and that $e^{\mathcal{T}^*} \colon \widehat{\mathcal{H}}^{1} \rightarrow \widehat{\mathcal{H}}^{1} $ is a bijection, we obtain the existence of some constant $\widetilde{c}_{0} \in \mathbb{R}$ such that
\begin{align*}
	\widetilde{\Psi}= \widetilde{c}_{0}\Psi_0.
\end{align*}

Recall however that $\widetilde{\Psi} \in \widetilde{\mathcal{V}}= \{\Psi_0\}^{\perp}$ by assumption, and therefore we must have $\widetilde{c}_{0}=0$ and thus $\widetilde{\Psi}=0$. This completes the proof of the first step.

\vspace{0.75cm}

\begin{mdframed}
	\textbf{Step 2:} $\mathcal{A}(\bt^*)^{\dagger} \colon \widetilde{\mathcal{V}} \rightarrow \widetilde{\mathcal{V}}^*$ is bounded below.
\end{mdframed}

Let $\widetilde{\Psi}\in\widetilde{\mathcal{V}}$ be arbitrary. For any $\Psi^*_{\perp} \in \{\Psi^{*}\}^{\perp} \subset \widehat{\mathcal{H}}^1$, i.e., any wave-function $\Psi^*_{\perp}$ that is $\widehat{\mathcal{L}}^2$-orthogonal to the eigenfunction $\Psi^*$ with associated eigenvalue $\mathcal{E}^*$, we define the function
\begin{align*}
	\widetilde{\Phi}_{\Psi_{\perp}}:= \mathbb{P}_0^{\perp} e^{-\mathcal{T}^*} \Psi^*_{\perp}\in \widetilde{\mathcal{V}},
\end{align*}
It is straightforward to observe that for all such $\Phi_{\Psi^{\perp}}$, it holds that
\begin{equation}\label{eq:bounded_below_0}
	\begin{split}
		\left\vert \left\langle  \Phi_{\Psi_{\perp}}, \mathcal{A}(\bt^*)^{\dagger} \widetilde{\Psi}\right\rangle_{\widehat{\mathcal{H}}^{1} \times \widehat{\mathcal{H}}^{-1}}\right\vert = \Big \vert &\underbrace{\left\langle e^{-\mathcal{T}^*} \Psi^*_{\perp}, e^{(\mathcal{T}^*)^{\dagger}}  \left(H- \mathcal{E}^*\right) e^{-(\mathcal{T}^*)^{\dagger}} \widetilde{\Psi}\right\rangle_{\widehat{\mathcal{H}}^{1} \times \widehat{\mathcal{H}}^{-1}}}_{:= \rm (I)}\\ 
		-&\underbrace{\left\langle \mathbb{P}_0 e^{-\mathcal{T}^*} \Psi^*_{\perp}, e^{(\mathcal{T}^*)^{\dagger}}  \left(H- \mathcal{E}^*\right) e^{-(\mathcal{T}^*)^{\dagger}} \widetilde{\Psi}\right\rangle_{\widehat{\mathcal{H}}^{1} \times \widehat{\mathcal{H}}^{-1}}}_{:= (\rm II)} \Big \vert.
	\end{split}
\end{equation}

We claim that the term (II) is identically zero for any choice of $\Psi^*_{\perp}$. To this end, observe that
\begin{align*}
	\mathbb{P}_0 e^{-\mathcal{T}^*} \Psi^*_{\perp}= \left(e^{-\mathcal{T}^*} \Psi^*_{\perp}, \Psi_0\right )_{\widehat{\mathcal{L}}^2}\Psi_0=  \left(\Psi^*_{\perp}, \Psi_0\right )_{\widehat{\mathcal{L}}^2}\Psi_0= \mathbb{P}_0\Psi^*_{\perp}.
\end{align*}

We therefore deduce that
\begin{align*}
	(\rm II)&=-\left\langle \mathbb{P}_0  \Psi^*_{\perp}, e^{(\mathcal{T}^*)^{\dagger}}  \left(H- \mathcal{E}^*\right) e^{-(\mathcal{T}^*)^{\dagger}} \widetilde{\Psi}\right\rangle_{\widehat{\mathcal{H}}^{1} \times \widehat{\mathcal{H}}^{-1}}\\
	&= -\left(\Psi_0, \Psi^*_{\perp}  \right)_{\widehat{\mathcal{L}}^{2} }\; \left\langle \Psi_0, e^{(\mathcal{T}^*)^{\dagger}}  \left(H- \mathcal{E}^*\right) e^{-(\mathcal{T}^*)^{\dagger}} \widetilde{\Psi} \right\rangle_{\widehat{\mathcal{H}}^{1} \times \widehat{\mathcal{H}}^{-1}},
\end{align*}
where we have used the fact that $\mathbb{P}_0 \Psi^*_{\perp} =\left(\Psi_0, \Psi^*_{\perp} \right)_{\widehat{\mathcal{L}}^{2} } \Psi_0$.

Notice however that the second term in the product above satisfies
\begin{align}\label{eq:bounded_below_1}
	\left\langle \Psi_0, e^{(\mathcal{T}^*)^{\dagger}}  \left(H- \mathcal{E}^*\right) e^{-(\mathcal{T}^*)^{\dagger}} \widetilde{\Psi}\right\rangle_{\widehat{\mathcal{H}}^{1} \times \widehat{\mathcal{H}}^{-1}}  = \left\langle  \left(H- \mathcal{E}^*\right) e^{\mathcal{T}^*} \Psi_0,  e^{-(\mathcal{T}^*)^{\dagger}}\widetilde{\Psi}\right\rangle_{\widehat{\mathcal{H}}^{-1} \times \widehat{\mathcal{H}}^{1}} =0,
\end{align}
where the last step follows from the fact that $He^{\mathcal{T}^*} \Psi_0 = H\Psi^*= \mathcal{E}^*\Psi^*$ by assumption. Thus, the term (II) is identically zero for any choice of $\Psi^*_{\perp} \in\{\Psi^{*}\}^{\perp} \subset \widehat{\mathcal{H}}^1$ as claimed, and we need only estimate the term (I).

An easy simplification reveals that 
\begin{align}\label{eq:bounded_below_2}
	(\rm I)=\left\langle \Psi^*_{\perp}, \left(H- \mathcal{E}^*\right) e^{-(\mathcal{T}^*)^{\dagger}} \widetilde{\Psi}\right\rangle_{\widehat{\mathcal{H}}^{1} \times \widehat{\mathcal{H}}^{-1}}.
\end{align}

Thanks to the ellipticity of the electronic Hamiltonian $H \colon \widehat{\mathcal{H}}^1 \rightarrow \widehat{\mathcal{H}}^{-1}$ and the simplicity of the eigenvalue $\mathcal{E}^*$, it is easy to deduce that the shifted Hamiltonian $H - \mathcal{E}^* \colon \widehat{\mathcal{H}}^1 \rightarrow \widehat{\mathcal{H}}^{-1}$ satisfies an inf-sup condition on $\{\Psi^{*}\}^{\perp} \subset \widehat{\mathcal{H}}^1$ (see also Remark \ref{rem:interp_const} for a detailed argument). In order to make use of this result and bound the term (I), we need only show that $\Psi^*_{\perp}$ and $ e^{-(\mathcal{T}^*)^{\dagger}} \widetilde{\Psi}$ are both elements of $\{\Psi^*\}^{\perp}$. The former inclusion is true by definition of $\Psi^*_{\perp}$ and as for latter, we see that
\begin{align*}
	\left(\Psi^*, e^{-(\mathcal{T}^*)^{\dagger}} \widetilde{\Psi}\right)_{\widehat{\mathcal{L}}^2}   = \left(e^{\mathcal{T}^*}\Psi_0, e^{-(\mathcal{T}^*)^{\dagger}} \widetilde{\Psi}\right)_{\widehat{\mathcal{L}}^2}= \left(\Psi_0,\widetilde{\Psi}\right)_{\widehat{\mathcal{L}}^2}= 0,
\end{align*}
where we have used the fact that $\widetilde{\Psi} \in \widetilde{\mathcal{V}}= \{\Psi_0\}^{\perp}$ by definition.

We can therefore deduce from Equation~\eqref{eq:bounded_below_2} that
\begin{align}\label{eq:bounded_below_3}
	\sup_{\Psi^*_{\perp} \in \{\Psi^{*}\}^{\perp}} \frac{\left\vert \left\langle \Psi^*_{\perp}, \left(H-\mathcal{E}^*\right)e^{-(\mathcal{T}^*)^{\dagger}} \widetilde{\Psi} \right\rangle_{\widehat{\mathcal{H}}^1 \times \widehat{\mathcal{H}}^{-1}} \right\vert }{\Vert \Psi^*_{\perp}\Vert_{\widehat{\mathcal{H}}^1 } }\geq \gamma \left\Vert e^{-(\mathcal{T}^*)^{\dagger}} \widetilde{\Psi}\right\Vert_{\widehat{\mathcal{H}}^{1}},
\end{align}
where $\gamma> 0$ is the inf-sup constant of the shifted Hamiltonian $H-\mathcal{E}^*$ on $\{\Psi^{*}\}^{\perp} \subset \widehat{\mathcal{H}}^1$. 

Recalling now that $\widetilde{\Psi}\in \widetilde{\mathcal{V}}$ was arbitrary and combining the estimates \eqref{eq:bounded_below_1}-\eqref{eq:bounded_below_3} with Equation \eqref{eq:bounded_below_0} we obtain that for all $\widetilde{\Psi} \in \widetilde{\mathcal{V}}$ it holds that
\begin{align*}
	\Vert \mathcal{A}(\bt^*)^{\dagger} \widetilde{\Psi}\Vert_{\widetilde{\mathcal{V}}^*}&= \sup_{0\neq \widetilde{\Phi} \in \widetilde{\mathcal{V}}} \frac{\big \vert \left \langle \widetilde{\Phi}, \mathcal{A}(\bt^*)^{\dagger}\widetilde{\Psi}\right \rangle_{\widehat{\mathcal{H}}^{1} \times \widehat{\mathcal{H}}^{-1}} \big \vert}{\Vert \widetilde{\Phi} \Vert_{\widehat{\mathcal{H}}^{1}} } \geq \sup_{0\neq {\Psi}_{\perp}^* \in  \{\Psi^{*}\}^{\perp}}\frac{\big\vert \left \langle \widetilde{\Phi}_{\Psi_{\perp}}, \mathcal{A}(\bt^*)^{\dagger}\widetilde{\Psi}\right \rangle_{\widehat{\mathcal{H}}^{1} \times \widehat{\mathcal{H}}^{-1}}\big\vert }{\Vert \widetilde{\Phi}_{\Psi_{\perp}} \Vert_{\widehat{\mathcal{H}}^{1}} }\\[0.75em]
	&=  \sup_{0\neq{\Psi}^*_{\perp} \in  \{\Psi^{*}\}^{\perp}} \frac{\left\vert \left\langle \Psi^*_{\perp}, \left(H-\mathcal{E}^*\right)e^{-(\mathcal{T}^*){\dagger}} \widetilde{\Psi} \right\rangle_{\widehat{\mathcal{H}}^1 \times \widehat{\mathcal{H}}^{-1}} \right\vert }{\Vert \mathbb{P}_0^{\perp} e^{-\mathcal{T}(\bt^*)} \Psi^*_{\perp}\Vert_{\widehat{\mathcal{H}}^1 } }\\
	&\geq \frac{1}{\Vert \mathbb{P}_0^{\perp} e^{-\mathcal{T}^*}\Vert_{\widehat{\mathcal{H}}^1 \to  \widehat{\mathcal{H}}^1}}\; \sup_{0\neq\Psi^*_{\perp} \in  \{\Psi^{*}\}^{\perp}} \frac{\left\vert \left\langle \Psi^*_{\perp}, \left(H-\mathcal{E}^*\right)e^{-(\mathcal{T}^*)^{\dagger}} \widetilde{\Psi}\right\rangle_{\widehat{\mathcal{H}}^1 \times \widehat{\mathcal{H}}^{-1}} \right\vert }{\Vert \Psi^*_{\perp}\Vert_{\widehat{\mathcal{H}}^1 } }\\[0.75em]
	&\geq \frac{\gamma}{\Vert \mathbb{P}_0^{\perp} e^{-\mathcal{T}^*}\Vert_{\widehat{\mathcal{H}}^1 \to  \widehat{\mathcal{H}}^1}}\; \left\Vert e^{-(\mathcal{T}^*)^{\dagger}} \widetilde{\Psi}\right\Vert_{\widehat{\mathcal{H}}^{1}}\\[0.75em]
	&\geq \frac{\gamma}{\Vert \mathbb{P}_0^{\perp} e^{-\mathcal{T}^*}\Vert_{\widehat{\mathcal{H}}^1 \to  \widehat{\mathcal{H}}^1} \Vert e^{(\mathcal{T}^*)^{\dagger}} \Vert_{\widehat{\mathcal{H}}^1 \to  \widehat{\mathcal{H}}^1} }\; \big\Vert \widetilde{\Psi}\big\Vert_{\widehat{\mathcal{H}}^{1}},
\end{align*}
where the final step follows from the fact that $e^{-(\mathcal{T}^*)^{\dagger}} \colon \widehat{\mathcal{H}}^{1} \rightarrow \widehat{\mathcal{H}}^{1}$ is a bijection. Defining now the constant $\Theta \in (0, \infty)$ as
\begin{align}\label{eq:Theta}
	\Theta :=  \Vert e^{(\mathcal{T}^*)^{\dagger}}\Vert_{\widehat{\mathcal{H}}^{1}\to \widehat{\mathcal{H}}^{1}} \Vert \mathbb{P}_0^{\perp} e^{-\mathcal{T}^*}\Vert_{\widehat{\mathcal{H}}^{1} \to \widehat{\mathcal{H}}^{1}}, 
\end{align}
and recalling that $\widetilde{\Psi} \in \widetilde{\mathcal{V}}$ was arbitrary, we deduce that
\begin{align*}
	\forall \widetilde{\Psi} \in \widetilde{\mathcal{V}} \colon \quad \Vert \mathcal{A}(\bt^*)^{\dagger} \widetilde{\Psi}\Vert_{\widetilde{\mathcal{V}}^*} \geq \frac{\gamma}{\Theta} \Vert\Psi\Vert_{\widehat{\mathcal{H}}^{1}},
\end{align*}
which completes the proof of the second step.

Combining the conclusions of \textbf{Step 1} and \textbf{Step 2} we deduce that the adjoint operator $\mathcal{A}(\bt^*)^{\dagger} \colon \widetilde{\mathcal{V}} \rightarrow \widetilde{\mathcal{V}}^*$ is an isomorphism, and from this it follows that the operator $\mathcal{A}(\bt^*) \colon \widetilde{\mathcal{V}} \rightarrow \widetilde{\mathcal{V}}^*$ is also an isomorphism. 
\end{proof}

Equipped with Theorem \ref{thm:CC} and recalling the discussion following Notation \ref{not:1}, we immediately obtain the desired invertibility result for the coupled cluster Fr\'echet derivative at any zero $\bt^* \in  \mathbb{V}$ of the coupled cluster function that is associated with a non-degenerate, intermediately normalised eigenfunction $\Psi^* \in \widehat{\mathcal{H}}^1$ of the electronic Hamiltonian.

\begin{corollary}[Invertibility of the Coupled Cluster Fr\'echet Derivative at $\bt^*$]\label{cor:CC_der_inv}~

Let the coupled cluster function $\mathcal{f} \colon \mathbb{V} \rightarrow \mathbb{V}^* $ be defined through Definition \ref{def:CC_function}, for any $\bt \in \mathbb{V}$ let $\mD \mathcal{f}(\bt)$ denote the Fr\'echet derivative of the coupled cluster function as defined through Equation \eqref{eq:CC_Jac_1}, let $\bt^* \in \mathbb{V}$ denote a zero of the coupled cluster function corresponding to an intermediately normalised eigenfunction $\Psi^* \in \widehat{\mathcal{H}}^1$ of the electronic Hamiltonian $H\colon\widehat{\mathcal{H}}^1 \rightarrow \widehat{\mathcal{H}}^{-1}$ with simple, isolated eigenvalue $\mathcal{E}^*$, let $\gamma >0$ denote the inf-sup constant of the shifted Hamiltonian $H-\mathcal{E}^*$ on $ \{\Psi^*\}^{\perp} \subset \widehat{\mathcal{H}}^1$, and let $\Theta > 0$ denote the constant defined through Equation \eqref{eq:Theta}. Then $\mD \mathcal{f}(\bt^*) \colon \mathbb{V} \rightarrow \mathbb{V}^*$ is an isomorphism and it holds that
\begin{align*}
	\Vert \mD \mathcal{f}(\bt^*)^{-1}\Vert_{\mathbb{V}^* \to \mathbb{V}} \leq \frac{\Theta}{\gamma}.
\end{align*}
\end{corollary}

Having completed our study of the coupled cluster Fr\'echet derivative, we are now finally ready to state the main result of this section, namely the local well-posedness of the single reference coupled cluster equations. As mentioned at the beginning of this section, we will do so by appealing to a classical result from non-linear numerical analysis.

\begin{theorem}[Local Well-Posedness of the Coupled Cluster Equations at $\bt^*$]\label{thm:CC_untruncated}~

Let the coupled cluster function $\mathcal{f} \colon \mathbb{V} \rightarrow \mathbb{V}^* $ be defined through Definition \ref{def:CC_function}, let $\bt^* \in \mathbb{V}$ denote a zero of the coupled cluster function corresponding to an intermediately normalised eigenfunction $\Psi^* \in \widehat{\mathcal{H}}^1$ of the electronic Hamiltonian~$H \colon\widehat{\mathcal{H}}^1 \rightarrow \widehat{\mathcal{H}}^{-1}$ with simple, isolated eigenvalue $\mathcal{E}^*$, let $\gamma >0$ denote the inf-sup constant of the shifted Hamiltonian $H-\mathcal{E}^*$ on $\{\Psi^{*}\}^{\perp} \subset \widehat{\mathcal{H}}^1$, let $\Theta > 0$ denote the constant defined through Equation \eqref{eq:Theta}, let the continuity constant $\alpha_{\bt^*} >0$ and the Lipschitz continuity function $\mL_{\bt^*} \colon \mathbb{R}_+ \rightarrow \mathbb{R}_+$ be defined according to Notation \ref{not:1}, and define the constant
\begin{align*}
	\mR:= \min_{\delta >0} \left\{\delta, ~\frac{\gamma}{\mL_{\bt^*}(\delta) \Theta},~ 2\frac{\alpha_{\bt^*}}{\mL_{\bt^*}(\delta)} \right\}.
\end{align*}

Then $\mathcal{f}\big(\mB_{\mR}(\bt^*)\big)$ is an open subset  $\mathbb{V}^*$, the restriction of $\mathcal{f}$ to $\mB_{\mR}(\bt^*)$ is a \emph{diffeomorphism} and for all $ \bs \in \mB_{\mR}(\bt^*)$ we have the error estimate
\begin{equation}\label{eq:CC_untruncated_error}
	\frac{1}{2} \frac{1}{\alpha_{\bt^*} } \Vert \mathcal{f}(\bs) \Vert_{\mathbb{V}^*} \leq \Vert \bt^* - \bs \Vert_{\mathbb{V}} \leq 2 \frac{\Theta}{\gamma} \Vert \mathcal{f}(\bs) \Vert_{\mathbb{V}^*}.
\end{equation}

In particular, $\bt^*$ is the unique solution of the continuous coupled cluster equations \eqref{eq:CC} in the open ball~$\mB_{\mR}(\bt^*)$.
\end{theorem}
\begin{proof}
The fact that the image under $\mathcal{f}$ of the open ball $\mB_{\mR}(\bt^*)$ is itself open and that $\mathcal{f}$ is a local diffeomorphism is a direct consequence of the inverse function theorem for Banach spaces (see, e.g., \cite[Chapter 9]{MR0238472}) while the error estimate is a direct application of \cite[Proposition 2.1]{MR1213837}. The fact that the assumptions of both results are indeed fulfilled by the coupled cluster function $\mathcal{f}$ is a consequence of Proposition \ref{prop:CC_der} and Corollaries \ref{cor:CC_der_Lipschitz} and \ref{cor:CC_der_inv}. 
\end{proof}

Next, let us comment on the constants that appear in the error estimate offered by Theorem \ref{thm:CC_untruncated}.

\begin{remark}[Interpretation of the Constants Appearing in Error Estimate \eqref{eq:CC_untruncated_error}] \label{rem:interp_const}~

Consider the setting of Theorem \ref{thm:CC_untruncated}. From the point of view of a posteriori error quantification, it is important to gain a better understanding of the constants $\gamma>0$ and $\Theta >0$.

Let us recall that the $\gamma >0$ is the inf-sup constant of the shifted Hamiltonian $H-\mathcal{E}^*$ on $\{\Psi^{*}\}^{\perp} \subset \widehat{\mathcal{H}}^1$. A crude lower bound for this constant can be obtained through the following procedure.

We begin by noting that the shifted Hamiltonian $H- \mathcal{E}^*_{\rm GS}+1$ defines a coercive operator on $\widehat{\mathcal{H}}^1$. Since the electronic Hamiltonian is additionally self-adjoint, we can introduce a new norm on $\widehat{\mathcal{H}}^1$ by setting
\begin{align*}
	\forall \Phi \in \widehat{\mathcal{H}}^1 \colon \qquad ||| \Phi |||^2_{\widehat{\mathcal{H}}^1} :=  \left\langle \Phi, \left(H- \mathcal{E}^*_{\rm GS}+1\right) \Phi \right \rangle_{\widehat{\mathcal{H}}^1 \times \widehat{\mathcal{H}}^{-1}},
\end{align*}
and it is clear that this new norm is equivalent to the canonical $\Vert \cdot \Vert_{\widehat{\mathcal{H}}^1}$ norm, i.e., ~$\exists c_{\rm equiv}>1$ such that
\begin{align*}
	\forall \Phi \in \widehat{\mathcal{H}}^1 \colon \quad \frac{1}{c_{\rm equiv}} ||| \Phi |||_{\widehat{\mathcal{H}}^1} \leq \Vert \Phi \Vert_{\widehat{\mathcal{H}}^1}\leq c_{\rm equiv} ||| \Phi |||_{\widehat{\mathcal{H}}^1}.
\end{align*}

In particular, the ellipticity of the electronic Hamiltonian given by Inequality \eqref{eq:ellip} also holds with respect to the new~$|||~\cdot~|||_{\widehat{\mathcal{H}}^1}$ norm and we have
\begin{align}\label{eq:rem_inter_1}
	\forall \Phi \in \widehat{\mathcal{H}}^1 \colon \quad  \left \langle \Phi, \left(H-\mathcal{E}^*\right) \Phi \right \rangle_{\widehat{\mathcal{H}}^1 \times \widehat{\mathcal{H}}^{-1}} \geq \frac{1}{4 c_{\rm equiv}} ||| \Phi |||^2_{\widehat{\mathcal{H}}^1} -\left(9NZ^2 -\mathcal{E}^*-\frac{1}{4}\right ) \Vert \Phi \Vert^2_{\widehat{\mathcal{L}}^2}. 
\end{align}

Moreover, the norm $|||~\cdot~|||_{\widehat{\mathcal{H}}^1}$ also induces a new dual norm $|||~\cdot~|||_{\widehat{\mathcal{H}}^{-1}}$ on the space $\widehat{\mathcal{H}}^{-1}$, and this new norm is also equivalent to the canonical dual norm $\Vert \cdot \Vert_{\widehat{\mathcal{H}}^{-1}}$.

Next, we claim that for any $\Phi \in \{\Psi^*\}^{\perp} \subset \widehat{\mathcal{H}}^1$ there exists ${\Phi}_{\rm flip} \in \{\Psi^*\}^{\perp}$ such that
\begin{align}\label{eq:rem_inter_2}
	\left \langle {\Phi}_{\rm flip}, \left(H-\mathcal{E}^*\right) \Phi \right \rangle_{\widehat{\mathcal{H}}^1 \times \widehat{\mathcal{H}}^{-1}} \geq \Lambda^* \Vert \Phi \Vert^2_{\widehat{\mathcal{L}}^2}, ~ \text{ where } ~ \Lambda^*:= \underset{\substack{\lambda \in \sigma(H)\\ \lambda \neq \mathcal{E}^* } }{\inf}\vert \lambda - \mathcal{E}^*\vert>0 ~ \text{ is the spectral gap at } \mathcal{E}^*. 
\end{align}

To see this, assume that $(\mathcal{E}^*, \Psi^*)$ is the $J^{\rm th}$ eigenpair of the electronic Hamiltonian, ordered non-decreasingly and counting multiplicity. Then we can write any $\Phi \in \{\Psi^*\}^{\perp}$ in the form 
\begin{align*}
	\Phi = \sum_{\ell=1}^{J+1} \mathbb{P}_{\ell}\Phi + \Phi^{\perp}, 
\end{align*}
where each $\mathbb{P}_{\ell}\colon \widehat{\mathcal{H}}^1 \rightarrow \widehat{\mathcal{H}}^1$ denotes the $\widehat{\mathcal{L}}^2$-orthogonal projector onto the span of the $\ell^{\rm th}$ eigenfunction, $\Phi^{\perp}:= \Phi - \sum_{\ell=1}^{J+1} \mathbb{P}_{j}\Phi$, and we emphasise that $\mathbb{P}_{J}\Phi=0$ since $\Phi \in \{\Psi^*\}^{\perp}$. Consequently, for any $\Phi \in \{\Psi^*\}^{\perp}$, we may define $\Phi_{\rm flip} \in \{\Psi^*\}^{\perp}$ as
\begin{align*}
	\Phi_{\rm flip} &:= -\sum_{\ell=1}^{J-1} \mathbb{P}_{\ell}\Phi + \mathbb{P}_{J+1}\Phi +\Phi^{\perp},\\ 
	\intertext{and a direct calculation shows that}
	\left \langle \Phi_{\rm flip}, \left(H-\mathcal{E}^*\right) \Phi \right \rangle_{\widehat{\mathcal{H}}^1 \times \widehat{\mathcal{H}}^{-1}} &\geq \min \left \{\mathcal{E}^*-\mathcal{E}_{J-1}, \mathcal{E}_{J+1}-\mathcal{E}^*  \right \} \Vert \Phi \Vert^2_{\widehat{\mathcal{L}}^2}=:\Lambda^* \Vert \Phi \Vert^2_{\widehat{\mathcal{L}}^2},
\end{align*}
where we have used $\mathcal{E}_{J-1}, \mathcal{E}_{J+1}$ to denote the $J-1$ and $J+1$ eigenvalues of the electronic Hamiltonian. The claim now readily follows. Additionally, it is readily verified that for any $\Phi \in \{\Psi^*\}^{\perp}$ with $\Phi_{\rm flip}$ constructed according to the above procedure, it holds that $||| \Phi_{\rm flip} |||_{\widehat{\mathcal{H}}^1}  = ||| \Phi |||_{\widehat{\mathcal{H}}^1} $.

Defining now the constant $q:= \dfrac{\Lambda^*}{\Lambda^*+ \left(9NZ^2 -\mathcal{E}^*-\frac{1}{4}\right ) } \in (0, 1)$ and combining the Estimates \eqref{eq:rem_inter_1} and \eqref{eq:rem_inter_2}, we deduce that for all $\Phi \in \{\Psi^*\}^{\perp}$ it holds that

\begin{align*}
	\sup_{0\neq \Psi \in \{\Psi^*\}^{\perp}} \frac{\left\vert \left \langle \Psi, (H-\mathcal{E}^*)\Phi \right \rangle_{\widehat{\mathcal{H}}^1 \times \widehat{\mathcal{H}}^{-1}} \right\vert}{||| \Psi |||_{\widehat{\mathcal{H}}^1}}&= q\sup_{0\neq \Psi \in \{\Psi^*\}^{\perp}} \frac{\left\vert \left \langle \Psi, (H-\mathcal{E}^*)\Phi \right \rangle_{\widehat{\mathcal{H}}^1 \times \widehat{\mathcal{H}}^{-1}} \right\vert}{||| \Psi |||_{\widehat{\mathcal{H}}^1}}\\[1em]
	&+ (1-q) \sup_{0\neq \Psi \in \{\Psi^*\}^{\perp}} \frac{\left\vert \left \langle \Psi, (H-\mathcal{E}^*)\Phi \right \rangle_{\widehat{\mathcal{H}}^1 \times \widehat{\mathcal{H}}^{-1}} \right\vert}{||| \Psi |||_{\widehat{\mathcal{H}}^1}}\\[1em]
	&\geq q\frac{\left\vert \left \langle \Phi, (H-\mathcal{E}^*)\Phi \right \rangle_{\widehat{\mathcal{H}}^1 \times \widehat{\mathcal{H}}^{-1}} \right\vert}{||| \Phi |||_{\widehat{\mathcal{H}}^1}}\\[1em]
	&+ (1-q)  \frac{\left\vert \left \langle \Phi_{\rm flip}, (H-\mathcal{E}^*)\Phi \right \rangle_{\widehat{\mathcal{H}}^1 \times \widehat{\mathcal{H}}^{-1}} \right\vert}{||| \Phi_{\rm flip} |||_{\widehat{\mathcal{H}}^1}}\\[1em]
	&\geq q\frac{1}{||| \Phi |||_{\widehat{\mathcal{H}}^1}}\left(\frac{1}{4 \;c_{\rm equiv}} ||| \Phi |||^2_{\widehat{\mathcal{H}}^1} -  \left(9NZ^2 -\mathcal{E}^*-\frac{1}{4}\right ) \Vert \Phi \Vert^2_{\widehat{\mathcal{L}}^2}\right) \\[1em]
	&+(1-q) \frac{1}{||| \Phi |||_{\widehat{\mathcal{H}}^1}}\Lambda^*  \Vert \Phi \Vert^2_{\widehat{\mathcal{L}}^2}\\[1em]
	&= q\frac{1}{4\; c_{\rm equiv}} ||| \Phi |||_{\widehat{\mathcal{H}}^1},
\end{align*}
where the cancellations in the last step occurs due to the definition of $q \in (0,1)$.

Recalling the definition of the constant $q$, we see that the inf-sup constant $\gamma$ is lower bounded by the product of the spectral gap and a factor depending on $\Lambda^*, \mathcal{E}^*, N$, and $Z$. Consequently, the residual-based CC error estimate \eqref{eq:CC_untruncated_error} may be expected to degrade if the spectral gap degrades.

Coming now to the constant $\Theta$, we see that it is simply the product of two operator norms involving the exponential cluster operator and its adjoint. Thanks to the continuity of the mapping $\mathbb{V} \ni \bt \mapsto e^{-\mathcal{T}(\bt)} \colon \widehat{\mathcal{H}}^1\rightarrow \widehat{\mathcal{H}}^1$, (and its adjoint) we deduce that these operators norms will be large when $\Vert \bt\Vert_{\mathbb{V}}$ is large, and therefore the residual-based CC error estimate \eqref{eq:CC_untruncated_error} is expected to degrade if $\Vert \bt\Vert_{\mathbb{V}}$ is large.
\end{remark}

We conclude this section by emphasising, in particular, that if the ground state energy of the electronic Hamiltonian $H \colon \widehat{\mathcal{H}}^1  \rightarrow \widehat{\mathcal{H}}^{-1} $ is simple, and the chosen reference determinant $\Psi_0$ is not orthogonal to the corresponding ground state wave-function, then the continuous coupled cluster equations \eqref{eq:Galerkin_Full_CC} are locally well-posed, and we have access to the residual-based error estimates given by Theorem \ref{thm:CC_untruncated}.

%% file: Truncated_CC_New.tex
Having understood the local well-posedness of the continuous coupled cluster function, the next step in our analysis is to study the discrete coupled cluster equations~\eqref{eq:CC_truncated}. Unfortunately, obtaining a local well-posedness result for an arbitrary choice of excitation subset or $N$-particle basis set is a highly non-trivial exercise. Indeed, as the subsequent exposition will show (see Lemma \ref{lem:Galerkin_inv_1} and Theorem \ref{thm:Galerkin_full_CC} below), our discrete local well-posedness analysis depends on being able to demonstrate that certain discrete inf-sup conditions hold and the establishment of these conditions for arbitrary discretisations is not obvious. For the purpose of this article therefore, we will limit ourselves to an analysis of the so-called Full-Coupled Cluster equations in a finite basis. The extension of our analysis to more general discretisations (the so-called \emph{truncated} CC equations \cite[Chapter 13]{helgaker2014molecular}) will be addressed in a forthcoming contribution.

Throughout this section, we assume the settings of Sections \ref{sec:2}-\ref{sec:5}. In particular, we will frequently refer to the notions of Section \ref{sec:3}. Let $\{\psi_{j}\}_{j \in \mathbb{N}}$ denote an $\mL^2(\R^3; \C)$-orthonormal basis for $\mH^1(\R^3;\C)$. For any $K \in \mathbb{N}$, we define $\mathcal{B}_K:= \left\{\psi_j \right\}_{j=1}^K$ and $\mX_K := \text{span } \mathcal{B}_K$. 

Recall that we denote by $N\in \mathbb{N}$ the number of electrons in the system under study. Our goal now is to use the sets $\{\mathcal{B}_K\}_{K\in \mathbb{N}}$ to construct a sequence of finite-dimensional, nested subspaces of the antisymmetric tensor product space $\widehat{\mathcal{H}}^1$ whose union is dense in $\widehat{\mathcal{H}}^1$. To avoid tedious notation in this construction, we will always assume that $K$ is a natural number such that $K \geq N$. Proceeding now, exactly as in Section \ref{sec:3}, we first introduce for each such $K$ the index set $\mathcal{J}_{K}^N \subset \{1, \ldots, K\}^N$ given by
\begin{align*}
	\mathcal{J}_{K}^N := \Big\{\boldsymbol{\ell} = (\ell_1, \ell_2, \ldots, \ell_N)\in \{1, \ldots, K\}^N \colon \ell_1 < \ell_2 < \ldots < \ell_N\Big\}.
\end{align*}

Next, we define for each $K$ the set of $\widehat{\mathcal{L}}^2$-orthonormal, $N$-particle determinants $\mathcal{B}_K^{N} \subset \widehat{\mathcal{H}}^1$ as
\begin{align*}
	\mathcal{B}^{N}_K := \left\{ \Psi_{\bold{k}}(\bold{x}_1, \bold{x}_2, \ldots, \bold{x}_N)=\frac{1}{\sqrt{N!}}\text{\rm det} \big(\psi_{k_i}(\bold{x}_j)\big)_{i, j=1}^N \colon \hspace{1mm} \bold{k}=(k_1, k_2, \ldots, k_N) \in \mathcal{J}_{K}^N \right\},
\end{align*} 
and we denote, as usual, $\Psi_0(\bold{x}_1, \ldots, \bold{x}_N):= \text{det}\big(\psi_{i}(\bold{x}_j)\big)_{i, j=1}^N$.

It now follows that we can define the sequence $\{\mathcal{V}_K\}_{K \geq N}$ of subspaces of $\widehat{\mathcal{H}}^1$ as $\mathcal{V}_K := \text{\rm span } \mathcal{B}^N_K$, and it holds that
\begin{align*}
	\forall~K\geq N\colon \quad \text{\rm dim }\mathcal{V}_K ={{K}\choose{K-N}}, \qquad \forall~ K_2 > K_1 \geq N  \colon \quad \mathcal{V}_{K_1} \subset \mathcal{V}_{K_2} \quad \text{and} \quad \overline{\bigcup_{\substack{K\geq N }} \mathcal{V}_K}^{\Vert \cdot \Vert_{\widehat{\mathcal{H}}^1}}= \widehat{\mathcal{H}}^1.
\end{align*}

Equipped with the sequence of finite-dimensional subspaces $\{\mathcal{V}_K\}_{K \geq N}$ whose union is dense in $\widehat{\mathcal{H}}^1$, our next task is to introduce a corresponding sequence of finite-dimensional coefficient spaces $\{\mathbb{V}_K\}_{K \geq N}$ whose union is dense in the Hilbert space of sequences $\mathbb{V}$ that was introduced through Definition \ref{def:coeff_space}. To this end, we require some definitions.

\begin{definition}[Excitation Index Sets For Finite Bases]\label{def:Excitation_Index_finite}~
	
	For each $K$ and each $j \in \{1, \ldots, N\}$ we define the index set $\mathcal{I}^K$ as
	\begin{align*}
		\mathcal{I}_j^K := \left\{ {{i_1, \ldots, i_j}\choose{\ell_1, \ldots, \ell_j}} \colon i_1 < \ldots< i_j \in \{1, \ldots, N\} \text{ and } \ell_1< \ldots < \ell_j  \in \{N+1, \ldots, K\} \right\},
	\end{align*}
	we set
	\begin{align*}
		\mathcal{I}^K:= \bigcup_{j=1}^N \mathcal{I}^K_j,
	\end{align*}
	and we emphasise that $\underset{K \geq N}{\bigcup} \mathcal{I}^K =\mathcal{I}$, i.e., the global excitation index defined through Definition \ref{def:Excitation_Index}.
\end{definition}

Consider Definition \ref{def:Excitation_Index_finite} of the excitation index sets $\mathcal{I}_j^K, ~j \in \{1, \ldots, N\}$. Since each $\mathcal{I}_j^K$ is a subset of the global excitation index set $\mathcal{I}$ defined through Definition \ref{def:Excitation_Index}, it follows that we can define for any $\mu \in \mathcal{I}_j^K$,  excitation and de-excitation operators $\mathcal{X}_{\mu} \colon \widehat{\mathcal{H}}^1 \rightarrow \widehat{\mathcal{H}}^1$ and $\mathcal{X}^{\dagger}_{\mu}\colon \widehat{\mathcal{H}}^1 \rightarrow \widehat{\mathcal{H}}^1$ through Definitions \ref{def:Excitation_Operator} and \ref{def:De-excitation_Operator} respectively. Moreover, the results of Theorem \ref{pro:excit} can be applied to these excitation and de-excitation operators, and the following remark summarises some additional properties of these elementary excitation and de-excitation operators.% corresponding to the index sets $\mathcal{I}_j^K, ~ K\in \mathbb{N}_N, ~j \in \{1, \ldots, N\}$.

\begin{remark}[Properties of Excitation and De-excitation Operators Related to the Index Set $\mathcal{I}^K$]\label{rem:properties}~
	
	Let the excitation index set $\mathcal{I}^K$ be defined according to Definition \ref{def:Excitation_Index_finite}. Then the finite-dimensional $N$-particle basis $\mathcal{B}_K^N$ and the finite-dimensional $N$-particle approximation space $\mathcal{V}_K$ have the decomposition
	\begin{align*}
		\mathcal{B}_K^N:=& \left\{\Psi_0\right\} \cup \{\mathcal{X}_{\mu} \Psi_0 \colon ~ \mu \in \mathcal{I}^K\},\\[0.5em]
		\mathcal{V}_K:=& \text{\rm span}\left\{\Psi_0\right\} \oplus \underbrace{\text{\rm span}\{\mathcal{X}_{\mu} \Psi_0 \colon ~ \mu \in \mathcal{I}^K\}}_{:= \widetilde{\mathcal{V}}_K}.
	\end{align*}
	Additionally, for any $\mu, \nu \in \mathcal{I}^K$ and $\sigma \in \mathcal{I} \setminus \mathcal{I}^K$ it holds that
	\begin{align*}
		\mathcal{X}_{\mu} \mathcal{X}_{\nu}\Psi_0 \in \widetilde{\mathcal{V}}_K \qquad &\text{and} \qquad \mathcal{X}_{\mu}^{\dagger} \mathcal{X}_{\nu}\Psi_0 \in \widetilde{\mathcal{V}}_K\\[0.5em]
		\mathcal{X}_{\mu} \mathcal{X}_{\sigma}\Psi_0 \notin \mathcal{B}_K^N \qquad &\text{and} \qquad \mathcal{X}_{\mu}^{\dagger} \mathcal{X}_{\sigma}\Psi_0 \notin \mathcal{B}_K^N,\\[0.5em]
		\mathcal{X}_{\sigma} \mathcal{X}_{\nu}\Psi_0 \notin \mathcal{B}_K^N \qquad &\text{and} \qquad \mathcal{X}_{\sigma}^{\dagger} \mathcal{X}_{\nu}\Psi_0 =0.
	\end{align*}
	
	Finally, as in Section \ref{sec:5} we will denote $\Psi_{\mu}:= \mathcal{X}_{\mu} \Psi_0$ for any $\mu \in \mathcal{I}^K$.
\end{remark}

Next we will introduce subspaces of coefficient vectors corresponding to the excitation index sets $\left\{\mathcal{I}^K\right\}_{K\geq N}$. The following construction is essentially an adaptation of Definition \ref{def:coeff_space} of the sequence space $\mathbb{V}$ to finite dimensions.

\begin{definition}[Finite-Dimensional Coefficient Spaces]\label{def:coeff_space_finite}~
	
	Let the excitation index set $\mathcal{I}_K$ be defined through Definition \ref{def:Excitation_Index_finite} for $K\geq N$, and let the Hilbert space of sequences~$\mathbb{V}$ be defined according to Definition \ref{def:coeff_space}. We define the Hilbert subspace of coefficients $\mathbb{V}_{K} \subset \mathbb{V}$ as the set
	\begin{equation}\label{eq:coeff_space_finite}
		\mathbb{V}_K:= \left\{\bold{t}:= (\bt_{\mu})_{\mu \in \mathcal{I}} \in \mathbb{V}\colon\quad \bt_{\mu}=0 ~\forall \mu \notin \mathcal{I}^K \right\},
	\end{equation}
	equipped with the $(\cdot, \cdot)_{\mathbb{V}}$ inner product.
\end{definition}

\begin{Notation}
	Consider Definition \ref{def:coeff_space_finite} of the Hilbert subspace of coefficients $\mathbb{V}_K$, and let $\bt \in \mathbb{V}_K$ denote an arbitrary element. In the sequel, for clarity of exposition we will frequently denote $\bt:= \bt_K := \{\bt_{\mu}\}_{\mu \in \mathcal{I}^K}$. In other words, by an abuse of notation, we will identify $\mathbb{V}_K$ with the set $\ell^2\left(\mathcal{I}^K\right)$ but equipped with the $(\cdot, \cdot)_{\mathbb{V}}$ inner product.
\end{Notation}

As can be expected, the coefficient subspaces $\{\mathbb{V}_K\}_{K\geq N}$ introduced through Definition \ref{def:coeff_space_finite} inherit many properties from the $N$-particle approximation spaces $\{\mathcal{B}_K^N\}_{K\geq N}$. Indeed, we have the following lemma.

\begin{lemma}[Density of Finite-Dimensional Coefficient Spaces]\label{lem:coeff_space_finite}~
	
	Let the infinite-dimensional Hilbert space of sequences $\mathbb{V} \subset \ell^2(\mathcal{I})$ be defined through Definition \ref{def:coeff_space} and let the Hilbert subspace of coefficients $\mathbb{V}_{{K}} \subset \mathbb{V}$ be defined through Definition \ref{def:coeff_space_finite} for $K\geq N$. Then it holds that
	\begin{align*}
		\forall~ K_2 > K_1 \geq N  \colon \quad \mathbb{V}_{K_1} \subset \mathbb{V}_{K_2} \qquad \text{and} \qquad  \overline{\bigcup_{\substack{K\geq N }} \mathbb{V}_K}^{\Vert \cdot \Vert_{\mathbb{V}}}= \mathbb{V}.
	\end{align*}
\end{lemma}
\begin{proof}
	The set inclusion is obvious so we focus on proving the density. Note that the density result would also be obvious had the sequence space $\mathbb{V}$ been equipped with the $\Vert \cdot \Vert_{\ell^2}$ norm. The density is slightly subtle precisely because we have equipped $\mathbb{V}$ with the non-standard $\Vert \cdot \Vert_{\mathbb{V}}$ norm.
 
 %Let $\bs^*= \{\bs_{\mu}\}_{\mu \in \mathcal{I}}$ denote an arbitrary element of $\mathbb{V}$. It follows that the function $\Psi_{\bs}^*:= \sum_{\mu \in \mathcal{I}} \bs_{\mu}\Psi_{\mu}$ is an element of $\text{span}\{\Psi_0\}^{\perp}\subset \widehat{\mathcal{H}}^1$.

	Recall that the union of the $N$-particle approximation spaces $\left\{\mathcal{V}_K\right\}_{K\geq N}$ is dense in $\widehat{\mathcal{H}}^1$. From this we deduce that the union of the $N$-particle approximation \emph{subspaces} $\{\widetilde{\mathcal{V}}_K\}_{K\geq N}$ is dense in $\text{span}\{\Psi_0\}^{\perp}$, where we remind the reader that $\widetilde{\mathcal{V}}_K= \text{\rm span} \{\Psi_{\mu}\colon ~ \mu \in \mathcal{I}^K\} = \{\Psi \in \mathcal{V}_K \colon (\Psi, \Psi_0)_{\widehat{\mathcal{L}}^2}=0\}$.
	
	Consequently, there exists a sequence of functions $\{\Psi_K\}_{K\geq N}$ with each $\Psi_K:= \sum_{\mu\in \mathcal{I}^K} \bt^K_\mu \Psi_{\mu} \in \widetilde{\mathcal{V}}_K$ such that $\lim_{K \to \infty}\Vert \Psi_K -\Psi_{\bs}^* \Vert_{\widehat{\mathcal{H}}^1}=0$. Defining for each $K \geq N$, the sequence $\bt_K \in \mathbb{V}_K $ as $\bt_K = \{\bt^K_{\mu}\}_{\mu \in \mathcal{I}^K}$, and using the definition of the $\Vert\cdot\Vert_{\mathbb{V}}$ norm now yields the required density.
\end{proof}

We are now ready to state the discrete coupled cluster equations corresponding to the approximation spaces we have introduced above. As mentioned at the beginning of this section, these equations are known in the quantum chemical literature as the Full-Coupled Cluster equations in a finite basis.

\vspace{1mm}
\textbf{Full-Coupled Cluster Equations in a Finite Basis:}~    

Let the excitation index set $\mathcal{I}^K$ be defined through Definition \ref{def:Excitation_Index_finite} for $K\geq N$, let the Hilbert subspace of coefficients $\mathbb{V}_K \subset \mathbb{V}$ be defined through Definition \ref{def:coeff_space_finite}, and let the coupled cluster function $\mathcal{f}\colon \mathbb{V} \rightarrow \mathbb{V}^*$ be defined through Definition \ref{def:CC_function}. We seek a coefficient vector $\bt_K \in \mathbb{V}_K$ such that for all coefficient vectors $\bs_K \in \mathbb{V}_K$ it holds that
\begin{equation}\label{eq:Galerkin_truncated_CC_1}
	\left\langle \bs_K, \mathcal{f}(\bt_K)\right\rangle_{\mathbb{V} \times \mathbb{V}^*}=0.
\end{equation}        

\vspace{3mm}
The remainder of this section will be concerned with the (local) well-posedness analysis of Equation \eqref{eq:Galerkin_truncated_CC_1}. We begin with a definition.

\begin{definition}[Restricted Coupled Cluster function on Full-CI spaces]\label{def:CC_function_FCI}~
	
	Let the excitation index set $\mathcal{I}^K$ be defined through Definition \ref{def:Excitation_Index_finite} for $K\geq N$ and let the Hilbert subspace of coefficients $\mathbb{V}_K \subset \mathbb{V}$ be defined through Definition \ref{def:coeff_space_finite}. We define the restricted coupled cluster function $\mathcal{f}_K \colon \mathbb{V}_K\rightarrow \mathbb{V}^*_K$ as the mapping with the property that for all $\bt_K, \bs_K \in \mathbb{V}_K$ it holds that
	\begin{align*}
		\langle \bs_K, \mathcal{f}_{K}(\bt_K)\rangle_{\mathbb{V}_K \times \mathbb{V}_K^*} := \left \langle \bs_K, \mathcal{f}(\bt_K)\right \rangle_{\mathbb{V}\times \mathbb{V}^*}.
	\end{align*}
\end{definition}

It is readily seen that solutions $\bt_K^* \in \mathbb{V}_K$ to the Full-CC equations in a finite basis \eqref{eq:Galerkin_truncated_CC_1} are nothing else than zeros of the restricted coupled cluster function $\mathcal{f}_K \colon \mathbb{V}_K \rightarrow \mathbb{V}_K^*$ defined through Definition \ref{def:CC_function_FCI}. The following result, whose proof can, for instance, be found in \cite{Schneider_1}, is essentially a finite-dimensional analogue of Theorem \ref{thm:Yvon} and establishes a relationship between these zeros of the restricted coupled cluster function and intermediately normalised eigenfunctions of the Full-CI Hamiltonian $H_{K} \colon \mathcal{V}_K \rightarrow \mathcal{V}_K^*$ defined through Equation \eqref{eq:FCI_Hamiltonian}.

%As we shall see, this analysis is made simple by the observation that, in some sense, the Full-CC equations in a finite basis have the same relation with the Full-CI minimisation problem \eqref{eq:FCI} as do the continuous coupled cluster equations~\eqref{eq:Galerkin_Full_CC} with the exact (infinite-dimensional) minimisation problem \eqref{eq:Ground_State}.

\begin{theorem}[Relation of Restricted Coupled Cluster Zeros and Full-CI Hamiltonian Eigenfunctions]\label{thm:Yvon_2}~
	
	Let the restricted coupled cluster function $\mathcal{f}_K \colon \mathbb{V}_K \rightarrow \mathbb{V}_K^*$ be defined through Definition \ref{def:CC_function_FCI}, and let the Full-CI Hamiltonian $H_{K} \colon \mathcal{V}_K \rightarrow \mathcal{V}_K^*$ be defined through Equation \eqref{eq:FCI_Hamiltonian}. Then 
	\begin{enumerate}
		\item For any zero $\bt_K^* = \{\bt_{\mu}^*\}_{\mu \in \mathcal{I}^K}\in \mathbb{V}_K$ of the restricted CC function, the function $\Psi_K^*=e^{\mathcal{T}_K^*}\Psi_0 \in \mathcal{V}_K$ with $\mathcal{T}_K^*=\sum_{\mu \in \mathcal{I}^K} \bt^*_\mu \mathcal{X}_\mu$ is an intermediately normalised eigenfunction of the Full-CI Hamiltonian. Moreover, the eigenvalue corresponding to the eigenfunction $\Psi_K^*$ coincides with the discrete CC energy $\mathcal{E}_{K, \rm CC}^*$ generated by $\bt_K^*$ as defined through Equation \eqref{eq:CC_energy_proj}. 
		
		\item Conversely, for any intermediately normalised eigenfunction $\Psi_K^* \in \mathcal{V}_K$ of the Full-CI Hamiltonian, there exists $\bt_K^* = \{\bt_{\mu}^*\}_{\mu \in \mathcal{I}^K}\in \mathbb{V}_K$ such that $\bt_K^*$ is a zero of the restricted CC function and $\Psi_K^*=e^{\mathcal{T}_K^*}\Psi_0 \in \mathcal{V}_K$ with $\mathcal{T}_K^*=\sum_{\mu \in \mathcal{I}^K} \bt^*_\mu \mathcal{X}_\mu$. Moreover, the discrete CC energy $\mathcal{E}_{K, \rm CC}^*$ generated by $\bt_K^*$ through through Equation~\eqref{eq:CC_energy_proj} coincides with the eigenvalue corresponding to the eigenfunction $\Psi_K^*$.
		
	\end{enumerate}
\end{theorem}

In view of Theorem \ref{thm:Yvon_2}, the goal of our analysis in this section will be two-fold: first, we would like to demonstrate, exactly as in the infinite-dimensional case, that solutions $\bt_K^* \in \mathbb{V}_K$ of the Full-CC equations \eqref{eq:Galerkin_truncated_CC_1} that correspond to non-degenerate eigenpairs of the Full-CI Hamiltonian are locally unique. Second, we wish to obtain a characterisation of the error between solutions $\bt_K^* \in \mathbb{V}_K$ of the Full-CC equations \eqref{eq:Galerkin_truncated_CC_1} and solutions $\bt^* \in \mathbb{V}$ of the continuous coupled cluster equations \eqref{eq:Galerkin_Full_CC}. For the latter analysis we will appeal to classical results from the numerical analysis of Galerkin discretisations of non-linear equations but the former task is essentially trivial since the Full-CC equations have the same structure as the continuous CC equations and hence our proofs from Section \ref{sec:5} can be copied with minor amendments. For the sake of brevity therefore, we simply state the final result on local uniqueness of solutions to the Full-CC equations in a finite basis \eqref{eq:Galerkin_truncated_CC_1}.

\begin{theorem}[Local Well-Posedness of the Full-Coupled Cluster Equations in a Finite Basis]\label{thm:CC_truncated_full}~
	
	Let $\mathbb{V}_K \subset \mathbb{V}$ denote the Hilbert subspace of coefficients as defined through Definition \ref{def:coeff_space_finite} for $K \geq N$, let the restricted coupled cluster function $\mathcal{f}_K \colon \mathbb{V}_K \rightarrow \mathbb{V}_K^*$ be defined through Definition \ref{def:CC_function_FCI}, let $\bt_K^*:= \{\bt^*_{\mu}\}_{\mu \in \mathcal{I}^K} \in \mathbb{V}_K$ denote a zero of the restricted coupled cluster function corresponding to any intermediately normalised eigenfunction $\Psi_K^* \in \mathcal{V}_K$ of the Full-CI Hamiltonian $H_{K} \colon \mathcal{V}_K \rightarrow \mathcal{V}_K^*$ with non-degenerate eigenvalue $\mathcal{E}_K^*$, let $\gamma_K >0$ denote the inf-sup constant of the shifted Full-CI Hamiltonian $H_K-\mathcal{E}_K^*$ on $\{\Psi_K^*\}^{\perp} \subset \mathcal{V}_K$, let $\Theta_K > 0$ be defined as~$\Theta_K :=  \Vert e^{(\mathcal{T}^*_K)^{\dagger}}\Vert_{\mathcal{V}_K\to \mathcal{V}_K} \Vert \mathbb{P}_0^{\perp} e^{-\mathcal{T}_K^*}\Vert_{\mathcal{V}_K \to \mathcal{V}_K}$ with $\mathcal{T}_K^*:= \sum_{\mu \in \mathcal{I}^K}\bt^*_{\mu} \mathcal{X}_{\mu} $, let the continuity constant $\alpha_{\bt_K^*} >0$ and the Lipschitz continuity function $\mL_{\bt_K^*} \colon \mathbb{R}_+ \rightarrow \mathbb{R}_+$ be defined according to Notation~\ref{not:1}, and define the constant
	\begin{align*}
		\mR:= \mR(K):= \min_{\delta >0} \left\{\delta, ~\frac{\gamma_K}{\mL_{\bt_K^*}(\delta) \Theta_K},~ 2\frac{\alpha_{\bt_K^*}}{\mL_{\bt_K^*}(\delta)} \right\}.
	\end{align*}
	
	Then $\mathcal{f}_K\big(\mB_{\mR}(\bt^*_K)\big)$ is an open subset of $\mathbb{V}_K^*$, the restriction of $\mathcal{f}_K$ to $\mB_{\mR}(\bt_K^*)$ is a \emph{diffeomorphism}, and for all $ \bs_K \in \mB_{\mR}(\bt_K^*)$ we have the error estimate
	\begin{equation}\label{eq:CC_full_error}
		\frac{1}{2} \frac{1}{\alpha_{\bt_K^*} } \Vert \mathcal{f}_K(\bs_K) \Vert_{\mathbb{V}_K^*} \leq \Vert \bt_K^* - \bs_K \Vert_{\mathbb{V}_K} \leq 2 \frac{\Theta_K}{\gamma_K} \Vert \mathcal{f}_K(\bs_K) \Vert_{\mathbb{V}_K^*}.
	\end{equation}
	
	In particular, $\bt_K^*$ is the unique solution of the Full-Coupled Cluster equations in a finite basis \eqref{eq:CC} in the open ball $\mB_{\mR_K}(\bt_K^*)$.
\end{theorem}

\begin{proof}
	The proof is essentially identical to the proof of Theorem \ref{thm:CC_untruncated} with some obvious modifications. We first obtain an expression for the Full-CC Jacobian $\mD \mathcal{f}_K(\bt_K) \colon \mathbb{V}_K  \rightarrow \mathbb{V}_K^*$ at any $\bt_K \in \mathbb{V}_K$ exactly as in the infinite-dimensional case. Thanks to Theorem \ref{thm:Yvon_2}, we can deduce from this expression that the Jacobian $\mD \mathcal{f}_K(\bt_K^*)$ at any zero $\bt_K^* \in \mathbb{V}_K$ of the restricted CC function has the form
	\begin{align*}
		\left \langle  \bw_K, \mD \mathcal{f}_K(\bt_K^*) \bs_K \right\rangle_{\mathbb{V}_K \times \mathbb{V}_K^*} = \left \langle  \sum_{\mu \in \mathcal{I}^K} \bw_{\mu} \mathcal{X}_\mu \Psi_0, e^{-\mathcal{T}_K^*} \left(H - \mathcal{E}_K^*\right) e^{\mathcal{T}_K^*} \sum_{\nu \in \mathcal{I}^K} \bs_{\nu} \mathcal{X}_\nu\Psi_0\right \rangle_{\widehat{\mathcal{H}}^1 \times \widehat{\mathcal{H}}^{-1}},
	\end{align*}
	for all $\bs_K, \bw_K \in \mathbb{V}_K $ with $\bs_K =\{\bs_{\nu}\}_{\nu \in \mathcal{I}^K}$ and $ \bw_K=\{\bw_{\mu}\}_{\mu \in \mathcal{I}^K}$ where $\mathcal{T}_K^*:= \sum_{\mu \in \mathcal{I}^K}\bt^*_{\mu} \mathcal{X}_{\mu} $. %where $\mathcal{E}_K^*$ denotes the eigenvalue corresponding to the eigenfunction $\Psi_K^* \in \mathcal{V}_K$ of the Full-CI Hamiltonian that is parameterised by $\bt_K^*$.
	
	In analogy with Definition \ref{def:A_op}, we can then introduce an operator $\mathcal{A}_K(\bt_K^*) \colon \widetilde{\mathcal{V}}_K \rightarrow \widehat{\mathcal{H}}^{-1}$ that characterises the action of the Full-CC derivative $\mD \mathcal{f}_K(\bt_K^*)$ and show that this operator is an isomorphism from $\widetilde{\mathcal{V}}_K$ to $\widetilde{\mathcal{V}}_K^*$ exactly as in Theorem \ref{thm:CC}. The local-uniqueness result then readily follows.
\end{proof}

Consider the setting of Theorem \ref{thm:CC_truncated_full}. For very small molecules discretised in minimal basis sets, it is possible to perform Full-CI calculations and thereby gain access to the derivative $\mD\mathcal{f}_K(\bt)$ of the restricted coupled cluster function at $\bt = \bt^*_{\rm FCI} \in \mathbb{V}_K$, i.e., at the coefficient vector $\bt^*_{\rm FCI}$ which generates the Full-CI ground state wave-function. It is natural to ask how the bounds that we have derived compare to the exact norm of the inverse $\mD\mathcal{f}^{-1}_K(t^*_{\rm FCI})$. While a comprehensive numerical study is left to a future contribution, some preliminary numerical results are given in Table \ref{table2} and Figures \ref{fig:01} and \ref{fig:02}. Based on these results, the lower bounds that we have derived for the operator norm $\Vert \mD\mathcal{f}^{-1}_K(\bt^*_{\rm FCI})\Vert^{-1}_{\mathbb{V}_K^* \to \mathbb{V}_K}$ are fairly sharp at equilibrium but tend to degrade in the bond dissociation~regime. 

\begin{table}[h!]
	\centering
	\begin{tabular}{||c| c| c| c| c||} 
		\hline \hline
		Molecule & $\Vert \bt^*_{\rm FCI}\Vert_{\mathbb{V}_K} $ &  \shortstack{Monotonicity constant \\$\Gamma$ from Eq. \eqref{eq:loc_mono_2}} &  $\Vert\mD\mathcal{f}^{-1}_K(t^*_{\rm FCI})\Vert^{-1}_{\mathbb{V}_K^* \to \mathbb{V}_K}$  & $\nicefrac{\gamma_K}{\Theta_K}$\\ [0.5ex] 
		\hline\hline
		{\rm BeH2} & 0.2343 &  \hphantom{-}0.0363 & 0.3379 &0.2568\\ 
		{\rm BH3} & 0.2844   & {\color{red}-0.0950} &  0.3060&0.2081\\
		\rm{HF} & 0.2038   & {\color{red}-0.0083} &0.2628&0.2164\\
		\rm{H2O} & 0.2687   &  \hphantom{-}0.0249 &0.4113&0.2784\\ 
		\rm{LiH} & 0.1792   & {\color{red}-0.0065} &0.2995&0.2529\\
		\rm{NH3} & 0.3074   & {\color{red}-0.0325} &0.3576&0.2789\\[1ex] 
		\hline\hline
	\end{tabular}
	\vspace{2mm}
	\caption{Examples of numerically computed constants for a collection of small molecules at equilibrium geometries. The calculations were performed in STO-6G basis sets with the exception of the HF and LiH molecules for which 6-31G basis sets were used. To simplify calculations, the canonical $\widehat{\mathcal{H}}^1$ norm was replaced with an equivalent norm induced by the diagonal part of the shifted Full-CI Hamiltonian.}\label{table2}
\end{table}

\begin{figure}[h!]
	\centering
	\begin{subfigure}{0.495\textwidth}
		\centering
		\includegraphics[width=\textwidth, trim={0cm, 0cm, 0cm, 0cm},clip=true]{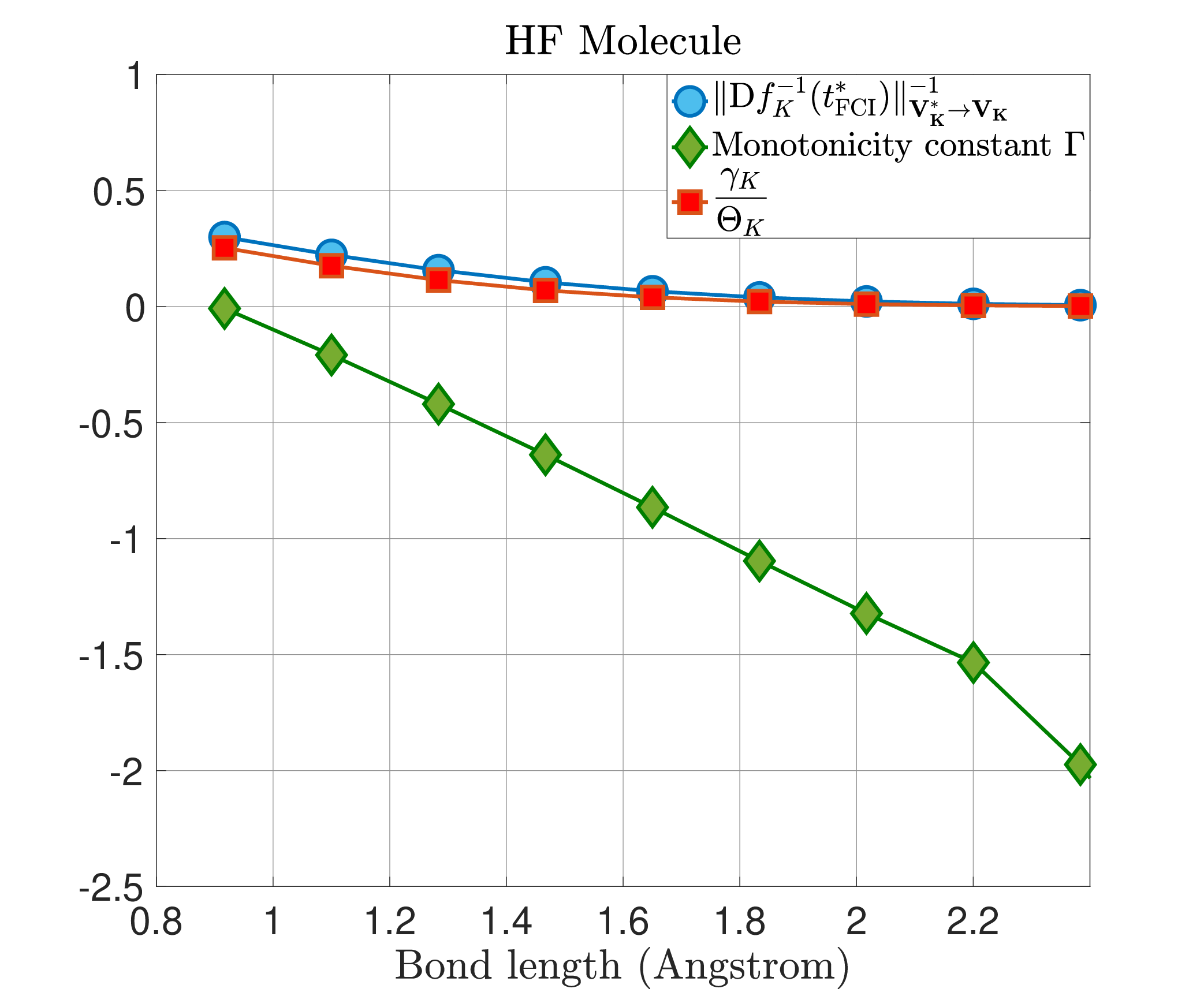} 
	\end{subfigure}\hfill
	\begin{subfigure}{0.495\textwidth}
		\centering
		\includegraphics[width=\textwidth, trim={0cm, 0cm, 0cm, 0cm},clip=true]{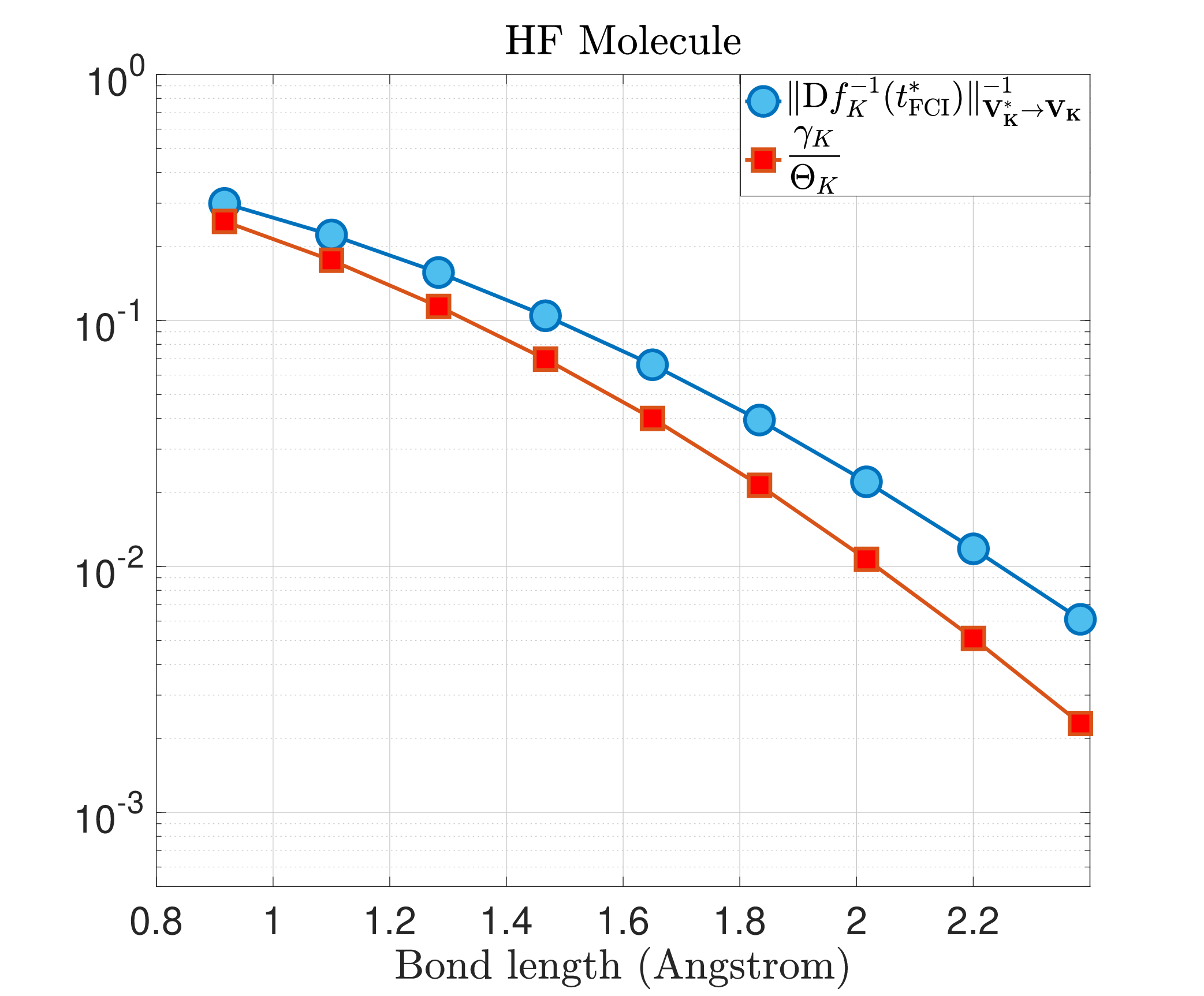} 
	\end{subfigure}
	\caption{Numerically computed constants for the HF molecule at different bond lengths. The equilibrium bond length is 0.9168 Angstrom. The figure on the right uses a log scale on the y-axis.}
	\label{fig:01}
\end{figure}

\begin{figure}[h!]
	\centering
	\begin{subfigure}{0.495\textwidth}
		\centering
		\includegraphics[width=\textwidth, trim={0cm, 0cm, 0cm, 0cm},clip=true]{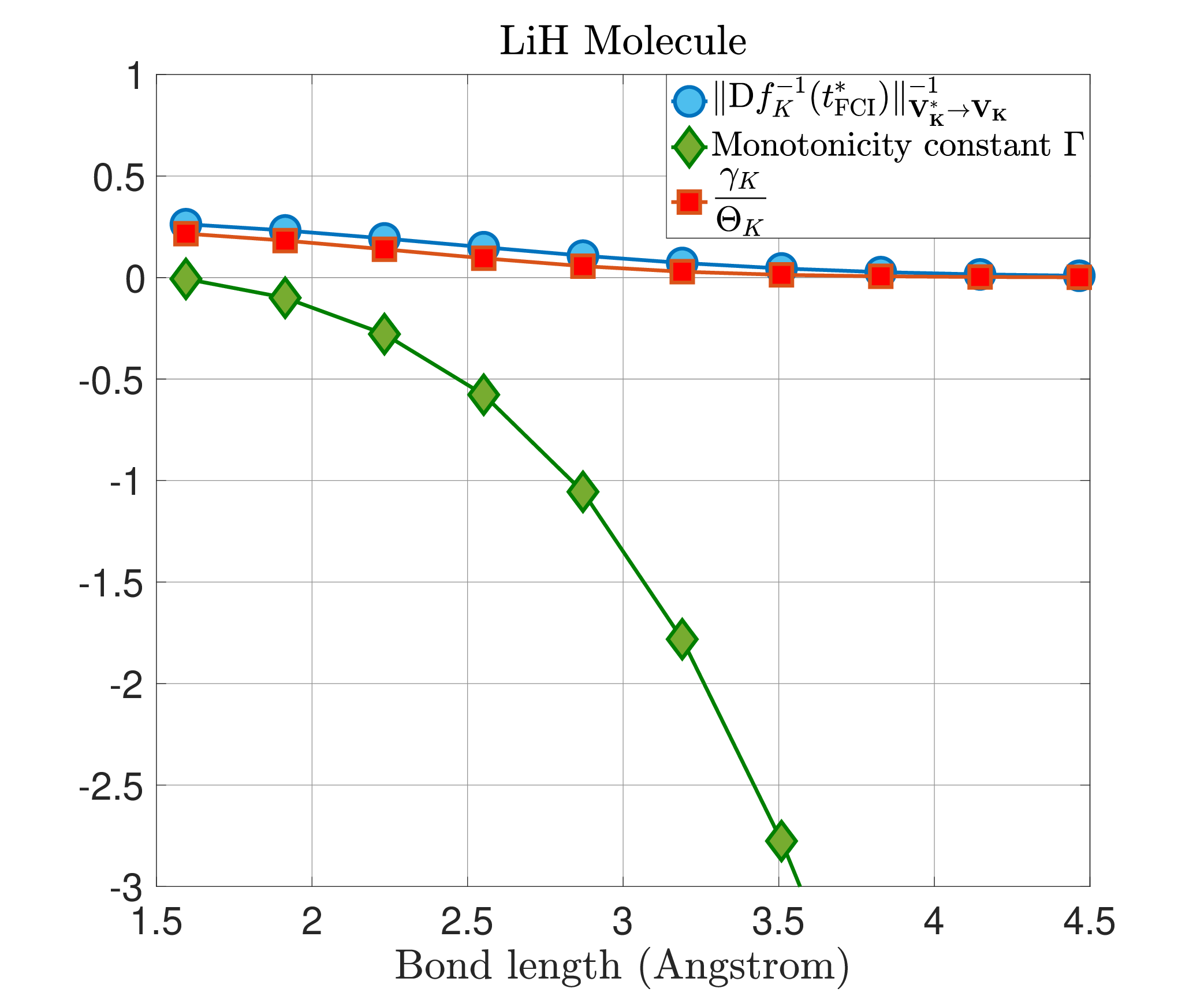} 
	\end{subfigure}\hfill
	\begin{subfigure}{0.495\textwidth}
		\centering
		\includegraphics[width=\textwidth, trim={0cm, 0cm, 0cm, 0cm},clip=true]{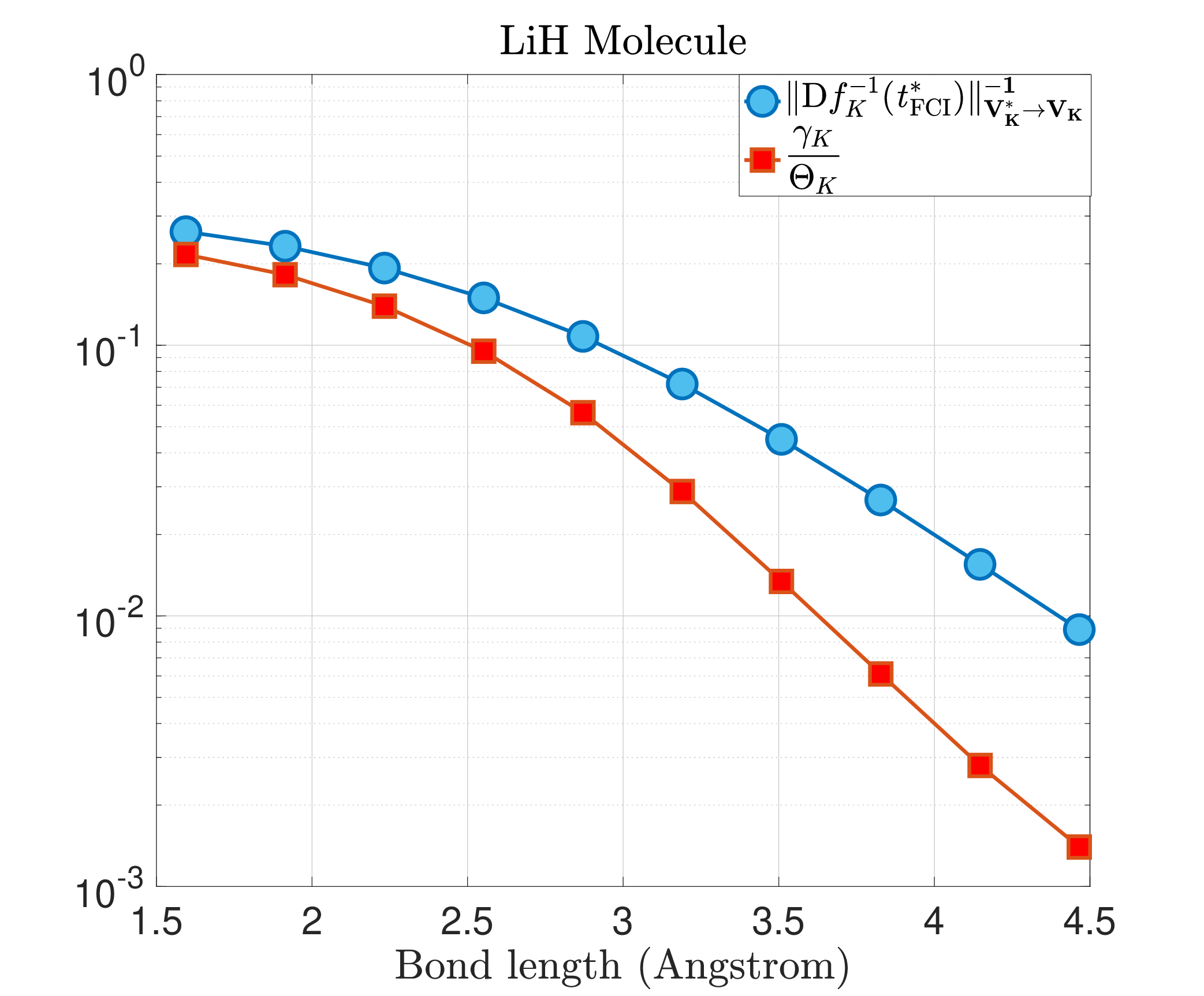} 
	\end{subfigure}
	\caption{Numerically computed constants for the LiH molecule at different bond lengths. The equilibrium bond length is 1.5949 Angstrom. The figure on the right uses a log scale on the y-axis.}
	\label{fig:02}
\end{figure}

We now turn to the second goal of this section, namely to a study of the error between solutions $\bt_K^* \in \mathbb{V}_K$ of the Full-CC equations \eqref{eq:Galerkin_truncated_CC_1} and solutions $\bt^* \in \mathbb{V}$ of the continuous coupled cluster equations \eqref{eq:Galerkin_Full_CC}. Since the Full-CC equations are simply Galerkin discretisations of the continuous CC equations, their local well-posedness can be deduced from classical results in non-linear numerical analysis. Indeed, we merely have to obtain an appropriate invertibility result for the coupled cluster Fr\'echet derivative restricted to the coefficient subspaces $\{\mathbb{V}_K\}_{K \geq N}$ and we must establish that the subspaces $\{\mathbb{V}_K\}_{K \geq N}$ have the approximation property with respect to $\mathbb{V}$. The latter demonstration is a simple consequence of the density of $\underset{K\geq N}{\cup}\mathbb{V}_K$ in $\mathbb{V}$ which has already been proven in Lemma~\ref{lem:coeff_space_finite}. We therefore focus on obtaining the required invertibility result.

We begin by defining projection operators corresponding to the various finite-dimensional approximation spaces we have introduced.

\vspace{5mm}
\begin{definition}[Projection Operators]\label{def:proj}~
	
	Let $\mathcal{V}_K= \text{\rm span}\{\Psi_0\} \cup \widetilde{\mathcal{V}}_K \subset \widehat{\mathcal{H}}^1$ denote the finite-dimensional $N$-particle approximation space for $K \geq N$ and let $\mathbb{V}_K \subset \mathbb{V}$ denote the Hilbert subspace of coefficients as defined through Definition \ref{def:coeff_space_finite}. Then
	\begin{itemize}
		\item We denote by $\mathbb{P}_{K} \colon \widehat{\mathcal{H}}^1 \rightarrow \widehat{\mathcal{H}}^1$, the $\widehat{\mathcal{H}}^1$-orthogonal projection operator onto $\mathcal{V}_K$ and by $\mathbb{P}_K^{\perp}$ its complement, i.e., $\mathbb{P}_K^{\perp}= \mathbb{I}-\mathbb{P}_K$. 
		
		\item We denote by $\Pi_{K} \colon \mathbb{V} \rightarrow \mathbb{V}$, the $( \cdot, \cdot )_{\mathbb{V}}$-orthogonal projection operator onto $\mathbb{V}_K$ and by $\Pi_K^{\perp}$ its complement, i.e., $\Pi_K^{\perp}= \mathbb{I}-\Pi_K$.
		
	\end{itemize}
\end{definition}

\begin{Notation}[Cluster Operators Involving Projections]\label{not:proj}~
	
	Consider the setting of Definition \ref{def:proj} and let $\bt \in \mathbb{V}$. In the sequel, we will frequently consider cluster operators generated by $\Pi_K \bt$ or $\Pi_K^{\perp}\bt$. We will therefore use the notation $\mathcal{T}(\Pi_K)$ and $ \mathcal{T}(\Pi_K^{\perp})$ respectively to denote these cluster operators, i.e., we denote
	\begin{align*}
		\mathcal{T}(\Pi_K):=& \sum_{\mu \in \mathcal{I}^{K}} \bs_{\mu} \mathcal{X}_{\mu} \quad \text{where }~ \{\bs_{\mu}\}_{\mu\in \mathcal{I}^K} = \Pi_K \bt \in \mathbb{V}_K, \quad \text{and}\\
		\mathcal{T}(\Pi_K^{\perp}):=& \sum_{\mu \in \mathcal{I}} \br_{\mu} \mathcal{X}_{\mu} \quad \hspace{2mm}\text{where }~ \{\br_{\mu}\}_{\mu \in \mathcal{I}} = \Pi_K^{\perp} \bt \in \mathbb{V}.
	\end{align*}
	
\end{Notation}

We are now ready to state the main technical lemma of this section. We emphasise that the proof of this lemma assumes that any isolated, simple eigenpair of the electronic Hamiltonian can be approximated by a sequence of simple eigenpairs of the Full-CI Hamiltonian.

\begin{lemma}[Invertibility of the coupled cluster Fr\'echet derivative on $\mathbb{V}_K$]\label{lem:Galerkin_inv_1}~
	
	Let the coupled cluster function $\mathcal{f} \colon \mathbb{V} \rightarrow \mathbb{V}^* $ be defined through Definition \ref{def:CC_function}, for any $\bt \in \mathbb{V}$ let $\mD \mathcal{f}(\bt)$ denote the Fr\'echet derivative of the coupled cluster function as defined through Equation \eqref{eq:CC_Jac_1}, let $\bt^* \in \mathbb{V}$ denote a zero of the coupled cluster function corresponding to an intermediately normalised eigenfunction $\Psi^* \in \widehat{\mathcal{H}}^1$ of the electronic Hamiltonian $H \colon\widehat{\mathcal{H}}^1 \rightarrow \widehat{\mathcal{H}}^{-1}$ with isolated, non-degenerate eigenvalue $\mathcal{E}^*$, let $\mathbb{V}_K \subset \mathbb{V}$ denote the Hilbert subspace of coefficients as defined through Definition~\ref{def:coeff_space_finite} for $K \geq N$, let the Full-CI Hamiltonian ${H}_{K} \colon \mathcal{V}_K \rightarrow \mathcal{V}_K^*$ be defined according to Definition \ref{eq:FCI_Hamiltonian} and assume that there exists a sequence of simple eigenpairs $(\Psi_K^*, \mathcal{E}_K^*) \in \mathcal{V}_K \times \mathbb{R}$ of the Full-CI Hamiltonians $\{{H}_K\}_{K \geq N}$, i.e.,
	\begin{align} \label{eq:new_hassan_0}
		\forall \Phi_K \in \mathcal{V}_K \colon \qquad \langle \Phi_K, {H}_K \Psi_K^*\rangle_{\mathcal{V}_K \times \mathcal{V}_K^*}&= \mathcal{E}_K^*  \langle \Phi_K, \Psi_K^*\rangle_{\widehat{\mathcal{H}}^1 \times \widehat{\mathcal{H}}^{-1}} ~\text{ with }~ \mathcal{E}_K^* \text{ simple} \quad \text{and such that}\\[2em]
		\lim_{K \to \infty} \Vert \Psi^* - \Psi_K^*\Vert_{\widehat{\mathcal{H}}^1} =0, \quad &\lim_{K \to \infty}\vert \mathcal{E}^* - \mathcal{E}_K^* \vert =0. \nonumber
	\end{align}
	Then for all $K$ sufficiently large, there exist a constant $\gamma_K >0$ uniformly bounded below in $K$, a constant $\Theta_K >0$ uniformly bounded above in $K$, a constant $\varepsilon_K> 0$ such that $\underset{K\to \infty}{\lim} \varepsilon_K =0$, a constant $\omega_K> 0$ such that $\underset{K\to \infty}{\lim} \omega_K =1$, and we have the estimate
	\begin{align*}
		\inf_{0\neq \bw_K \in \mathbb{V}_K} \sup_{0 \neq \bs_K \in \mathbb{V}_K}\frac{ \left \langle \bw_K, \mD \mathcal{f}(\bt^*)\bs_K\right \rangle_{\mathbb{V} \times \mathbb{V}^*}}{\Vert \bw_K\Vert_{\mathbb{V}} \Vert \bs_K\Vert_{\mathbb{V}}} \geq \frac{\nicefrac{\gamma_{K}}{\omega_K} - \varepsilon_K}{\Theta_K}.
	\end{align*}    
\end{lemma}
\begin{proof}

	Let $\bw_K= \{\bw_{\mu}\}_{\mu \in \mathcal{I}^K}, \bs_K=\{\bs_{\mu}\}_{\mu \in \mathcal{I}^K} \in \mathbb{V}_K$ be arbitrary and let the bounded linear operator $\mathcal{A}(\bt^*) \colon \widetilde{\mathcal{V}}\rightarrow \widetilde{\mathcal{V}}^*$ be defined according to Definition \ref{def:A_op}. It follows from Corollary \ref{cor:CC_der} that
	\begin{align*}
		\left\langle \bw_K, \mD \mathcal{f}(\bt^*)\bs_K \right\rangle_{\mathbb{V} \times \mathbb{V}^*} &= \left\langle \mathcal{W}_K \Psi_0, \mathcal{A}(\bt^*) \mathcal{S}_K\Psi_0 \right\rangle_{\widehat{\mathcal{H}}^1 \times \widehat{\mathcal{H}}^{-1}}= \left\langle \mathcal{W}_K \Psi_0, e^{-\mathcal{T}^*} \left({H}- \mathcal{E}^*\right)e^{\mathcal{T}^*} \mathcal{S}_K \Psi_0\right\rangle_{\widehat{\mathcal{H}}^1 \times \widehat{\mathcal{H}}^{-1}},
	\end{align*}
	where $\mathcal{W}_K:= \sum_{\mu \in \mathcal{I}^K} \bw_{\mu} \mathcal{X}_{\mu}$ and $\mathcal{S}_K:= \sum_{\mu \in \mathcal{I}^K} \bs_{\mu} \mathcal{X}_{\mu}$. To avoid tedious notation, let us define $\Phi_{\mathcal{W}}:= \mathcal{W}_K \Psi_0 \in \widetilde{\mathcal{V}}_K$ and $\Phi_{\mathcal{S}}:= \mathcal{S}_K \Psi_0\in \widetilde{\mathcal{V}}_K$. Obviously, we now have
	\begin{align}\label{eq:inv_new_0}
		\left \langle \Phi_{\mathcal{W}}, e^{-\mathcal{T}^*} \left({H}- \mathcal{E}^*\right)e^{\mathcal{T}^*} \Phi_{\mathcal{S}} \right \rangle_{\widehat{\mathcal{H}}^1 \times \widehat{\mathcal{H}}^{-1}}= \left \langle e^{\mathcal{T}^*}\Phi_{\mathcal{S}},  \left({H}- \mathcal{E}^*\right)e^{-(\mathcal{T}^*)^{\dagger}} \Phi_{\mathcal{W}} \right \rangle_{\widehat{\mathcal{H}}^{1} \times \widehat{\mathcal{H}}^{-1}}.
	\end{align}

	Since $\Psi^*$ is intermediately normalisable by assumption, there exists $\widetilde{K_0} \in \mathbb{N}$ such that for all $K \geq \widetilde{K_0}$, the eigenfunction $\Psi_K^* \in \mathcal{V}_K$ is intermediately normalisable. In the remainder of this proof, we assume that indeed $K \geq \widetilde{K_0}$ and we denote by $\bt_K^*:= \{\bt^*_{\mu}\}_{\mu \in \mathcal{I}} \in \mathbb{V}_K$ the coefficient vector with the property that $\Psi_K^*= e^{\mathcal{T}_K^*}\Psi_0$ where $\mathcal{T}_K^* := \sum_{\mu \in \mathcal{I}^{K}} \bt^*_{\mu}\mathcal{X}_{\mu}$. Let us emphasise here that since $\Psi_K^*$ is an eigenfunction of the Full-CI Hamiltonian, it follows from Theorem \ref{thm:Yvon_2} that $\bt_K^*$ is a zero of the restricted coupled cluster function $\mathcal{f}_K \colon \mathbb{V}_K \rightarrow \mathbb{V}_K^*$ defined through Definition \ref{def:CC_function_FCI}.
	
	Recalling now that $\Phi_{\mathcal{S}}$ is arbitrary (due to the fact that the sequence $\bs_K=\{\bs_{\mu}\}_{\mu \in \mathcal{I}^K} \in \mathbb{V}_K$ was chosen arbitrarily), we may in particular set for any $\Phi^*_{K, \perp}\in \{ \Psi_K^*\}^{\perp} \subset\mathcal{V}_K$:
	\begin{align*}
		\Phi_{\mathcal{S}}:= \mathbb{P}_0^{\perp} e^{-\mathcal{T}^*({\Pi_K})} \Phi^*_{K, \perp}\in \widetilde{\mathcal{V}}_K,
	\end{align*}
	where $\mathcal{T}^*({\Pi_K})$ denotes the cluster operator generated by $\Pi_K \bt^*\in \mathbb{V}_K$ (recall Notation \ref{not:proj}) and we have used the fact that, thanks to the properties of the excitation operators $\{\mathcal{X}_{\mu}\}_{\mu \in \mathcal{I}}$ given by Remark \ref{rem:properties}, it holds that $e^{-\mathcal{T}^*(\Pi_K)} \Psi_K \in \mathcal{V}_K$ for any $\Psi_K \in \mathcal{V}_K$.

	Plugging in this choice of $\Phi_{\mathcal{S}}$ in Equation \eqref{eq:inv_new_0} now yields
	\begin{align}\nonumber
		\left \langle \mathcal{W}_K \Psi_0, \mathcal{A}(\bt^*) \mathcal{S}_K\Psi_0\right\rangle_{\widehat{\mathcal{H}}^1 \times \widehat{\mathcal{H}}^{-1}}&= \underbrace{\left \langle e^{\mathcal{T}^*}e^{-\mathcal{T}^*(\Pi_K)}  \Phi^*_{K, \perp},  \left({H}- \mathcal{E}^*\right)e^{-(\mathcal{T}^*)^{\dagger}}\Phi_{\mathcal{W}} \right \rangle_{\widehat{\mathcal{H}}^{1} \times \widehat{\mathcal{H}}^{-1}}}_{:=\rm (I)}\\ \label{eq:inv_new_1}
		&- \underbrace{\left \langle e^{\mathcal{T}^*} \mathbb{P}_0 e^{-\mathcal{T}^*(\Pi_K)}  \Phi^*_{K, \perp},  \left({H}- \mathcal{E}^*\right)e^{-(\mathcal{T}^*)^{\dagger}}\Phi_{\mathcal{W}} \right \rangle_{\widehat{\mathcal{H}}^{1} \times \widehat{\mathcal{H}}^{-1}}}_{:= \rm (II)}.
	\end{align}
	
	We claim that the term (II) is identically zero. Indeed, using the fact that $e^{\mathcal{T}^*} \Psi_0=\Psi^*$ by assumption, a straightforward calculation shows that
	\begin{align*}
		{\rm (II)}&=\left (  \Psi_0,  e^{-\mathcal{T}^*(\Pi_K)}  \Phi^*_{K, \perp}\right )_{\widehat{\mathcal{L}}^2} \left \langle e^{\mathcal{T}^*} \Psi_0,  \left({H}- \mathcal{E}^*\right)e^{-(\mathcal{T}^*)^{\dagger}}\Phi_{\mathcal{W}} \right \rangle_{\widehat{\mathcal{H}}^{1} \times \widehat{\mathcal{H}}^{-1}},
	\end{align*}
	and the second term in the product above is zero as ${H}e^{\mathcal{T}^*} \Psi_0= \mathcal{H}\Psi^*= \mathcal{E}^*\Psi^*$.
	
	It therefore remains to simplify the term (I). To this end, we observe that we can write
	\begin{align*}
		{\rm (I)}&=\left \langle e^{\mathcal{T}^*(\Pi_K^{\perp})}\Phi^*_{K, \perp},  \left({H}- \mathcal{E}^*\right)e^{-(\mathcal{T}^*)^{\dagger}}\Phi_{\mathcal{W}} \right \rangle_{\widehat{\mathcal{H}}^{1} \times \widehat{\mathcal{H}}^{-1}}\\[1em]
		&=\underbrace{\left \langle \Phi^*_{K, \perp},  \left({H}- \mathcal{E}^*\right)e^{-(\mathcal{T}^*)^{\dagger}}\Phi_{\mathcal{W}} \right \rangle_{\widehat{\mathcal{H}}^{1} \times \widehat{\mathcal{H}}^{-1}}}_{:= (\rm IA)}\\[1em]
		&+\underbrace{\left \langle \left(e^{\mathcal{T}^*(\Pi_K^{\perp})}-\mathbb{I}\right)\Phi^*_{K, \perp},  \left({H}- \mathcal{E}^*\right)e^{-(\mathcal{T}^*)^{\dagger}}\Phi_{\mathcal{W}} \right \rangle_{\widehat{\mathcal{H}}^{1} \times \widehat{\mathcal{H}}^{-1}}}_{:= (\rm IB)}.
	\end{align*}
	
	Focusing first on the term (IB) and using the Cauchy-Schwarz inequality, we may write
	\begin{align}\label{eq:inv_new_3}
		{\rm (IB)}&\geq -\left \Vert e^{\mathcal{T}^*(\Pi_K^{\perp})}-\mathbb{I}\right \Vert_{\widehat{\mathcal{H}}^{1} \to {\widehat{\mathcal{H}}^{1}}} \left \Vert {H}-\mathcal{E}^*\right\Vert_{\widehat{\mathcal{H}}^{1} \to {\widehat{\mathcal{H}}^{-1}}} \big \Vert \Phi^*_{K, \perp}\big \Vert_{\widehat{\mathcal{H}}^{1}}\big \Vert e^{-(\mathcal{T}^*)^{\dagger}}\Phi_{\mathcal{W}}\big \Vert_{\widehat{\mathcal{H}}^{1}}.
	\end{align}
	
	We now claim that in fact
	\begin{align*}
		\lim_{K \to \infty} \left \Vert e^{\mathcal{T}^*(\Pi^{\perp}_K)}-\mathbb{I} \right \Vert_{\widehat{\mathcal{H}}^{1} \to \widehat{\mathcal{H}}^{1}}=0.
	\end{align*}
	
	Indeed, thanks to the boundedness properties of the excitation operators given in Theorem~\ref{pro:excit}, it holds that
	\begin{align*}
		\Vert e^{-\mathcal{T}^*(\Pi_K)}\Vert_{\widehat{\mathcal{H}}^{1} \to \widehat{\mathcal{H}}^{1}} \leq e^{\Vert \mathcal{T}^*(\Pi_K)\Vert_{\widehat{\mathcal{H}}^{1} \to \widehat{\mathcal{H}}^{1}}} \leq e^{\beta \Vert \Pi_K\bt^*\Vert_{\mathbb{V}}} \leq e^{\beta \Vert \bt^*\Vert_{\mathbb{V}}},
	\end{align*}
	where the constant $\beta>0$ depends only on $N$.
	
	Therefore, we need only show that $\lim_{K \to \infty} \Vert e^{\mathcal{T}^*} - e^{\mathcal{T}^*(\Pi_K)} \Vert_{\widehat{\mathcal{H}}^{1} \to \widehat{\mathcal{H}}^{1}}=0$. Recall however, that the exponential function is of class $\mathscr{C}^{\infty}$ on the algebra of bounded operators on $\widehat{\mathcal{H}}^1$, and thus it suffices to show that
	\begin{align*}
		\lim_{K \to \infty} \Vert \mathcal{T}^*- \mathcal{T}^*(\Pi_K)\Vert_{\widehat{\mathcal{H}}^{1} \to \widehat{\mathcal{H}}^{1}}=0.
	\end{align*}
	But this is an obvious consequence of the density of the coefficient spaces $\{\mathbb{V}_K\}_{K\geq N}$ in $\mathbb{V}$. Indeed,
	\begin{align}\label{eq:inv_new_3b}
		\lim_{K \to \infty} \Vert \mathcal{T}^*- \mathcal{T}^*(\Pi_K)\Vert_{\widehat{\mathcal{H}}^{1} \to \widehat{\mathcal{H}}^{1}} \leq \beta \lim_{K \to \infty} \Vert \bt^* - \Pi_K \bt^*\Vert_{\mathbb{V}}=0.
	\end{align}

	Consequently, combining Equations \eqref{eq:inv_new_3} and \eqref{eq:inv_new_3b}, we obtain the existence of a constant $\varepsilon_{1, K} >0$ with the property that $\lim_{K \to \infty}\varepsilon_{1, K} =0$ and such that 
	\begin{equation}\label{eq:new_hassan}
		{\rm (IB)}\geq -\varepsilon_{1, K} \big \Vert \Phi^*_{K, \perp}\big \Vert_{\widehat{\mathcal{H}}^{1}}\big \Vert e^{-(\mathcal{T}^*)^{\dagger}}\Phi_{\mathcal{W}}\big \Vert_{\widehat{\mathcal{H}}^{1}}.
	\end{equation}

	Let us now return to the term (IA). Notice that we may write
	\begin{align*}
		{\rm (IA)}= \underbrace{\left \langle \Phi^*_{K, \perp},  \left({H}- \mathcal{E}^*\right)e^{-(\mathcal{T}^*_K)^{\dagger}}\Phi_{\mathcal{W}} \right \rangle_{\widehat{\mathcal{H}}^1 \times \widehat{\mathcal{H}}^{-1}} }_{:= (\rm IAA)} + \underbrace{\left \langle \Phi^*_{K, \perp},  \left({H}- \mathcal{E}^*\right)\left(e^{-(\mathcal{T}^*)^{\dagger}}-e^{-(\mathcal{T}^*_K)^{\dagger}}\right)\Phi_{\mathcal{W}} \right \rangle_{\widehat{\mathcal{H}}^{1} \times \widehat{\mathcal{H}}^{-1}}}_{:= (\rm IAB)},
	\end{align*}
	where we recall from Equation \eqref{eq:new_hassan_0} that $\bt^*_K = \{\bt^*_{\mu}\}_{\mu \in \mathcal{I}^K} \in \mathbb{V}_K$ is the coefficient vector such that $e^{\mathcal{T}_K^*}\Psi_0=\Psi_K^* \in \mathcal{V}_K$. 
	
	We first simplify the term (IAB). Thanks to the Cauchy-Schwarz inequality we may write
	\begin{align}
		{\rm (IAB)}\geq - \left \Vert e^{-(\mathcal{T}^*)^{\dagger}}-e^{-(\mathcal{T}^*_K)^{\dagger}}\right \Vert_{\widehat{\mathcal{H}}^{1} \to {\widehat{\mathcal{H}}^{1}}} \left \Vert e^{(\mathcal{T}^*)^{\dagger}}\right \Vert_{\widehat{\mathcal{H}}^{1} \to {\widehat{\mathcal{H}}^{1}}} \left \Vert {H}-\mathcal{E}^*\right\Vert_{\widehat{\mathcal{H}}^{1} \to {\widehat{\mathcal{H}}^{-1}}} \big \Vert \Phi^*_{K, \perp}\big \Vert_{\widehat{\mathcal{H}}^{1}}\big \Vert e^{-(\mathcal{T}^*)^{\dagger}}\Phi_{\mathcal{W}}\big \Vert_{\widehat{\mathcal{H}}^{1}}
	\end{align}
	
	We now claim that in fact $\lim_{K \to \infty} \left \Vert e^{-(\mathcal{T}^*)^{\dagger}}-e^{-(\mathcal{T}_K^*)^{\dagger}}\right \Vert_{\widehat{\mathcal{H}}^{1} \to {\widehat{\mathcal{H}}^{1}}} =0$. Indeed, an easy calculation using the continuity properties of cluster operators given by Theorem \ref{pro:excit}, shows the existence of a constant $\widetilde{\beta^{\dagger}}>0$, depending only on~$N$, such that for any $K \geq N$ it holds that 
	\begin{align}\label{eq:Hassan_latest}
		\left \Vert e^{-(\mathcal{T}^*)^{\dagger}}-e^{-(\mathcal{T}_K^*)^{\dagger}}\right \Vert_{\widehat{\mathcal{H}}^{1} \to {\widehat{\mathcal{H}}^{1}}} \leq \widetilde{\beta^{\dagger}} \left \Vert e^{\mathcal{T}^*}\Psi_0 -e^{\mathcal{T}_K^*}\Psi_0 \right \Vert_{\widehat{\mathcal{H}}^{1}}=  \left \Vert \Psi^*-\Psi_K^* \right \Vert_{\widehat{\mathcal{H}}^{1}}.
	\end{align}
	
	The claim now follows by using the convergence of the approximate eigenvector $\Psi_K^* \in \mathcal{V}_K$ to $\Psi^* \in \widehat{\mathcal{H}}^1$ from Equation~\eqref{eq:new_hassan_0}. Consequently, we obtain the existence of a constant $\varepsilon_{2, K} >0$ with the property that $\lim_{K \to \infty}\varepsilon_{2, K} =0$ and such that 
	\begin{equation}\label{eq:new_hassan_3}
		{\rm (IB)}\geq -\varepsilon_{2, K} \big \Vert \Phi^*_{K, \perp}\big \Vert_{\widehat{\mathcal{H}}^{1}}\big \Vert e^{-(\mathcal{T}^*)^{\dagger}}\Phi_{\mathcal{W}}\big \Vert_{\widehat{\mathcal{H}}^{1}}.
	\end{equation}

	Focusing finally on the term (IAA), a simple calculation shows that for any $\Phi_{\mathcal{W}} \in \widetilde{\mathcal{V}}_K$, we have that $e^{-(\mathcal{T}^*_K)^{\dagger}}\Phi_{\mathcal{W}} \in \{\Psi_K^*\}^{\perp} \subset \mathcal{V}_K$. Furthermore, $\mathcal{E}^*$ is a simple, isolated eigenvalue by assumption and $\lim_{K \to \infty} \mathcal{E}_K^* = \mathcal{E}^*$. Since $ \Phi^*_{K, \perp} \in \{\Psi_K^*\}^{\perp} \subset \mathcal{V}_K$ is arbitrary, we therefore deduce the existence of $\widehat{K_0} \in \mathbb{N}$ sufficiently large such that for all $K \geq \widehat{K_0}$ the shifted Full-CI Hamiltonian ${H}_K- \mathcal{E}^*$ satisfies an inf-sup condition on $\{\Psi_K^*\}^{\perp}  \subset \mathcal{V}_K$, and as a consequence, 
	\begin{equation}\label{eq:new_hassan_2}
		\sup_{0 \neq \Phi^*_{K, \perp} \in \{\Psi_K^*\}^{\perp}} \; \frac{\rm (IAA)}{\Vert \Phi^*_{K, \perp}\Vert_{\widehat{\mathcal{H}}^1}}= \sup_{0 \neq \Phi^*_{K, \perp} \in \{\Psi_K^*\}^{\perp}}  \frac{\big \langle \Phi^*_{K, \perp},  \left({H}- \mathcal{E}^*\right)e^{-(\mathcal{T}^*_K)^{\dagger}}\Phi_{\mathcal{W}} \big \rangle_{\widehat{\mathcal{H}}^{1} \times \widehat{\mathcal{H}}^{-1}}}{\Vert \Phi^*_{K, \perp}\Vert_{\widehat{\mathcal{H}}^1}} \geq \gamma_K \big\Vert e^{-(\mathcal{T}^*_K)^{\dagger}}\Phi_{\mathcal{W}}\big\Vert_{\widehat{\mathcal{H}}^{1}},
	\end{equation}
	where $\gamma_K$ denotes the inf-sup constant of the shifted Full-CI Hamiltonian ${H}_K- \mathcal{E}^*$ on $\{\Psi_K^*\}^{\perp} \subset \mathcal{V}_K$ for $K \geq \widehat{K_0}$. For the remainder of this proof, we assume that indeed $K \geq\widehat{K_0}$.
	
	Notice that this last bound can be written as
	\begin{align*}
		\gamma_K \big\Vert e^{-(\mathcal{T}^*_K)^{\dagger}}\Phi_{\mathcal{W}}\big\Vert_{\widehat{\mathcal{H}}^{1}}= \gamma_K \big\Vert e^{-(\mathcal{T}_K^*-\mathcal{T}^* )^{\dagger}} e^{-(\mathcal{T}^*)^{\dagger}}\Phi_{\mathcal{W}}\big\Vert_{\widehat{\mathcal{H}}^{1}}\geq \frac{\gamma_K}{\Vert  e^{(\mathcal{T}_K^*-\mathcal{T}^* )^{\dagger}}\Vert_{\widehat{\mathcal{H}}^{1} \to \widehat{\mathcal{H}}^{1}}}\big\Vert e^{-(\mathcal{T}^*)^{\dagger}}\Phi_{\mathcal{W}}\big\Vert_{\widehat{\mathcal{H}}^{1}},
	\end{align*}
	where the inequality follows from the invertibility of the exponential map. Using now a similar calculation to the one used to obtain Inequality \eqref{eq:Hassan_latest}, it can easily be shown that $\lim_{K \to \infty} \Vert  e^{(\mathcal{T}_K^*-\mathcal{T}^* )^{\dagger}}\Vert_{\widehat{\mathcal{H}}^{1} \to \widehat{\mathcal{H}}^{1}}=1$. Consequently, we obtain the existence of constant $\omega_{K}> 0$ with the property that $\lim_{K \to \infty} \omega_K =1$ and such that 
	\begin{equation}\label{eq:Hassan_latest_2}
		\sup_{0 \neq \Phi^*_{K, \perp} \in \{\Psi_K^*\}^{\perp}} \; \frac{\rm (IAA)}{\Vert \Phi^*_{K, \perp}\Vert_{\widehat{\mathcal{H}}^1}} =\sup_{0 \neq \Phi^*_{K, \perp} \in \{\Psi_K^*\}^{\perp}}  \frac{\left \langle \Phi^*_{K, \perp},  \left({H}- \mathcal{E}^*\right)e^{-(\mathcal{T}^*_K)^{\dagger}}\Phi_{\mathcal{W}} \right \rangle_{\widehat{\mathcal{H}}^{1} \times \widehat{\mathcal{H}}^{-1}}}{\Vert \Phi^*_{K, \perp}\Vert_{\widehat{\mathcal{H}}^1}}\geq \frac{\gamma_k}{\omega_k} \big\Vert e^{-(\mathcal{T}^*)^{\dagger}}\Phi_{\mathcal{W}}\big\Vert_{\widehat{\mathcal{H}}^{1}}.
	\end{equation}

	Combining now the estimates \eqref{eq:inv_new_0}-\eqref{eq:Hassan_latest_2} allows us to conclude that 
	\begin{align*}
		\sup_{0 \neq \bs_K \in \mathbb{V}_K}\frac{ \left \langle \bw_K, \mD \mathcal{f}(\bt^*)\bs_K\right \rangle_{\mathbb{V} \times \mathbb{V}^*}}{ \Vert \bs_K\Vert_{\mathbb{V}}} =& \sup_{0 \neq \Phi_{\mathcal{S}} \in \widetilde{\mathcal{V}}_K} \frac{\left \langle \mathcal{W}_K \Psi_0, \mathcal{A}(\bt^*) \mathcal{S}_K\Psi_0\right\rangle_{\widehat{\mathcal{H}}^1 \times \widehat{\mathcal{H}}^{-1}}}{ \Vert\Phi_{\mathcal{S}}\Vert_{\widehat{\mathcal{H}}^1}}\\
		\geq& \sup_{0 \neq \Phi^*_{K, \perp} \in \{\Psi_K^*\}^{\perp} }\frac{(\rm IAA) + (\rm IAB) + (\rm IB)}{ \Vert\mathbb{P}_0^{\perp} e^{-\mathcal{T}^*({\Pi_K})} \Phi^*_{K, \perp}\Vert_{\widehat{\mathcal{H}}^1}}\\[1em]
		\geq& \frac{\nicefrac{\gamma_K}{\omega_K}-\varepsilon_{1, K} -\varepsilon_{2, K}}{\left\Vert \mathbb{P}_0^{\perp} e^{-\mathcal{T}^*({\Pi_K})}\right\Vert_{\widehat{\mathcal{H}}^1} }\Vert e^{-(\mathcal{T}^*)^{\dagger}}\Phi_{\mathcal{W}}\big \Vert_{\widehat{\mathcal{H}}^{1}}\\[1em]
		\geq&\frac{ \nicefrac{\gamma_K}{\omega_K}-\varepsilon_{1, K} -\varepsilon_{2, K} }{ \Vert e^{(\mathcal{T}^*)^{\dagger}}\Vert_{\widehat{\mathcal{H}}^1 \to \widehat{\mathcal{H}}^1} \Vert \mathbb{P}_0^{\perp} e^{-\mathcal{T}^*({\Pi_K})}\Vert_{\widehat{\mathcal{H}}^1 \to \widehat{\mathcal{H}}^1}} \big\Vert\Phi_{\mathcal{W}} \big\Vert_{\widehat{\mathcal{H}}^{1}}.
	\end{align*}
	Defining the constants $\Theta_K:=  \Vert e^{(\mathcal{T}^*)^{\dagger}}\Vert_{\widehat{\mathcal{H}}^1 \to \widehat{\mathcal{H}}^1} \Vert \mathbb{P}_0^{\perp} e^{-\mathcal{T}^*({\Pi_K})}\Vert_{\widehat{\mathcal{H}}^1 \to \widehat{\mathcal{H}}^1}$,  and $\varepsilon_K:= \varepsilon_{1, K} +\varepsilon_{2, K}$ and taking the infimum over all coefficient vectors $\bw_K\in \mathbb{V}_K$ now yields the required estimate. The fact that the constant $\Theta_K$ is uniformly bounded above in $K$ is a consequence of the continuity properties of exponential cluster operators together with the density of the union of subspaces $\underset{K\geq N}{\cup} \mathbb{V}_K$ in $\mathbb{V}$. The fact that the inf-sup constant $\gamma_K$ is uniformly bounded below in $K$ is a consequence of the eigenvalue convergence $\mathcal{E}_K^* \to \mathcal{E}^*$ (see also the arguments in Remark \ref{rem:interp_const}). 
\end{proof}

Equipped with Lemma \ref{lem:Galerkin_inv_1}, we are now ready to state the final result of this section, which concerns the error between the ground state solution of the Full-CC equations in a finite basis \eqref{eq:Galerkin_truncated_CC_1} and the exact solutions of the continuous CC equations \eqref{eq:Galerkin_Full_CC}.

\begin{theorem}[Error Estimates for Full-CC in a Finite Basis]\label{thm:Galerkin_full_CC}~
	
	Let the coupled cluster function $\mathcal{f} \colon \mathbb{V} \rightarrow \mathbb{V}^* $ be defined through Definition \ref{def:CC_function}, for any $\bt \in \mathbb{V}$ let $\mD \mathcal{f}(\bt)$ denote the Fr\'echet derivative of the coupled cluster function as defined through Equation \eqref{eq:CC_Jac_1}, let $\bt^* \in \mathbb{V}$ denote a zero of the coupled cluster function corresponding to an intermediately normalised eigenfunction $\Psi^* \in \widehat{\mathcal{H}}^1$ of the electronic Hamiltonian ${H} \colon\widehat{\mathcal{H}}^1 \rightarrow \widehat{\mathcal{H}}^{-1}$ with isolated, non-degenerate ground state eigenvalue $\mathcal{E}^*$, let $\mathbb{V}_K \subset \mathbb{V}$ denote the Hilbert subspace of coefficients as defined through Definition \ref{def:coeff_space_finite} for $K\geq N$, let the Full-CI Hamiltonian ${H}_{K} \colon \mathcal{V}_K \rightarrow \mathcal{V}_K^*$ be defined according to Definition \ref{eq:FCI_Hamiltonian}, assume that there exists a sequence of simple eigenpairs $(\Psi_K^*, \mathcal{E}_K^*) \in \mathcal{V}_K \times \mathbb{R}$ of the Full-CI Hamiltonians $\{{H}_K\}_{K \geq N}$, i.e.,
	\begin{align*} 
		\forall \Phi_K \in \mathcal{V}_K \colon \qquad \langle \Phi_K, {H}_K \Psi_K^*\rangle_{\mathcal{V}_K \times \mathcal{V}_K^*}&= \mathcal{E}_K^*  \langle \Phi_K, \Psi_K^*\rangle_{\widehat{\mathcal{H}}^1 \times \widehat{\mathcal{H}}^{-1}} ~\text{ with }~ \mathcal{E}_K^* \text{ simple} \quad \text{and such that}\\[2em]
		\lim_{K \to \infty} \Vert \Psi^* - \Psi_K^*\Vert_{\widehat{\mathcal{H}}^1} =0, \quad &\lim_{K \to \infty}\vert \mathcal{E}^* - \mathcal{E}_K^* \vert =0, \nonumber
	\end{align*}
	and let the constants $\gamma_K, \Theta_K, \varepsilon_K, \omega_K >0$ be defined as in the proof of Lemma \ref{lem:Galerkin_inv_1}. 
	
	Then there exists $K_0 \in \mathbb{N}$ and a constant $\delta_0>0$ such that for all $K \geq K_0$ there exists a unique solution $\bt_{K}^*\in \mathbb{V}_K$ to the Full-CC equations in a finite basis \eqref{eq:Galerkin_truncated_CC_1} in the closed ball $\overline{\mathbb{B}_{\delta_K}(\bt^*) }$ where $\delta_K= \delta_0 \dfrac{\nicefrac{\gamma_K}{\omega_K}-\varepsilon_K}{\Theta_K}$. 
	
	Moreover, there exists a constant ${\rm C}>0$ such that $\forall K\geq K_0$ we have the quasi-optimality result
	\begin{equation}\label{eq:FullCC_optimality}
		\Vert \bt^*_{K} -\bt^*\Vert_{\mathbb{V}} \leq {\rm C}\frac{\Theta_K}{\nicefrac{\gamma_K}{\omega_K}-\varepsilon_K}\; \inf_{\bs_K \in \mathbb{V}_K}\Vert \bs_K - \bt^*\Vert_{\mathbb{V}},
	\end{equation}
	and we have the residual-based error estimate
	\begin{equation}\label{eq:FullCC_error}
		\Vert \bt^*_{K} -\bt^*\Vert_{\mathbb{V}} \leq 2\left \Vert \mD \mathcal{f}\left(\bt^*_{K}\right)^{-1}\right\Vert_{\mathbb{V}^* \to \mathbb{V}} \;\left\Vert \mathcal{f}\left(\bt^*_{K}\right)\right\Vert_{\mathbb{V}^*}.
	\end{equation}
\end{theorem}
\begin{proof}
	As mentioned at the beginning of this section, the Full-Coupled Cluster equations in a finite basis \eqref{eq:Galerkin_truncated_CC_1} are simply a Galerkin discretisation of the continuous coupled cluster equations \eqref{eq:Galerkin_Full_CC}. Galerkin discretisations of non-linear equations have been widely studied in the literature on non-linear numerical analysis. In particular, the proof of Theorem \ref{thm:Galerkin_full_CC} is a direct application of~\cite[Theorem 7.1]{MR1470227}. We merely have to confirm that the assumptions of~\cite[Theorem 7.1]{MR1470227} hold, and this amounts to 
	
	\begin{enumerate}
		\item Establishing that the coupled cluster Fr\'echet derivative at $\bt^* \in \mathbb{V}$, which we denote $\mD \mathcal{f}(\bt^*)$, satisfies the discrete inf-sup condition
		\begin{align}\label{eq:inv_new_4}
			\exists \Upsilon_K >0\colon \qquad    \inf_{0\neq \bw_K \in \mathbb{V}_K} \sup_{0 \neq \bs_K \in \mathbb{V}_K}\frac{ \left \langle \bw_K, \mD \mathcal{f}(\bt^*)\bs_K\right \rangle_{\mathbb{V} \times \mathbb{V}^*}}{\Vert \bw_K\Vert_{\mathbb{V}} \Vert \bw_S\Vert_{\mathbb{V}}} \geq \Upsilon_K;
		\end{align}
		
		\item Establishing that the coefficient subspaces $\{\mathbb{V}_K\}_{K\geq N}$ satisfy the following approximability condition:
		\begin{align}\label{eq:inv_new_5}
			\lim_{K \to \infty} \;\inf_{0\neq \bs_K \in \mathbb{V}_K}\; \frac{1}{\Upsilon_K^2}\;\Vert \bt^* - \bs_K\Vert_{\mathbb{V}} =0.
		\end{align}
	\end{enumerate}

	The discrete inf-sup condition \eqref{eq:inv_new_4} has been established in Lemma \ref{lem:Galerkin_inv_1} with constant $\Upsilon_K = \dfrac{\nicefrac{\gamma_K}{\omega_K} - \varepsilon_K}{\Theta_K}$ which will obviously be positive for all $K$ sufficiently large since $\varepsilon_K \to 0$. It therefore remains to establish the approximability result \eqref{eq:inv_new_5} but this is a simple consequence of the previously exploited fact that the union of subspaces $\underset{K\geq N}{\cup} \mathbb{V}_K$ is dense in $\mathbb{V}$ together with the fact that, as shown in the proof of Lemma \ref{lem:Galerkin_inv_1}, the constant $\gamma_K$ is uniformly bounded below in $K$ and the constant $\Theta_K$ is uniformly bounded above in $K$.
\end{proof}

We conclude this section with several remarks.

\begin{remark}[Necessity of Assumptions of Lemma \ref{lem:Galerkin_inv_1} in Theorem \ref{thm:Galerkin_full_CC}]~
	
	Consider the setting of Lemma \ref{lem:Galerkin_inv_1} and Theorem \ref{thm:Galerkin_full_CC} and recall in particular the assumption that any isolated, simple eigenpair of the electronic Hamiltonian can be approximated by a sequence of simple eigenpairs of the Full-CI Hamiltonian as expressed through Equation \eqref{eq:new_hassan_0}. It is readily seen from the proof of Lemma \ref{lem:Galerkin_inv_1} that this assumption is not required if one considers invertibility of the CC Fréchet derivative $\mD \mathcal{f}(\bt^*)$ on $\mathbb{V}_K$ at $\bt^* = \bt^*_{\rm FCI}$. Indeed, in this special case, the discrete inf-sup condition for the shifted Full-CI Hamiltonian ${H}_K- \mathcal{E}^*$ on $\{\Psi^*_K\}^{\perp} \subset \mathcal{V}_K$ can be replaced with the coercivity of the shifted electronic Hamiltonian ${H}- \mathcal{E}^*_{\rm GS}$ on $\{\Psi^*_{\rm GS}\}^{\perp} \subset \widehat{\mathcal{H}}^1$. In this special case therefore, the proof of Lemma \ref{lem:Galerkin_inv_1} holds without any assumption beyond the simplicity of the ground state energy and the intermediate normalisability of the ground state wave-function. Thus we can deduce the asymptotic local well-posedness of the Full-CC equations in a finite-basis \eqref{eq:Galerkin_Full_CC} in a neighbourhood of $\bt_{\rm GS}^* \in \mathbb{V}$ according to Theorem \ref{thm:Galerkin_full_CC} without any additional assumptions.
\end{remark}

\vspace{2mm}

\begin{remark}[Comparing the Conclusions of Theorems \ref{thm:CC_truncated_full} and \ref{thm:Galerkin_full_CC}]~
	
	Consider the settings of Theorems \ref{thm:CC_truncated_full} and \ref{thm:Galerkin_full_CC}. Let us emphasise here that, in contrast to Theorem \ref{thm:CC_truncated_full}, Theorem \ref{thm:Galerkin_full_CC} does not explicitly require that the ground state wave-function in $\mathcal{V}_K$ of the Full-CI Hamiltonian be intermediately normalisable or that the associated ground state eigenvalue be  simple. Instead these properties are inherited (for $K$ large enough) from the properties of the exact electronic Hamiltonian ${H} \colon \widehat{\mathcal{H}}^1 \rightarrow \widehat{\mathcal{H}}^{-1}$ thanks to the density of the $N$-particle approximation spaces $\{\mathcal{V}_K\}_{K\geq N}$ in $\widehat{\mathcal{H}}^1$. More significantly, Theorem \ref{thm:Galerkin_full_CC} provides error estimates for the Full-CC equations in a finite basis with respect to the zeros of the exact (infinite-dimensional) coupled cluster function.
\end{remark}

\vspace{3mm}

\begin{remark}[Error Estimates for the Discrete Coupled Cluster Energies]~
	
	It is natural, at this point, to ask whether a priori and residual-based error estimates of the form \eqref{eq:FullCC_optimality} and \eqref{eq:FullCC_error} can be obtained for the Full-CC discrete energies. Quasi-optimal a priori error estimates for the discrete CC energies have been obtained by Schneider and Rohwedder \cite[Theorem 4.5]{MR3110488} using the dual weighted residual-based approach developed by Rannacher and coworkers \cite{soner2003adaptive}. The arguments of Schneider and Rohwedder can readily be seen to apply in our framework, and the proof and statement of \cite[Theorem 4.5]{MR3110488} can be thus be copied nearly word-for-word, the only difference being that the local monotonicity constant that appears in the a priori error estimate in  \cite[Theorem 4.5]{MR3110488} is replaced with the discrete inf-sup constant that we have derived in Lemma \ref{lem:Galerkin_inv_1}. The establishment of residual-based error estimates for the discrete CC energies, which requires considerably more work but can be achieved using the tools developed in this article, will be addressed in a forthcoming contribution. 
\end{remark}